\documentclass[a4paper,fleqn]{cas-dc}

\usepackage[numbers]{natbib}
\usepackage{subcaption}
\usepackage{comment}
\usepackage{tikz}

\def\tsc#1{\csdef{#1}{\textsc{\lowercase{#1}}\xspace}}
\tsc{WGM}
\tsc{QE}
\newcommand{\rhohat}{\hat{\rho}}
\newcommand{\vhat}{\hat{v}}
\newcommand{\M}{M(\rhohat(1),\vhat(1),1)}
\newcommand{\w}{\vec{w}(\rhohat(1),\vhat(1),1)}
\newcommand{\loc}{\text{loc}}
\newcommand{\crit}{\text{crit}}
\DeclareMathOperator*{\argmin}{arg\,min}

\newtheorem{theorem}{Theorem}
\newtheorem{lemma}[theorem]{Lemma}
\newproof{pf}{Proof}
\newproof{pot}{Proof of Theorem \ref{thm}}

\begin{document}
\let\WriteBookmarks\relax
\def\floatpagepagefraction{1}
\def\textpagefraction{.001}

\shorttitle{Multidimensional Cutoff Implosions}    

\shortauthors{Luong et al.}  

\title [mode = title]{Finite time implosions of the multidimensional compressible Euler equations with a finite far-field cutoff}  

\tnotemark[1] 

\tnotetext[1]{This work was supported by the US Department of Energy through the Los Alamos National Laboratory. Los Alamos National Laboratory is operated by Triad National Security, LLC, for the National Nuclear Security Administration of the US Department of Energy (Contract No. 89233218CN000001).} 

%

\author[1]{Jack Luong}[orcid = 0000-0003-4674-8578]

\cormark[1]

\fnmark[1]
\fntext[fn1]{Jack Luong and Andrea L. Bertozzi are supported by Los Alamos National Lab award CW93449 C5677 and ONR grant N00014-23-1-2565.}
\ead{jqluong@g.ucla.edu}

\ead[url]{}

\credit{Conceptualization, Methodology, Software, Validation, Formal Analysis, Investigation, Writing - Original, Writing - Review \& Editing, Visualization}

\affiliation[1]{organization={Department of Mathematics, University of California, Los Angeles},
            addressline={520 Portola Plaza}, 
            city={Los Angeles},
            postcode={90095}, 
            state={California},
            country={USA}}

\author[1]{Reed Spitzer}[orcid = 0000-0001-9949-4343]

\fnmark[2,3]
\fntext[2]{Reed Spitzer and Andrea L. Bertozzi are supported by Simons Foundation Math + X Grant 510776}
\fntext[3]{Reed Spitzer is funded by the Development and Implementation Graduate Student Fellowship Program between LANL Los Alamos National Laboratory and University of California Southern California Hub Consortium under contract 89233218CNA000001}
\ead{reedbruin12@g.ucla.edu}

\ead[url]{}

\credit{Methodology, Software, Formal Analysis, Investigation, Writing - Review \& Editing, Visualization}

\author[2]{Scott D. Ramsey}


\ead{ramsey@lanl.gov}

\ead[url]{}

\credit{Conceptualization, Methodology, Writing - Review \& Editing, Supervision, Project administration, Funding acquisition}

\affiliation[2]{organization={Applied Physics, Los Alamos National Laboratory},
            addressline={P. O. Box 1663, MS T082}, 
            city={Los Alamos},
            postcode={97545}, 
            state={New Mexico},
            country={USA}}

\author[1,4]{Andrea L. Bertozzi}[orcid = 0000-0003-0396-7391]

\fnmark[1,2,4]
\fntext[4]{A portion of this work was conducted while Andrea L. Bertozzi was an Ulam scholar at the Center for Nonlinear Studies at Los Alamos National Laboratory.}
\ead{bertozzi@g.ucla.edu}

\ead[url]{}

\credit{Conceptualization, Methodology, Formal Analysis, Writing - Review \& Editing, Supervision, Project administration, Funding acquisition}

\affiliation[3]{organization={Department of Mechanical Engineering, University of California, Los Angeles},
            addressline={BOX 951597, 48-121 E4}, 
            city={Los Angeles},
            postcode={90095}, 
            state={California},
            country={USA}}

\author[2]{Roy Baty}[orcid = 0000-0003-0499-2636]


\ead{rbaty@lanl.gov}

\ead[url]{}

\credit{Conceptualization, Supervision, Project administration, Funding acquisition}

\cortext[1]{Corresponding author}



\begin{abstract}
We study the problem of finite time singularities in the isentropic compressible Euler equations in multiple dimensions with radial symmetry.  Our initial data consists of the radially symmetric Kidder self-similar implosion with a continuous constant density cutoff.  Motivated by our recent work on  the one-dimensional problem (Luong et al, Discrete and Continuous Dynamical Systems, 2027, 58: 197-221), we present high-resolution radially symmetric adaptive mesh refinement simulations that when combined with asymptotic similarity solution analysis present a full picture of the ensuing dynamics of the implosion. From the cutoff point in the initial data, an expanding wave emerges. The endpoints of the wave follow the "imploding" and "exploding" characteristic speeds associated with the system of hyperbolic conservation laws.  This wave has much more complex structure than the exact non-centered rarefaction wave that emerges in the one-dimensional problem.  As in the one-dimensional case, the inner Kidder solution is swallowed up by the expanding wave as time approaches the singularity time.  However, unlike the one-dimensional case, a blowup occurs along the point where the `imploding' side of the expanding wave attaches to the Kidder solution.  This composite blowup has an asymptotic self-similar form that has an interpretation in terms of the Gurderley-type self-similarity structure.  The numerical method for the PDE necessitates a Lagrangian scheme to avoid the numerical diffusion inherent in Eulerian methods.  We use a leapfrog method coupled to recursive refinement of the mesh near the blowup, which allows for very high resolution simulations, resolving many orders of magnitude of the self-similar blowup.  Consistent with prior rigorous results for these equations, the blowup does not involve mass concentration and exhibits a critical $L^p$ space for the concentration of the density, depending on the dimension of space. 
\end{abstract}





\maketitle

\section{Introduction}\label{sec:Intro}
\subsection{Background}
Nonlinear hyperbolic conservation laws exhibit \textit{finite time blowup} - the phenomenon in which solutions to nonlinear hyperbolic conservation laws lose regularity in finite time. The prototypical example of a finite time blowup is the formation of a discontinuity in the gradient of solution variables, referred to as a \textit{shock}. In the study of gas dynamics, another form of finite time blowup may occur where the solution variables themselves grow to infinity in finite time. An \textit{implosion} occurs when a flow directed towards the origin collapses into itself. Implosions occur in processes such as cavitation \cite{Hunter_cavity}, sonoluminescence \cite{Brenner_sonoluminescence}, gas cloud accretion of galaxies \cite{Woodward_implosion}, star formation \cite{Hunter_supernova}, and inertial confinement fusion (ICF) \cite{Kidder1974}.

\par We are interested in implosions for the Euler equations under special considerations. The radially symmetric, compressible Euler equations detailing the conservation of mass, momentum, and entropy are
\begin{equation}\label{eq:con_mass}
    \frac{\partial \rho}{\partial t} + v \frac{\partial \rho}{\partial r} + \rho\frac{\partial v}{\partial r} + \frac{n\rho v}{r} = 0,
\end{equation}
\begin{equation}\label{eq:con_mom}
    \frac{\partial v}{\partial t} + v \frac{\partial v}{\partial r} + \frac{1}{\rho} \frac{\partial P}{\partial r} = 0,
\end{equation}    
\begin{equation}\label{eq:con_entropy}
    \frac{\partial}{\partial t} \ln{P\rho^{-\gamma}} + v\frac{\partial}{\partial r}\ln{P \rho^{-\gamma}} = 0,
\end{equation}
where $\gamma$ is the heat capacity ratio, $P$ is the pressure, $\rho$ is the density, $v$ is the velocity, and $\ln{P\rho^{-\gamma}}$ is the entropy. The spatial index $n$ is $0$ for planar geometry, $1$ for cylindrical geometry, and $2$ for spherical geometry. We refer to equations \eqref{eq:con_mass}, \eqref{eq:con_mom}, and \eqref{eq:con_entropy} as the \textit{Full Euler Equations}, in contrast to the simplified setting we consider. When the entropy is initially constant throughout the domain, we neglect the conservation of entropy equation and write the pressure as a function only of the density, 
\begin{equation}\label{eq:pressure}
P(\rho) = C\rho^{\gamma} 
\end{equation}
where $C$ is a positive constant. The \textit{isentropic}, radially symmetric, compressible Euler equations are thus
\begin{equation}\label{eq:isen_mass}
    \frac{\partial \rho}{\partial t} + v \frac{\partial \rho}{\partial r} + \rho\frac{\partial v}{\partial r} + \frac{n\rho v}{r} = 0,
\end{equation}
\begin{equation}\label{eq:isen_mom}
    \frac{\partial v}{\partial t} + v \frac{\partial v}{\partial r} + C\gamma\rho^{\gamma - 2} \frac{\partial \rho}{\partial r} = 0.
\end{equation}
Equations \eqref{eq:isen_mass} and \eqref{eq:isen_mom} are also referred to as the \textit{potential flow} equations when the velocity is also assumed to be irrotational, which is guaranteed for flows with radial symmetry. The local speed of sound is
\begin{equation}\label{eq:sos}
    c = \sqrt{C\gamma}\rho^{\frac{\gamma - 1}{2}}.
\end{equation}

\subsection{Related Work}
We review the relevant background on finite time blowup in this subsection. First, we discuss classical results for shocks and general finite time blowup. Afterwards, the imploding solution of Guderley \cite{Guderley1942} and subsequent work is summarized. Most relevant to this current work are the linear velocity solutions introduced by Sedov \cite{Sedov_book} and applied in \cite{Kidder1974}.

\subsubsection{Shocks}
The majority of the literature around finite time blowup centers around shocks, and we mention here a few key results. In his seminal work, Lax  \cite{Lax1964} used Riemann invariants to show shock formation for small initial data of $2 \times 2$ systems of hyperbolic conservation laws. Shock formation is described in more detail for the 1D isentropic Euler equations in \cite{Lebaud_thesis}. For large data, the authors of \cite{Chen_shock} extend the use of Riemann invariants for the isentropic Euler equations.

\par For higher dimensions, Sideris \cite{Sideris_1985} gave the first result on finite time blowup for the Full Euler equations. Assuming far field initial conditions, Sideris studied the evolution of an energy to determine finite time blowup. However, this argument cannot reveal whether the blowup is a shock or an implosion. The author of \cite{Yin_3DShock} proved this finite time singularity can be a shock under the additional assumption of spherical symmetry. By relaxing the condition of spherical symmetry and instead assuming irrotational solutions, Christodoulou and Miao \cite{Christodolou_irrotational} proved initial compression leads to shocks for potential flow with constant far field initial data. In a similar setting of constant far field initial density, Chen and Wang \cite{Chen_2022} show global existence of weak solutions to the isentropic Euler equations with spherical symmetry in higher dimensions by vanishing viscosity. They note the global existence is given in $L^p_{\text{loc}}$, which permits the existence of implosion solutions. Furthermore, they show these solutions cannot concentrate mass. 

\subsubsection{Guderley}
\par While implosions are not as well-studied as shocks, there is still a significant amount of literature related to implosions, especially in recent years. The first study of implosions for compressible fluids is the work of Guderley \cite{Guderley1942}. By casting the problem into a system of ODEs in the self-similar variables, Guderley uses the corresponding phase portrait to construct the imploding solution. The solution features a shock converging to the origin which drives blowup in velocity and pressure, although density remains bounded. Guderley's approach has formed the backbone for substantial subsequent work in imploding solutions. For recent and direct continuations to Guderley's work, Jang et al.\ \cite{Jang_2025} give a rigorous construction of the Guderley solution while Cialdea et al.\ \cite{Cialdea_2025} show the Guderley solution can arise from non-shocked initial data. In an attempt to generalize the Guderley solution for non-ideal gases, the authors of \cite{Boyd2017} discuss the necessary form of the adiabatic bulk modulus to produce self-similar solutions to the Guderley problem.

\par Since the Guderley solution begins with shocked initial data, it remains unclear whether an implosion can arise from continuous or even smooth initial data. Many recent results have answered this question for the isentropic Euler equations with self-similar solutions. Jennsen and Tsikkou \cite{Jenssen2020} construct self-similar, radially symmetric, and continuous imploding solutions with strictly positive pressure, in contrast to the vanishing pressure of the Guderley solutions. Their solution has unbounded energy, but they argue they can form a finite mass and energy variant of the solution with a finite speed of propagation argument.

\par Solutions that are smooth up until implosion were studied in Merle et al.\ \cite{implosion_1}. The authors rigorously show the existence of smooth imploding self-similar solutions to the isentropic, radially symmetric Euler equations with careful phase portrait analysis. These solutions are used to construct imploding solutions to the Navier-Stokes equations in \cite{implosion_2}. Merle et al.\ show this could be done for almost every choice of heat capacity ratio $\gamma > 1$. However, one of the omitted $\gamma$ is $\gamma = 5/3$ which has physical significance as it is the heat capacity ratio of a monatomic gas. Buckmaster et al.\ improve the result by expanding the range of permissible $\gamma$ to all $\gamma > 1$. They use a computer-assisted proof to handle the special case of $\gamma = 5/3$ and rigorously handle error bounds with interval arithmetic. Shao et al.\ \cite{shao2025blowup3dcompressiblenavierstokes} give a construction specifically designed to handle the degenerate case of $\gamma = 5/3$ for the purpose of constructing smooth implosions for the Navier-Stokes equations.

\par Implosions in a non-radial symmetric setting are discussed in \cite{caolabora2025nonradialimplosioncompressibleeuler} where the authors show the radially symmetric implosions of Merle et al.\ are stable for certain non-radial perturbations. A different approach of generalizing the result of Merle et al.\ is to ``lift'' the radially symmetric imploding profile to the axisymmetric Euler equations in higher dimensions leading to vorticity blowup, done for 2D in \cite{Chen_Shkoller_Vicol_2024} and 3D in \cite{Chen_vorticity_blowup}. In the setting of the Full Euler equations with spherical symmetry, Chen et al.\ construct smooth self-similar implosions \cite{vicol_implosion}. These implosions are fully non-isentropic as the density reaches infinity in finite time while the pressure may remain finite. The shocked solution emerging from the implosion is considered in their follow-up work \cite{vicol_explosion}.

\par Despite the recent developments on implosions for the isentropic Euler equations, some practical concerns remain. The imploding solutions presented in \cite{implosion_1} and built upon in further work are not explicitly constructed. Moreover, attempts to understand these simulations computationally are stymied by the inherent instability of the solutions \cite{biasi_compute}. Given the instability of these profiles, it is unclear how to interpret them physically and numerically. 

\subsubsection{Kidder}
\par Given these drawbacks, we instead turn to the class of self-similar imploding solutions inspired by the work of Sedov \cite{Sedov_book} and popularized in the seminal work of Kidder \cite{Kidder1974}. Sedov considered solutions to the Euler equations where the velocity is affine in the spatial variable. An example of an imploding solution to the Full Euler equations satisfying this principle with a Gaussian density profile can be found in \cite{Hora_gaussian}. A characterization of imploding profiles with the ansatz of Sedov is described in \cite{Li_2006}. A recent survey of these solutions is discussed in \cite{jenssen2026}.

\par Motivated by the application of ICF \cite{Kidder1974}, Kidder constructs an explicit self-similar imploding solution to the 3D, isentropic, radially symmetric Euler equations with $\gamma = 5/3$ by considering homogeneous isentropic compression - compression where each volume element is compressed identically. We note the concept of homogeneous compression is expressed mathematically as the linear velocity ansatz attributed to Sedov. Kidder expands his solution to the compression of a hollow shell in \cite{Kidder1976}. Both solutions have formed the cornerstone of self-similar imploding solutions in ICF modeling (see \cite{Meyer_book}). The Guderley and Kidder solutions can be considered under the same viewpoint, along with other classic self-similar solutions to the compressible Euler equations, by studying the underlying self-similar dynamical system as done in \cite{Meyer_1982}.

\par The downside of Kidder's solution, and more generally solutions with linear velocity, is the unboundedness of initial data or solution as the flow progresses. In the case of Kidder's solution, the initial density is shown to be radially unbounded. Moreover, the density approaches infinity almost everywhere in the domain. For solutions with linear velocity in general, the linear growth in velocity implies unbounded velocity in the far field. Both of these properties render the physical implementation of Kidder's solution impractical. 

\par Despite concerns about unboundedness, the explicit form and apparent numerical stability of the Kidder solution are desirable enough to try to find workarounds. One attempt at ameliorating the nonphysical initial data is to continuously cutoff the initial density by a constant at some finite spatial point. The nonphysical, radially unbounded initial density is replaced with the milder condition of enforcing a constant density outside a finite ball. This setting is similar to that of \cite{Sideris_1985} and \cite{Christodolou_irrotational}. The authors of \cite{Luong_1D} explore this solution in 1D and show a nonstandard rarefaction emerges from the cutoff point and suppresses the implosion. In this work, we consider the problem in higher dimensions.

\par This paper is organized as follows. After a brief introduction to the imploding solution of Sedov and Kidder, we form the initial data with a continuous cutoff in Section \ref{sec:cutoff}. With a region of influence argument, we show why an implosion is expected to occur. Other properties such as the exact scaling of the imploding point and the mass in the imploding region are derived. In Section \ref{sec:self_similar}, we propose the self-similar ansatzes describing the blowup near implosion time. By substituting the self-similar ansatzes into the potential flow equations, we derive the dynamical system in the self-similar variable and discuss some of its features. We also present numerical solutions to the dynamical system. A brief discussion of uniqueness of the constructed solution through the framework of Weak-Strong Uniqueness is presented in Section \ref{sec:unique}. The predictions presented in the previous sections are compared to numerical simulations. Section \ref{sec:numerics} describes the numerical procedures used and comparisons between numerical simulations and analysis. Finally, we conclude with remarks and potential for future work in Section \ref{sec:conclusion}.

\section{Cutoff Implosion Solution}\label{sec:cutoff}
In this section we introduce the \textit{Kidder} initial data and solution of \cite{Sedov_book} and \cite{Kidder1974} for general dimension. A derivation can be found in \cite{Luong_1D} and in \cite{Ramsey_linear_velocity} for the Full Euler Equations. From the Kidder initial data, we construct the \textit{cutoff implosion} initial data and discuss some of the properties of the cutoff implosion solution.

\subsection{Kidder Solution}
The \textit{linear velocity ansatz}, as introduced by Sedov in \cite{Sedov_book}, is defined as
\begin{equation}\label{eq:lin_vel}
    v(r,t) = r \frac{\dot{R}(t)}{R(t)}
\end{equation}
where $R(t)$ is the \textit{scale radius}. The scale radius is interpreted as the Lagrangian position of fluid particles. Kidder arrives at the linear velocity ansatz by considering the homogeneous compression of a gas \cite{Kidder1974}. We define the parameter $\eta$ as
\begin{equation}\label{eq:eta}
    \eta = (n+1)(1 - \gamma),
\end{equation}
recalling $n = d - 1$ where $d$ is the dimension. Note $\eta = -2$ corresponds to monatomic gases in various dimensions. One such pair is $n = 2$, $\gamma = 5/3$ corresponding to spherical geometry which is the case considered in \cite{Kidder1974}. For the rest of this paper we take $\eta = -2$, and thus 
\begin{equation}\label{eq:gamma}
    \gamma = \frac{n+3}{n+1}.
\end{equation}
By also selecting particular initial conditions for $R(t)$ and its derivatives, the scale radius is 
\begin{equation}\label{eq:scale_radius}
    R(t) = \left(1 - t^2\right)^{1/2},
\end{equation}
and the Lagrangian position of a particle is
\begin{equation}\label{eq:lagrangian}
    X(t,r_0) = r_0\left(1 - t^2\right)^{1/2}.
\end{equation}
With an explicit form of the scale radius, the density and velocity are
\begin{equation}\label{eq:original_imploding_solution}
\begin{split}
    \rho(r,t) &= \left( \frac{1}{n+3} \right)^{(n+1)/{2}} \left( \frac{r}{1-t^2} \right)^{n+1}, \\
    \quad v(r,t) &= -\frac{rt}{1 - t^2},
\end{split}
\end{equation}
with initial condition
\begin{equation}\label{eq:original_imploding_condition}
    \rho(r,0) = \left( \frac{1}{n+3} \right)^{(n+1)/{2}} r^{n+1}, \quad v(r,0) = 0,
\end{equation}
assuming $C = 1$ and the density at the origin vanishes. The Kidder solution achieves implosion at time $t =1$ with all particles arriving at the origin simultaneously. However, the initial density and the density and velocity for times $0<t<1$ are unbounded in the far field.
This issue is the main motivation for cutting off the initial condition in the far field and observing what happens.

\par For $n > 0$, the Kidder solution begins subsonic and becomes supersonic everywhere except the origin before imploding. By comparing the velocity to the speed of sound using \eqref{eq:original_imploding_solution}, the Kidder solution remains subsonic up until time
\begin{equation}\label{eq:sonic_time}
    t^* = \left( \frac{1}{n+1} \right)^{1/2}
\end{equation}
and remains supersonic afterwards. The time the solution becomes supersonic occurs is strictly less than $1$ for all $n > 0$.

\subsection{Cutoff Implosion Initial Data}
By continuously cutting off the initial data of \eqref{eq:original_imploding_condition} at some point $r^*$, we obtain the \textit{cutoff implosion} initial data
\begin{equation}\label{eq:cutoff_imploding_condition}
\begin{split}
    &\rho(r,0) = \begin{cases}
    \left( \frac{1}{n+3} \right)^{(n+1)/{2}} r^{n+1} &\text{ for } r \leq r^* \\
    \left( \frac{1}{n+3} \right)^{(n+1)/{2}} (r^*)^{n+1} &\text{ for } r \geq r^*
    \end{cases}, \\
    &v(r,0) = 0.
\end{split}
\end{equation}
Now, the initial condition for the density is bounded, unlike the Kidder initial condition. The constant far field condition acts as a barrier to prevent particles from the whole domain from being pulled to the origin. While the initial data is continuous, it is not differentiable, and as a result, we expect new behavior to emerge from the cutoff point $r^*$. The Lagrangian trajectories of particles originating from $r_0 < r^*$ are expected to head towards the origin, while the Lagrangian trajectories of particles originating from $r_0 > r^*$ remain stationary - creating a gap in the characteristic diagram as seen in Figure \ref{fig:characterestics}.

\begin{figure}[pos=h]
    \centering    
    \includegraphics[width=1\linewidth]{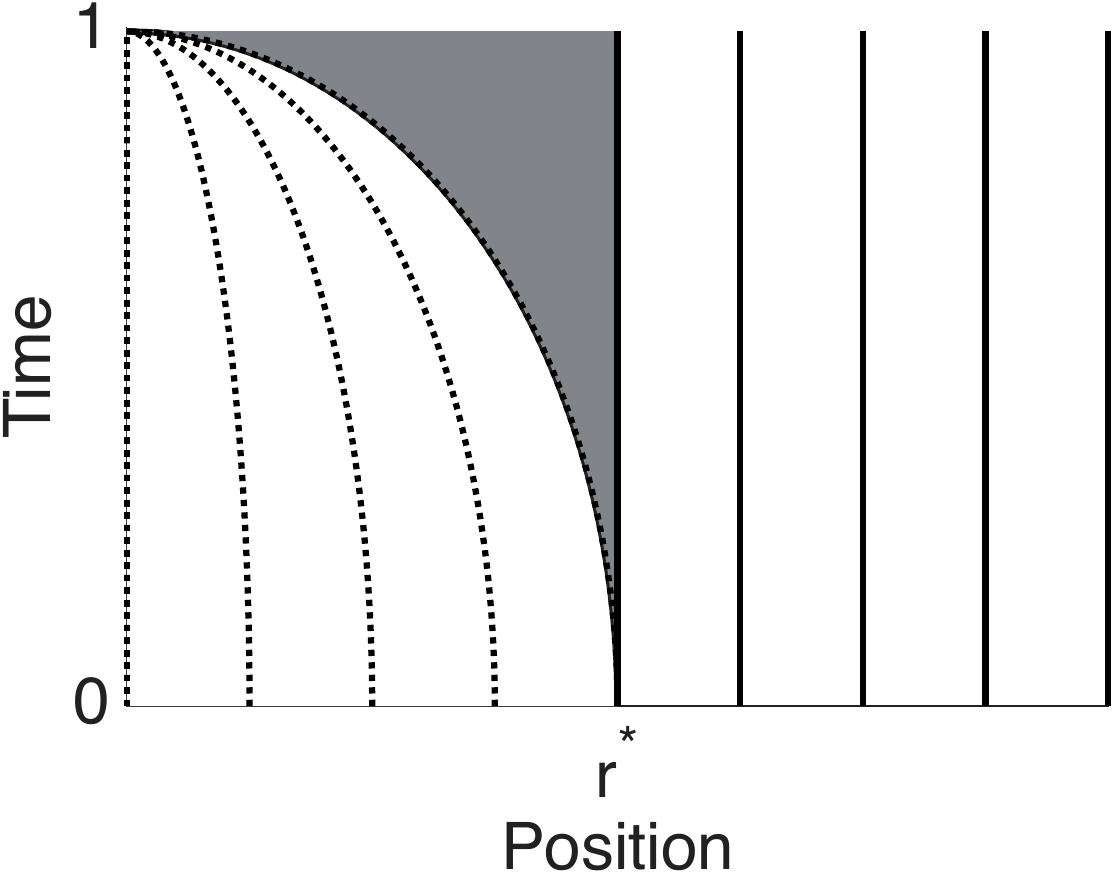}
    \caption{Predicted characteristics stemming from the cutoff implosion initial data. The characteristics left of the cutoff position $r^*$ (dashed lines) are expected to follow the Lagrangian position described in \eqref{eq:lagrangian}. The characteristics right of the cutoff position $x_0$ (solid lines) are expected to remain vertical since the initial velocity and density is constant in this region.}
    \label{fig:characterestics}
\end{figure}

The $1$-dimensional case is covered in \cite{Luong_1D}. The authors show a rarefaction emerges from the cutoff point with nonstandard rarefaction variable $y = x/(t+1)$. The left endpoint of the rarefaction region arrives at the origin at time $t = 1$, which is the same as the original implosion time. However, an implosion does not occur as both the density and velocity remain bounded near the origin. In this work, we consider the case for all higher dimensions - in other words, when $n > 0$. 

\par Analysis in higher dimensions is complicated by the presence of the source term $n\rho v/r$ in \eqref{eq:isen_mass}. We do not expect a traditional rarefaction to emerge from the cutoff point unlike in $1$D. Asymptotic expansions, as presented in \cite{Floch_part1,Floch_part2}, can be used to approximate the solution for small times. The rarefaction solution of the Riemann problem without the source term is used as the leading order term. As a result, the rarefaction solution can be used as an approximate solution for the equations with source term for small time, a feature leveraged by numerical methods based on the asymptotic expansions described in \cite{Ben_artzi_1984,Toro_book}. Global asymptotic expansions exist when the source term depends only on the state \cite{Floch_global}. However, we are interested in the behavior of the solution for larger times, and our source term depends on the spatial variable.

\subsection{Region of Influence of the Cutoff Position}
\par A ``rarefaction'' emerges over a region influenced by the cutoff point $r^*$. In higher dimensions, this region is a spherical shell (annulus in 2D). The outer and inner radii of this shell travel at the speeds of the eigenvalues of the isentropic Euler equations, $\lambda_{\pm} = v \pm c$ respectively where $c$ is given by \eqref{eq:sos}. As a nonlinear system, the eigenvalues are not constant but depend on the local state of the material. Since the eigenvalues represent the speed of the radii, the inner radius $r_{l}$ and outer radius $r_r$ satisfy the ODEs
\begin{equation}\label{eq:endpoint_ODE}
    \dot{r_l} = \lambda_{-}(r_l), \quad \dot{r_r} = \lambda_{+}(r_r), \quad r_l(0) = r_r(0) = r^*.
\end{equation}
Due to the diverging structure of the characteristics from $r^*$, we expect the solution to be continuous. This implies the solution should connect continuously to the Kidder solution at the inner radius and to the constant state at the outer radius. Both of these states are explicitly known. Note the speed of sound $c$ from \eqref{eq:sos} can be written in terms of $n$ given the assumed relationship between $\gamma$ and $n$ used to derive the Kidder solution:
\begin{equation}\label{eq:sos_special}
    c = \left( \frac{n+3}{n+1} \right)^{1/2}\rho^{1/(n+1)}.
\end{equation}
The outer radius $r_r(t)$ satisfies the ODE
\begin{equation}\label{eq:r_r_ODE}
    \dot{r}_r = \frac{r^*}{\sqrt{n+1}}, \quad r_r(0) = r^*,
\end{equation}
which has the solution
\begin{equation}\label{eq:r_r}
    r_r(t) = r^*\left(\frac{t}{\sqrt{n+1}} + 1\right).
\end{equation}
The inner radius $r_l(t)$ satisfies the more complicated ODE
\begin{equation}\label{eq:r_l_ODE}
    \dot{r}_l = -r_l \left( \frac{t + (n+1)^{-\frac{1}{2}}}{1 - t^2} \right), \quad r_l(0) = r^*
\end{equation}
which has solution
\begin{equation}\label{eq:r_l}
    r_l(t) = r^*\left( (1 - t)^{\frac{1+(n+1)^{-1/{2}}}{2}} (1+t)^{\frac{1 - (n+1)^{-1/{2}}}{2}} \right).
\end{equation}
Both formulas \eqref{eq:r_l} and \eqref{eq:r_r} reduce to the left and right endpoint formulas derived in \cite{Luong_1D} when $n = 0$. The outer radius moves linearly, regardless of $n$. On the other hand, the inner radius always arrives at the origin at time $t = 1$, also regardless of $n$. 

\par As $r_l(t)$ descends to the origin, every $v$ characteristic from $(0, r^*)$ will eventually cross the $v - c$ characteristic stemming from $r^*$. 
Suppose a Lagrangian particle initially starts at $r_0 \in (0,r^*)$. By equating its Lagrangian position ascribed in \eqref{eq:lagrangian} with the formula for the inner radius in \eqref{eq:r_l}, the two paths collide at time
\begin{equation}\label{eq:time_cross}
    t = \frac{1 - \tau}{1 + \tau}, \quad  \tau = \left( \frac{r_0}{r^*} \right)^{2\sqrt{n+1}}
\end{equation}
with $t \in (0,1)$ since $\tau < 1$. So, every particle originating on $(0,r^*)$ is eventually affected by the left moving characteristic emanating from $r^*$.

\par We can track the density and velocity at the inner radius by evaluating the Kidder solution at $r_l(t)$. From the construction of $r_l(t)$, the cutoff implosion solution is the same as the Kidder solution on the interval $[0,r_l(t)]$. The density and velocity of the cutoff implosion solution evaluated at $r_l(t)$ is
\begin{equation}\label{eq:rho_r_l}
\begin{split}
    \rho(r_l(t),t) &= \left( n + 3 \right)^{-(n+1)/{2}} (r^*)^{n+1} (1-t)^{\frac{(n+1)^{1/{2}} - (n+1)}{2}} \\
    &\times (1+t)^{\frac{-(n+1)^{1/2} - (n+1)}{2}}
\end{split}
\end{equation}
and
\begin{equation}\label{eq:v_r_l}
    v(r_l(t),t) = -r^*t (1-t)^{{\frac{-1+(n+1)^{-1/{2}}}{2}}}(1+t)^{{\frac{-1-(n+1)^{-1/{2}}}{2}}}.
\end{equation}
Whether an implosion can occur is determined by the exponent of the $(1-t)$ term. When $n > 0$, the exponent is negative implying the density at $r_l(t)$ goes to positive infinity and the velocity at $r_l(t)$ goes to negative infinity as $t$ approaches $1$. Note $n = 0$ is a special case, as the $(1-t)$ term drops out entirely leading to finite limits for the density and velocity instead. This behavior matches the limiting behavior observed in \cite{Luong_1D}. So, implosion occurs in all dimensions higher than $1$. 

\par Note $\rho(r_l(t),t)$ is not a monotone increasing function in $t$ on $[0,1)$ while $v(r_l(t),t)$ is a monotone decreasing function. Although $\rho(r_l(t),t)$ eventually goes to infinity as $t \to 1$, there will be a time interval when the density decreases instead. In contrast, $v(r_l(t),t)$ will always decrease to negative infinity. Since on the interval $[0,r_l(t)]$ the velocity is the same as the Kidder solution and $v(r,t) = 0$ for $r \geq r_r(t)$, we expect $v(r_l(t),t)$ to be a global minima for velocity at any given time.

\section{Self-Similar Blowup Limit}\label{sec:self_similar}
In our analysis thus far, we have concluded an implosion occurs by tracking the state of the inner radius $r_l(t)$. However, such analysis does not give insight on the structure of the solution in a neighborhood around $r_l(t)$. While inwards from $r_l(t)$ lies the Kidder solution, it is not clear how the solution outwards from $r_l(t)$ behaves. 


\par
In the 1D problem we found an exact similarity in the "rarefaction" domain.  Our numerical solutions show that this does not happen in higher dimensions. Suppose to the contrary there exists an exact similarity solution with density $\tilde{\rho}$ with similarity variable $y$ on the interval $[r_l(t),r_r(t)]$ for $0 < t < 1$. Since $r_l(t)$ is monotone decreasing and $r_r(t)$ is monotone increasing, the interval $[r_l(t_1), r_r(t_1)]$ is contained in the interval $[r_l(t_2), r_r(t_2)]$ when $0 < t_1 < t_2 < 1$. If the exact similarity solution $\tilde{\rho}$ is continuous, then $\tilde{\rho}\left([r_l(t_1), r_r(t_1)]\right) \subset \tilde{\rho}\left([r_l(t_2), r_r(t_2)]\right)$. However, in Figure \ref{fig:min_loss} there exist values in the range of $\rho$ on $[r_l(t),r_r(t)]$ that appear in earlier times but not later times. In Subfigure \ref{subfig:min_loss_1}, the minimum of $\rho$ on $[r_l(t_1),r_r(t_1)]$ is less than $1$. But, in Subfigure \ref{subfig:min_loss_2}, the minimum of $\rho$ on $[r_l(t_2),r_r(t_2)]$ is greater than $1$. This contradicts the property that the image of $\tilde{\rho}$ on a subinterval should be contained in the image of $\tilde{\rho}$ on a a larger interval. This suggests that the ensuing dynamics is more complicated.

\begin{figure*}[pos=h]
     \centering
     \begin{subfigure}[h]{0.495\linewidth}
         \centering
         \includegraphics[width=\textwidth]{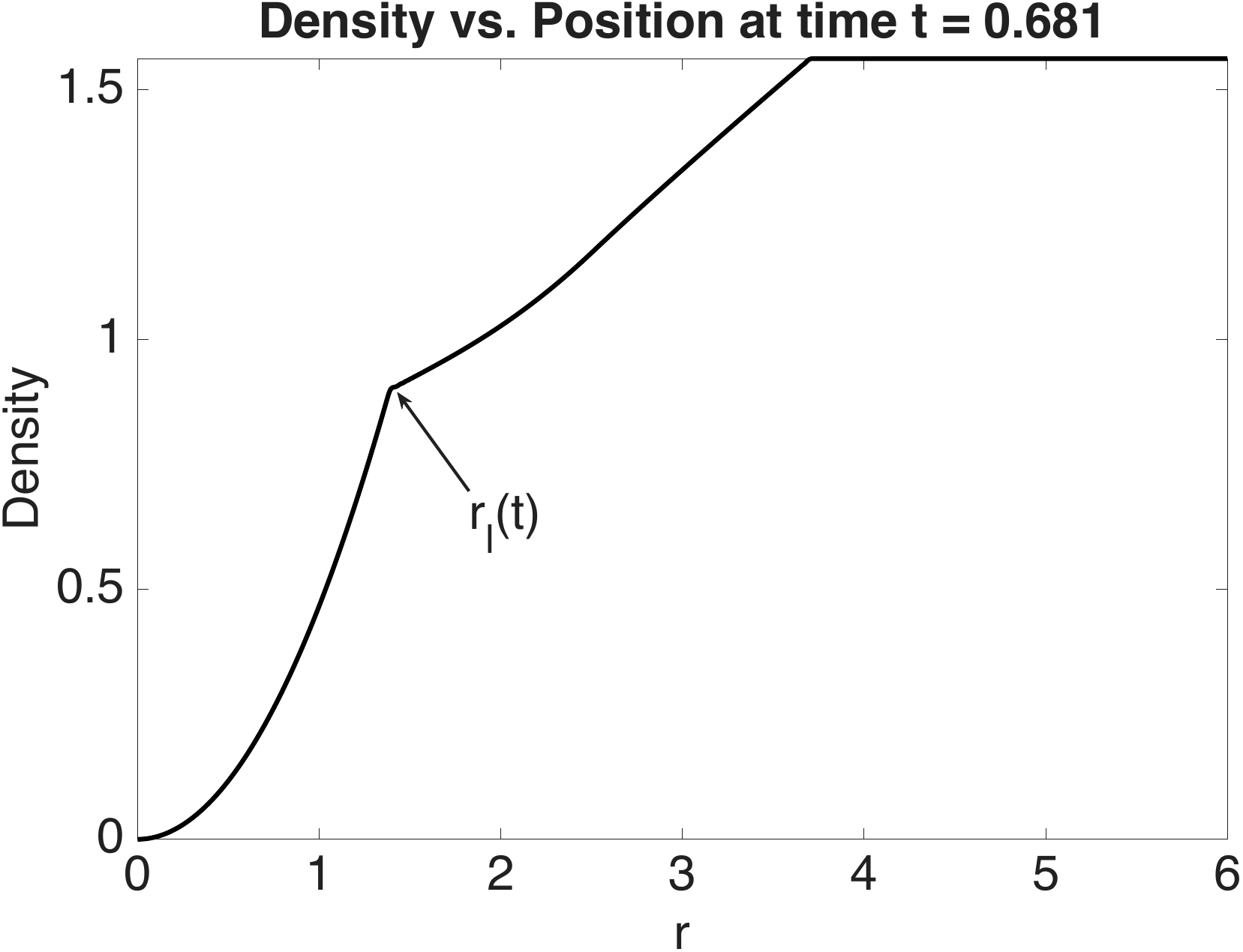}
         \caption{Density at an earlier time $t_1$ where $\min_{r \in [r_l(t_1),r_r(t_1)]}\rho(r,t_1) < 1$. }
         \label{subfig:min_loss_1}
     \end{subfigure}
     \hfill
     \begin{subfigure}[h]{0.495\linewidth}
         \centering
         \includegraphics[width=\textwidth]{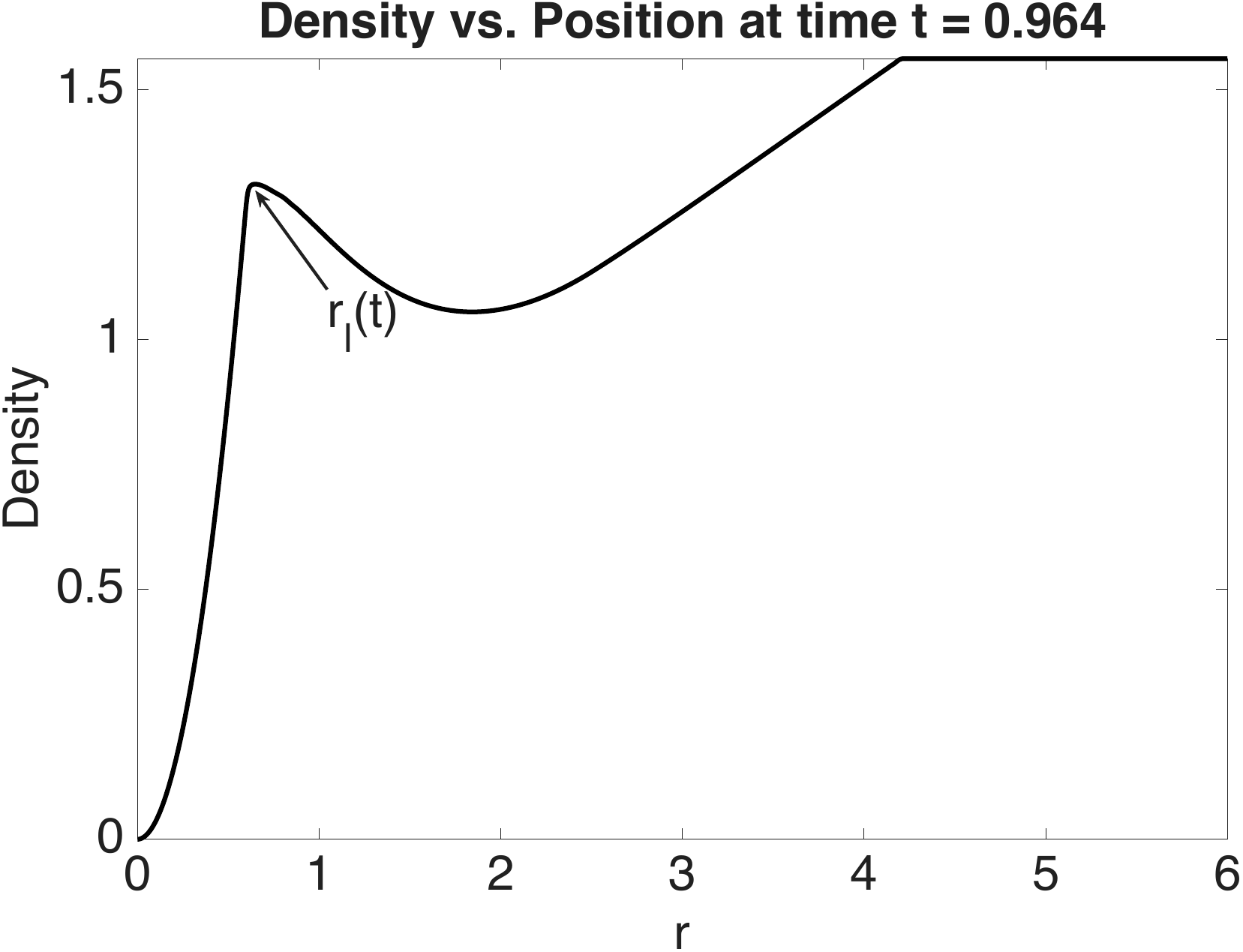}
         \caption{Density at a later time where $\min_{r \in [r_l(t_2),r_r(t_2)]}\rho(r,t_2) > 1$.}
         \label{subfig:min_loss_2}
     \end{subfigure}
        \caption{Snapshots of solutions to equations \eqref{eq:isen_mass} and \eqref{eq:isen_mom} from the cutoff implosion initial data \eqref{eq:cutoff_imploding_condition} with parameters $n = 1$ and $r^* = 2.5$. The solution on the interval $[r_l(t),r_r(t)]$ cannot be an exact similarity solution since the image of the density on a subinterval must be contained in the image of any larger interval.}
        \label{fig:min_loss}
\end{figure*}

\par Given the blowup of the solution at $r_l(t)$, we search for a \textit{self-similar solution of the first kind} that appears as the asymptotic limit of the cutoff implosion solution. For this type of scaling, the similarity variable and exponents are determined from dimensional analysis and invariant scalings of the problem \cite{Barenblatt_self_similar}. To derive the self-similar solution, we first state the functional form in general and derive the corresponding system of ODEs. The system of ODEs can be further simplified using knowledge of the density and velocity at the inner radius derived in the previous section. Suppose we have the similarity variable
$y = r/{r_l(t)}$ and self-similar forms for the density and velocity
\begin{equation}\label{eq:ss_ansatz}
    \rho(y,t) = \tau(t) \rhohat(y), \quad v(y,t) = -\tau(t)^{\beta}\vhat(y)
\end{equation}
where $\tau(t)$ is a function of time that goes to infinity as $t \to 1$ and $\beta$ is a to be determined exponent. The motivation behind considering the similarity variable $y$ is the time dependent length scale, $r_l(t)$, is derived from the left moving characteristic at $r^*$ at time $t = 0$ and acts as a principal length scale in the implosion. Additionally, $r_l(t)$ is asymptotically equivalent to the time scaling of the Kidder solution similarity variable close to implosion time. 

\par The Kidder solution is invariant under this transformation, so it can be rewritten as 
\begin{equation}\label{eq:og_ss}
\begin{split}
    \rho(y,t) &= \tau(t)\rhohat(y), \quad \rhohat(y) = \rhohat(1)y^{n+1}, \\
    v(y,t) &= -\tau(t)^{\beta}\vhat(y), \quad \vhat(y) = \vhat(1)y, \\
    \tau(t) &= \left( \frac{r_l}{1 - t} \right)^{n+1} .
\end{split}
\end{equation}
Substituting the ansatzes in \eqref{eq:ss_ansatz} into the potential flow equations \eqref{eq:isen_mass} and \eqref{eq:isen_mom} yields
\begin{equation}\label{eq:ss_system_dirty}
    \begin{cases}
\tau_t \hat{\rho} - \tau \frac{y}{r_l}\dot{r_l}\hat{\rho}' - \frac{\tau^{\beta+1}}{r_l}\hat{\rho}\hat{v}' - \frac{\tau^{\beta+1}}{r_l}\hat{v}\hat{\rho}' - \frac{n\tau^{\beta + 1}\hat{\rho}\hat{v}}{y r_l} &= 0, \\
-\beta \tau_t \tau^{\beta - 1}\hat{v} + \tau^{\beta} \frac{y}{r_l}\dot{r_l}\hat{v}' + \frac{\tau^{2\beta}}{r_l}\hat{v}'\hat{v} + \gamma \frac{\tau^{\gamma - 1}}{r_l} \hat{\rho}'\hat{\rho}^{\gamma - 2} &= 0.
\end{cases}
\end{equation}

\subsection{\textit{A Priori} Reductions}
To further simplify the ODE system in \eqref{eq:ss_system_dirty}, we use the knowledge of the implosion dynamics to explicitly determine the unknown terms. In doing so, we consider the implosion dynamics in the limit to implosion time and evaluate time-dependent, non $(1-t)$ terms at $t = 1$. This is justified since we wish to find the self-similar profile the solution to the PDE asymptotically approaches when $t$ is very close to $t = 1$. The resulting solution is a genuine self-similar solution with respect to the new variables and serves as an asymptotic limit.

\subsubsection{Derivation of $\tau,\beta,\rhohat(1),\vhat(1)$}
From the definition of $y$, the spatial point $r = r_l(t)$ maps to $y = 1$. The exponent of $(1-t)$ in $r_l(t)$ can be expressed as $\frac{1}{s^*}$ where $s^* = \frac{2}{1 + (n+1)^{-1/2}}$. Evaluating \eqref{eq:og_ss} at $y = 1$,
\begin{equation*}
    \begin{split}
        \tau(t)\rhohat(1) &= (1-t)^{\frac{(n+1)^{{1}/{2}} - (n+1)}{2}}  \left( n + 3 \right)^{-\frac{n+1}{2}} k^{n+1},
    \end{split}
\end{equation*}
where all of the non $(1-t)$ terms in \eqref{eq:rho_r_l} are evaluated at $t = 1$ and 
\begin{equation}\label{eq:k}
k = r^*2^{{\frac{-1-(n+1)^{-1/{2}}}{2}}}.
\end{equation}
A similar derivation in $v$ yields
\begin{equation*}
    -\tau(t)^{\beta}\vhat(1) = -(1-t)^{{\frac{-1+(n+1)^{-1/{2}}}{2}}} k.
\end{equation*}
Thus, $\tau(t)$ encodes the implosion dynamics and is
\begin{equation}\label{eq:tau}
\tau(t) = (1-t)^{\frac{(n+1)^{1/{2}} - (n+1)}{2}} \sim r_l(t)^{-(n+1)(s^* -1)}
\end{equation}
with $\beta = (n+1)^{-1}$.
The initial condition at $y = 1$ for $\rhohat$ and $\vhat$ is also identified as
\begin{equation}\label{eq:ss_ic}
\begin{split}
    \rhohat(1) = \left( n + 3 \right)^{-(n+1)/{2}} k^{n+1}  , \quad
    \vhat(1) = k.
\end{split}
\end{equation}

\subsubsection{Relation of $\dot{r_l}(t)$ to $\tau(t)$}
Recall $r_l(t)$ satisfies the ODE \eqref{eq:endpoint_ODE}. Substituting the ansatzes \eqref{eq:ss_ansatz} into ODE \eqref{eq:endpoint_ODE}, we obtain
\begin{equation}\label{eq:r_l_tau_relation}
    \dot{r}_l(t) = -\tau^{\beta} \left( \hat{v}(1) + \left( \frac{n+3}{n+1} \right)^{\frac{1}{2}}\hat{\rho}(1)^{\frac{1}{n+1}} \right).
\end{equation}
Equivalently, $\dot{r}_l(t) = -A\tau^{\beta}$, where $A$ is
\begin{equation}\label{eq:A}
    A = k \left( 1 + (n+1)^{-\frac{1}{2}} \right)
\end{equation}
after substituting in the initial conditions \eqref{eq:ss_ic}. Note $A$ is a positive constant.

\subsubsection{Relation of $\tau_t(t)$ to $\tau^{\beta + 1}(t) / r_l(t)$}
By using the explicit forms of $\tau(t)$ and $r_l(t)$, we see $\tau(t)$ satisfies the ODE
\begin{equation}\label{eq:r_l_tau_ODE}
    \tau_t(t) = \lambda \frac{\tau^{\beta + 1}(t)}{r_l(t)},
\end{equation}
where $\lambda$ is the constant
\begin{equation}\label{eq:lambda}
    \lambda = k\left((n+1) - (n+1)^{1/2}  \right).
\end{equation}
Note $\lambda$ is a positive constant.

\subsection{Autonomous Transformation}
By using the \textit{a priori} knowledge collected in the previous subsections, the non-autonomous ODE system \eqref{eq:ss_system_dirty} simplifies to
\begin{equation}\label{eq:ss_system}
\begin{split}
\begin{bmatrix}
Ay - \hat{v} & -\hat{\rho} \\
\frac{n+3}{n+1} \hat{\rho}^{\frac{1-n}{1+n}} & \hat{v} - A y
\end{bmatrix}\begin{bmatrix}
\hat{\rho}' \\ \hat{v}'
\end{bmatrix} = \begin{bmatrix}
\frac{n\hat{\rho}\hat{v}}{y} - \lambda \hat{\rho} \\
\beta \lambda\hat{v}
\end{bmatrix}
\end{split}.
\end{equation}

As done by Guderley \cite{Guderley1942} (see \cite{Meyer_1982,implosion_1,Sedov_book}  for more details), consider the Emden transformation
\begin{equation}
\begin{split}
    \hat{\rho}(y) &= \left( \frac{n+1}{n+3} \right)^{\frac{(n+1)}{2}}y^{n+1}\sigma(x)^{n+1}, \\
    \hat{v}(y) &= y w(x), \quad y = e^x.
\end{split}
\end{equation}
In this change of variables, $\sigma$ is related to the speed of sound. Moreover, the Kidder solution, as written in the self-similar form \eqref{eq:og_ss}, is mapped by the Emden transformation to the constant solution \begin{equation}\label{eq:ss_auto_ic}
    \sigma(0) = (n+1)^{-1/{2}} k, \quad w(0) = k, \quad P^* = \begin{bmatrix}
        \sigma(0) & w(0)
    \end{bmatrix}^T,
\end{equation}
where $P^*$ is referred to as a \textit{sonic point}. The initial condition of the non-autonomous system \eqref{eq:ss_ic} is also mapped to \eqref{eq:ss_auto_ic} under the Emden transformation. The full system \eqref{eq:ss_system} becomes
\begin{equation}\label{eq:ss_system_auto}
\begin{split}
\begin{bmatrix}
    (A - w) & - \frac{1}{n+1}\sigma \\ (n+1)\sigma & (w - A) 
\end{bmatrix}\begin{bmatrix}
    \sigma' \\ w'
\end{bmatrix} = \\
\begin{bmatrix}
    -\frac{\lambda\sigma}{n+1} - A\sigma +  2w\sigma \\\beta\lambda w + Aw - w^2 - (n+1) \sigma^2 
\end{bmatrix},
\end{split}
\end{equation}
and can be formally written as
\begin{equation}\label{eq:ss_system_auto_informal}
\begin{split}
\begin{bmatrix}
\sigma' \\ w'
\end{bmatrix} &= \left(\sigma^2 - (A - w)^2\right)^{-1}\begin{bmatrix}
    (w - A) & \frac{1}{n+1}\sigma \\ -(n+1)\sigma & (A - w) 
\end{bmatrix}\times \\
&\begin{bmatrix}
    \frac{-\lambda\sigma}{n+1} - A\sigma +  2w\sigma \\ \beta\lambda w + Aw - w^2 - (n+1) \sigma^2 
\end{bmatrix}.
\end{split}
\end{equation}
System \eqref{eq:ss_system_auto_informal} is recognized as the celebrated autonomous self-similar system featured in \cite{Guderley1942} and \cite{implosion_1} with our choice of parameters. The zero level set of the denominator $\Delta = \left(\sigma^2 - (A - w)^2\right)$ is the \textit{sonic line} and is the determinant of the matrix $M(\sigma,w)$ of the left hand side of \eqref{eq:ss_system_auto}. The initial condition \eqref{eq:ss_auto_ic} lies on the sonic line, implying singularity of system \eqref{eq:ss_system_auto_informal}. 

\subsection{Solutions of the Autonomous System}
Motivated by the numerical simulations presented in Section \ref{sec:numerics}, we search for a self-similar profile that attaches to the Kidder solution and decays to zero for large values of $y = e^x$. This subsection presents an explanation on how additional solutions are constructed.
Denote $\vec{b}(\sigma,w)$ as the right hand side vector of \eqref{eq:ss_system_auto} and take $M(\sigma,w)$ as the mass matrix defined earlier. Noting the anti-involutive structure of $M(\sigma,w)$ for any choice of $(\sigma,w)$ and the non-invertibility of $M(\sigma(0),w(0))$, we discuss the eigenvalues and eigenvectors of $M(\sigma(0),w(0))$.
\begin{lemma}\label{lem:eigen}
    Suppose a $2 \times 2$ real matrix $N$ has the entries
    \[
    N = \begin{bmatrix}
        a & b \\ c & -a
    \end{bmatrix}
    \]
    and $\det{N} = 0$. Then, $N$ has eigenvalue $\lambda = 0$ with only one associated eigenvector $\vec{v} = \begin{bmatrix}
        a & c
    \end{bmatrix}^T$.
\end{lemma}
\begin{pf}
    Suppose $N$ has $\det{N} = 0$, so $a^2 + bc = 0$. The eigenvalues $\lambda$ of $N$ satisfy the characteristic equation
    \[
    \lambda^2 - a^2 - bc = 0,
    \]
    but given $\det{N} = 0$, the only eigenvalue is $\lambda = 0$. Computing the eigenvectors for $\lambda = 0$ shows there is only one eigenvector, $\vec{v} = \begin{bmatrix}
        a & c
    \end{bmatrix}^T$.
\end{pf}
Since $M(\sigma(0),w(0))$ satisfies the conditions in Lemma \ref{lem:eigen}, $M(\sigma(0),w(0))$ has only an eigenvalue at $\lambda = 0$ and only one corresponding eigenvector. When considering \eqref{eq:ss_system_auto} as a matrix equation evaluated at \eqref{eq:ss_auto_ic}, the vector on the right hand side $\vec{b}(\sigma(0),w(0))$ is the zero vector. Thus, valid initial slopes $\begin{bmatrix}
    \sigma'(0) & w'(0)
\end{bmatrix}^T$ are rescaled eigenvectors of $M(\sigma(0),w(0))$, which take the form
\begin{equation}\label{eq:ss_auto_slope}
    \begin{bmatrix}
    \sigma'(0) \\ w'(0)
\end{bmatrix} = \alpha \begin{bmatrix}
A - w(0) \\ (n+1) \sigma(0)
\end{bmatrix} =\alpha \begin{bmatrix}
k(n+1)^{-1/2} \\ k(n+1) ^{1/2}
\end{bmatrix},
\end{equation}
where $\alpha$ is a scalar and the expression for the eigenvector is taken from Lemma \ref{lem:eigen}.
\par We have thus established the possibility of finding a nontrivial solution to \eqref{eq:ss_system_auto} with initial condition \eqref{eq:ss_auto_ic} and its initial slope. We search for solutions with $\alpha < 0$ that describe a heteroclinic orbit from the sonic point to the origin. The phase portrait of \eqref{eq:ss_system_auto} is presented in in Fig 4 and such a heteroclinic orbit is colored in green. For smooth and bounded autonomous systems, a heteroclinic orbit would start at $x=-\infty$ and end at $x=+\infty$.  The singularity in the ODE at the sonic point allows for solutions that start at finite $x$ and flow to zero in as $x\to+\infty$.

\subsection{Phase Portraits}

\begin{figure*}[pos=h]
        \centering
     \begin{subfigure}[h]{0.70\linewidth}
         \centering
         \includegraphics[width=\textwidth]{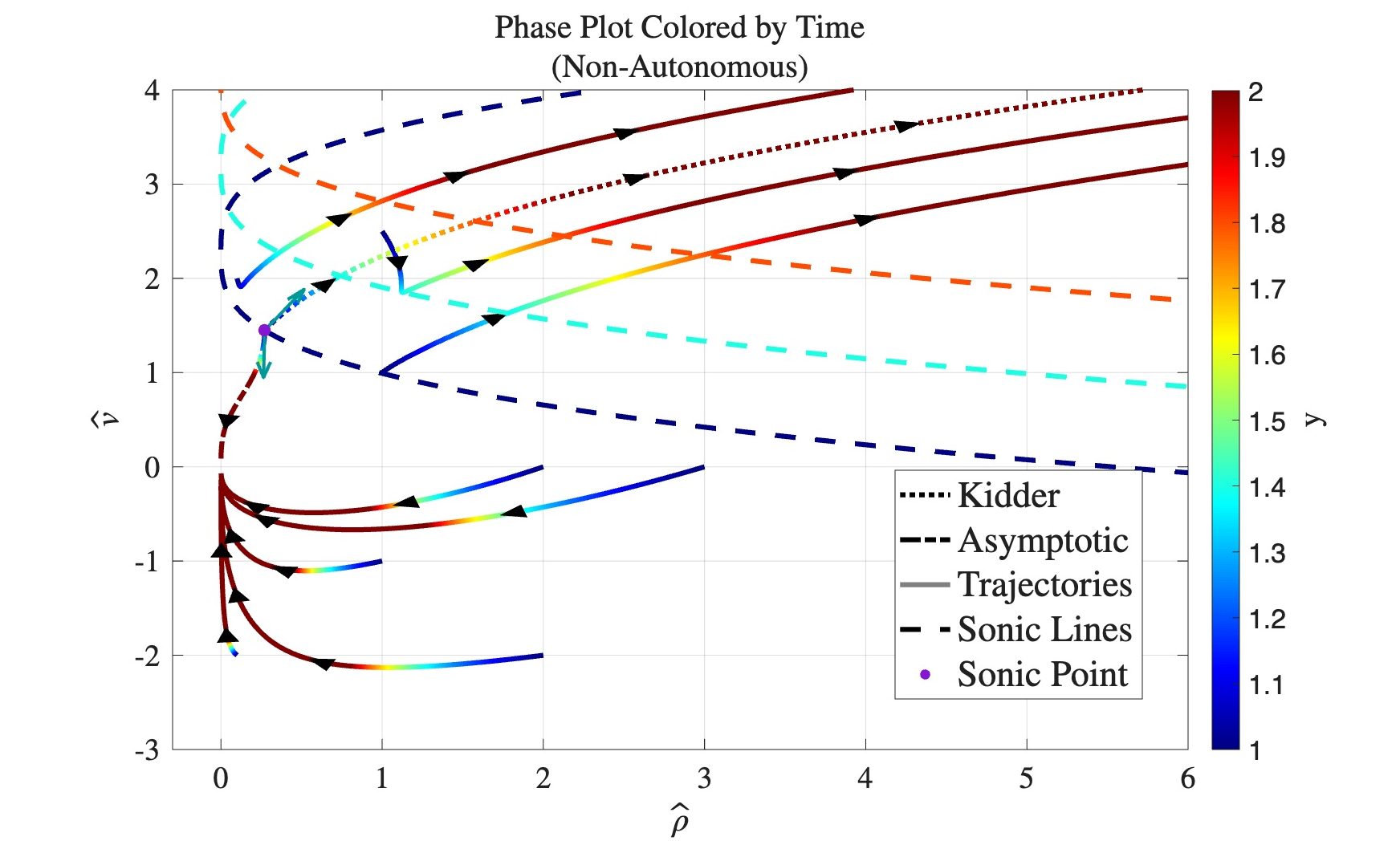}
         \caption{Non-autonomous phase plot}
         \label{subfig:pp_auto_zoom_out}
     \end{subfigure}
     \begin{subfigure}[h]{0.28\linewidth}
         \centering
         \includegraphics[width=\textwidth]{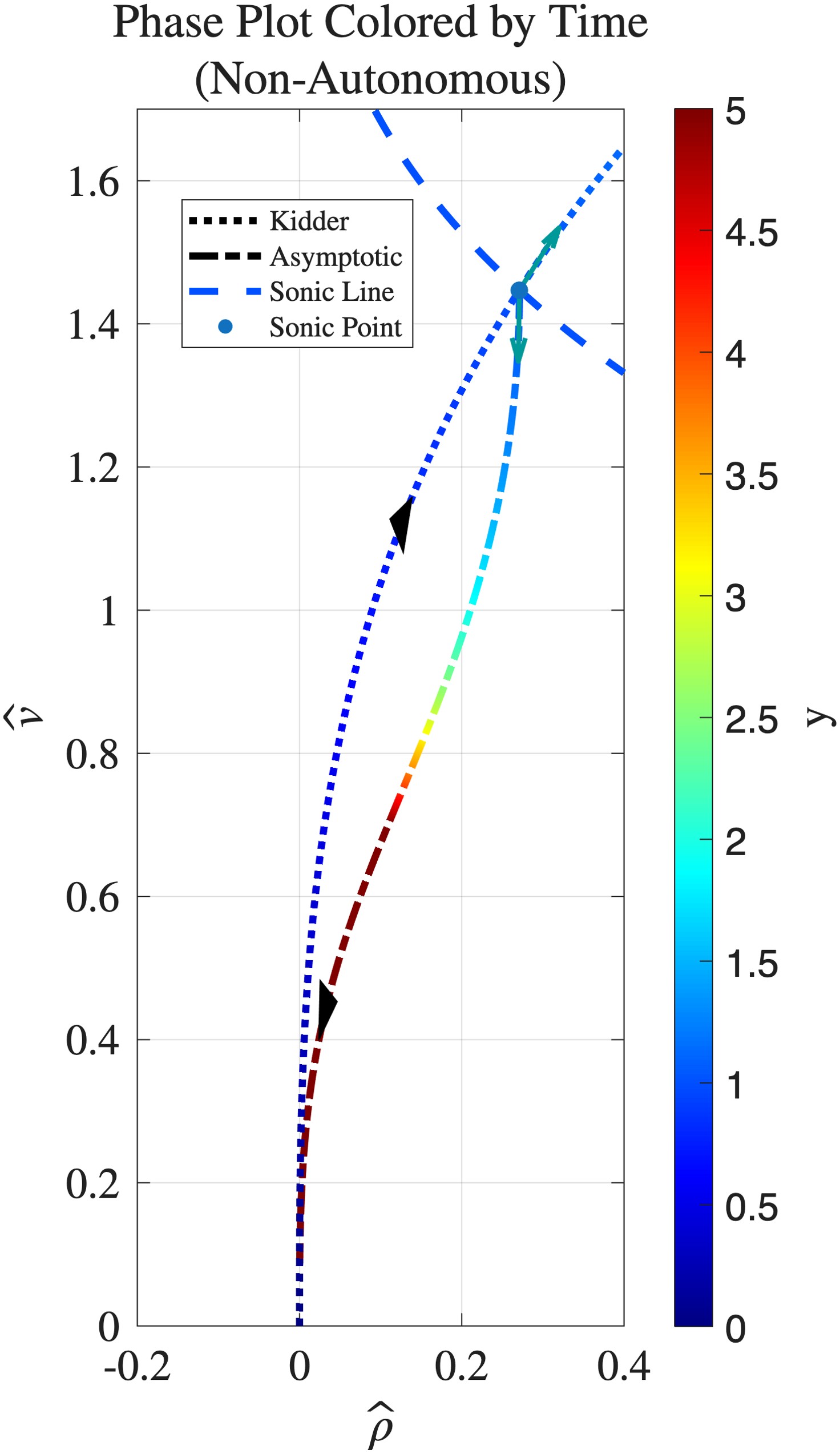}
         \caption{Composite Kidder and Asymptotic solution}
         \label{subfig:pp_auto_zoom}
     \end{subfigure}
    \caption{Phase portrait of non-autonomous system \eqref{eq:ss_system} for $n = 2$ and $r^* = 2.5$ in \ref{subfig:pp_auto_zoom_out} for $1 \leq y \leq 2$. The sonic line for various $y$ are displayed as dashed lines. Vectors indicating permissible initial slopes at the sonic point are drawn. The composite Kidder and asymptotic solution is represented in phase space for $0 \leq y \leq 5$ in \ref{subfig:pp_auto_zoom_out}.}
    \label{fig:pp_non_auto}
\end{figure*}

\begin{figure*}[pos=h]
    \centering
    \includegraphics[width=0.95\textwidth]{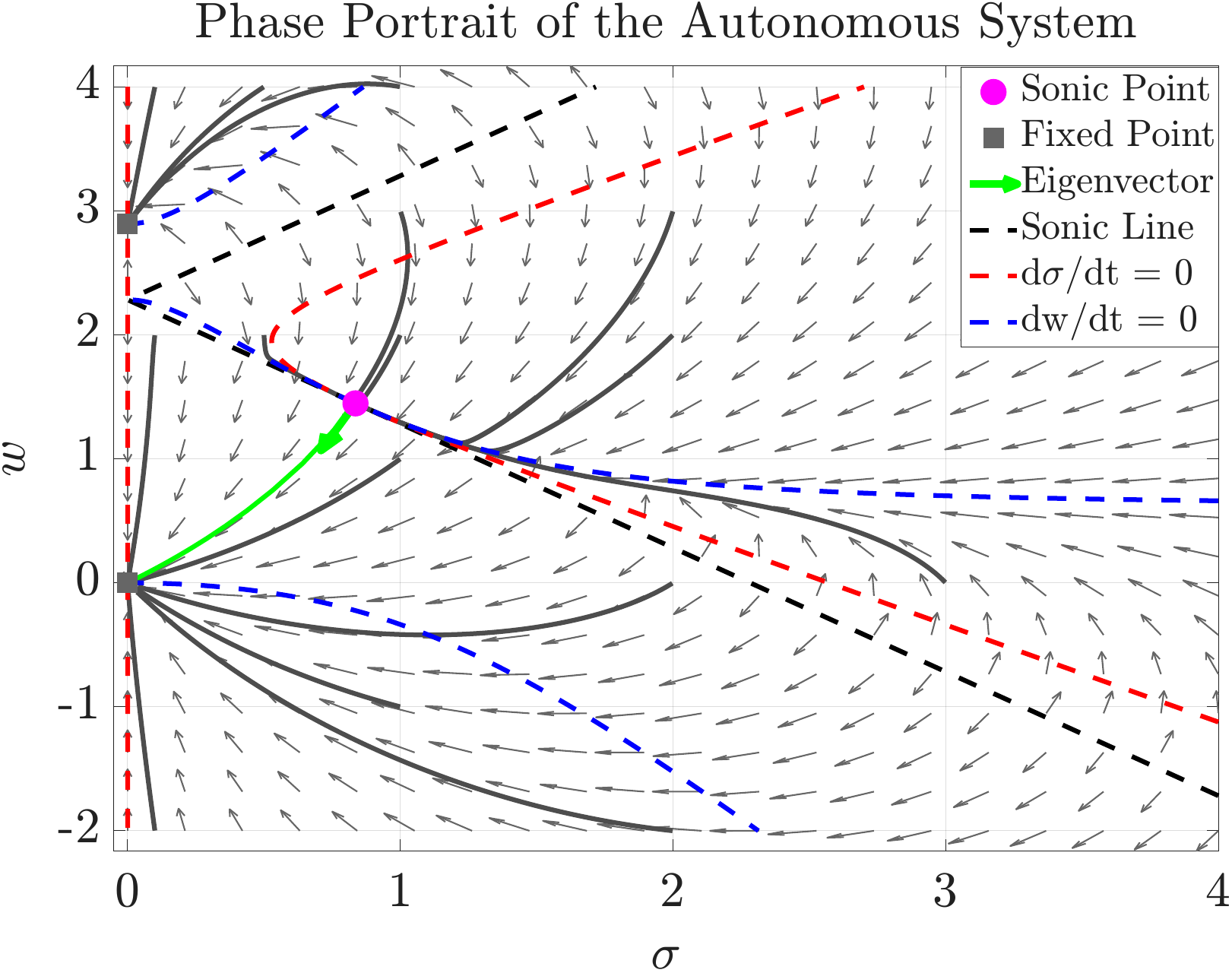}
    \caption{Phase portrait of autonomous system \eqref{eq:ss_system_auto} for $n = 2$ and $r^* = 2.5$. The fixed points, nullclines, and sonic line are displayed. At the sonic point, the eigenvector \eqref{eq:ss_auto_slope} is shown indicating the only valid direction to leave the sonic point in the direction of the origin. The green curve represents the limiting self-similar profile to the right of $r_l(t)$ for the cutoff implosion solution near implosion time.}
    \label{fig:pp_auto}
\end{figure*}

We present phase portraits for the non-autonomous system and autonomous system in Figure \ref{fig:pp_non_auto} and Figure \ref{fig:pp_auto} respectively. This section reviews these dynamical systems in more detail.  The sonic line is also plotted in the phase portraits. In the non-autonomous system, the sonic line varies in $y$ and is plotted for various $y$. In the autonomous system, the sonic line is fixed. A region of the phase portrait is carved out by the sonic lines, denoted as the \textit{sonic cone}. The nullclines are also plotted.

\par We first discuss features of the autonomous phase portrait of Figure \ref{fig:pp_auto}. The fixed points are points such that $\sigma'$ and $w'$ are zero, which correlate with the $\sigma,w$ such that the right hand side of \eqref{eq:ss_system_auto} is the zero vector. In the physically relevant regime of $\rho \geq 0$, there are three fixed points: the origin, $P^*$, and the point $P_1 = \begin{bmatrix}
    0 & 2k
\end{bmatrix}^T$. The origin and $P_1$ are both attractive fixed points which lie outside the sonic cone. On the other hand, $P^*$ lies on the sonic line itself. Trajectories starting inside the sonic cone cannot cross the sonic line, so they either must converge to $P^*$ or escape to infinity. From the phase portrait, we see trajectories within the sonic cone approach $P^*$. Now for trajectories outside of the sonic cone near $P^*$, they move away from $P^*$ and converge to the attractive fixed point at the origin. We observe $P^*$ exhibits half-stability as it is attractive to trajectories in the sonic cone while it is repulsive to trajectories outside the sonic cone. Note both the origin and $P_1$ are true fixed points such that the only permissible slope passing through both fixed points are the zero vector, in contrast to $P^*$. Note despite both nullclines intersecting the sonic line at $\begin{bmatrix}
    0 & A
\end{bmatrix}^T$ in Figure \ref{fig:pp_auto}, this is not a fixed point. When system \eqref{eq:ss_system_auto} is evaluated at $\begin{bmatrix}
    0 & A
\end{bmatrix}^T$, the left hand side of \eqref{eq:ss_system_auto} is the zero vector while the right hand side of \eqref{eq:ss_system_auto} is $\begin{bmatrix}
    0 & \beta \lambda A
\end{bmatrix}^T$, which is not the zero vector.

\par In the non-autonomous system, the sonic line is still a separatrix in the phase plane, but the sonic line varies with $y$. Trajectories beginning below the sonic line still approach the origin, like in the autonomous case. However, trajectories starting above the sonic line exhibit different behavior. Some approach infinity in $\rhohat$ and $\vhat$ monotonically. The Kidder solution in non-autonomous coordinates does this. On the other hand, some trajectories begin to converge to the origin but deflects off the moving sonic line. After the deflection, the trajectory moves towards infinity. Another group of trajectories move towards $\rhohat = 0,\vhat = \infty$, corresponding to the fixed point $P_1$ in autonomous coordinates. Subfigure \ref{subfig:pp_auto_zoom} shows the composite Kidder and asymptotic solution on the phase plot. From $y = 0$, the Kidder solution moves from the origin with the sonic line up until $y = 1$. Afterwards, the asymptotic solution returns to the origin.  

\par The non-autonomous system \eqref{eq:ss_system} is shown to be the same as the one derived in \cite{Guderley1942} and \cite{implosion_1} up to rescaling in the specific case where $n$ and $\gamma$ are given by the relation $-2 = (n + 1)(1 - \gamma)$. Subsequently, the autonomous system given in \eqref{eq:ss_system_auto} and \eqref{eq:ss_system_auto_informal} is a special case of the dynamical system \cite{Guderley1942} and further explored by authors such as \cite{Meyer_1982}, \cite{implosion_1}, \cite{shao2025blowup3dcompressiblenavierstokes}. Details on the equivalence is given in Appendix \ref{ap:equal}. 

\par The dynamical system in \cite{Guderley1942} can be written in the form $w' = {\Delta_1}/{\Delta}$ and $\sigma' = {\Delta_2}/{\Delta}$ where $\Delta$ is the sonic line. Trajectories connecting endpoints on opposite sides of the sonic line must pass through the sonic line, but can only do so smoothly at \textit{sonic points} where $\Delta = \Delta_1 = \Delta_2$. The only sonic points with non-vanishing density are $P_2$ and $P_3$. It is through sonic point $P_2$ which Merle. et al construct their smooth imploding profile \cite{implosion_1}. They seek a solution that begins to the right of the sonic lines and terminates at the origin, which is to the left of the sonic lines. Another important point is the fixed point that does not lie on the sonic line, $P_5$. 

\par In the work of Merle et al.\ \cite{implosion_1} and subsequent work such as Shao et al.\ \cite{shao2025blowup3dcompressiblenavierstokes}, they construct a sequence of smooth imploding solutions with blowup speeds $s_n < s^*$, where their similarity variable is defined as $y = r(T - t)^{1/{s}}$. Note the limiting value $s^*$ corresponds to the previously defined value of $s^*$ as the reciprocal $(1-t)$ exponent of $r_l(t)$ in \eqref{eq:r_l}. However, Merle et al.\ note their construction fails when $n = 2$ and $\gamma = {5}/{3}$ because as $s \to s^*$, the three points $P_2,P_3,P_5$ all coincide with each other. The work of Shao et al.\ specifically handles this degeneracy. For our system, the initial condition $P^*$ corresponds to this special coincidence. Solutions to \eqref{eq:ss_system_auto} that arrive at the origin may be able to be viewed as the limiting profile to the sequence of smooth imploding solutions constructed in \cite{shao2025blowup3dcompressiblenavierstokes}.

\subsection{Numerical Solutions of the Self-Similar System} 
We directly compute numerical solutions of the non-\\autonomous system \eqref{eq:ss_system} to determine the limiting self-similar profiles $\rhohat$ and $\vhat$. Lemma \ref{lem:eigen} can be applied to \eqref{eq:ss_system} at initial condition \eqref{eq:ss_ic} to show initial slopes belong to the affine space
\begin{equation}\label{eq:ss_slope}
\begin{split}
    V &= \left\{ \begin{bmatrix}
        \mu \\ 0
    \end{bmatrix} + \alpha \begin{bmatrix}
n\hat{\rho}(1)\hat{v}(1) - \lambda \hat{\rho}(1) \\
\beta \lambda\hat{v}(1)
\end{bmatrix}^T : \alpha \in \mathbb{R} \right\}, \\
    \mu &= k^{n+1}(n+3)^{-\frac{n+1}{2}} \left( (n+1) - (n+1)^{\frac{1}{2}} \right).
\end{split}
\end{equation} 
We use numerical method \textsc{ode23t} in \textsc{MATLAB} to solve \eqref{eq:ss_system} at initial condition \eqref{eq:ss_ic} and initial slope \eqref{eq:ss_slope} with $\alpha = -3$ for $n = 1$ and $\alpha = -2$ for $n = 2$. The choice of $\alpha$ is determined heuristically to ensure the solutions decay to zero. We directly solve the system as an ODE with a mass matrix to avoid unnecessary matrix inversions and for ease of providing the initial slope. These solutions for $2$D and $3$D are shown in Figure \ref{fig:hat}. From the numerical solutions, we see $\rhohat$ and $\vhat$ are monotone decreasing in $y$. Additionally, $\vhat$ (and equivalently $\sigma$) are positive for $y \geq 1$, which is behavior established for the profiles of \cite{shao2025blowup3dcompressiblenavierstokes}.

\begin{figure*}[pos=h]
     \centering
     \begin{subfigure}[h]{0.495\linewidth}
         \centering
         \includegraphics[width=\textwidth]{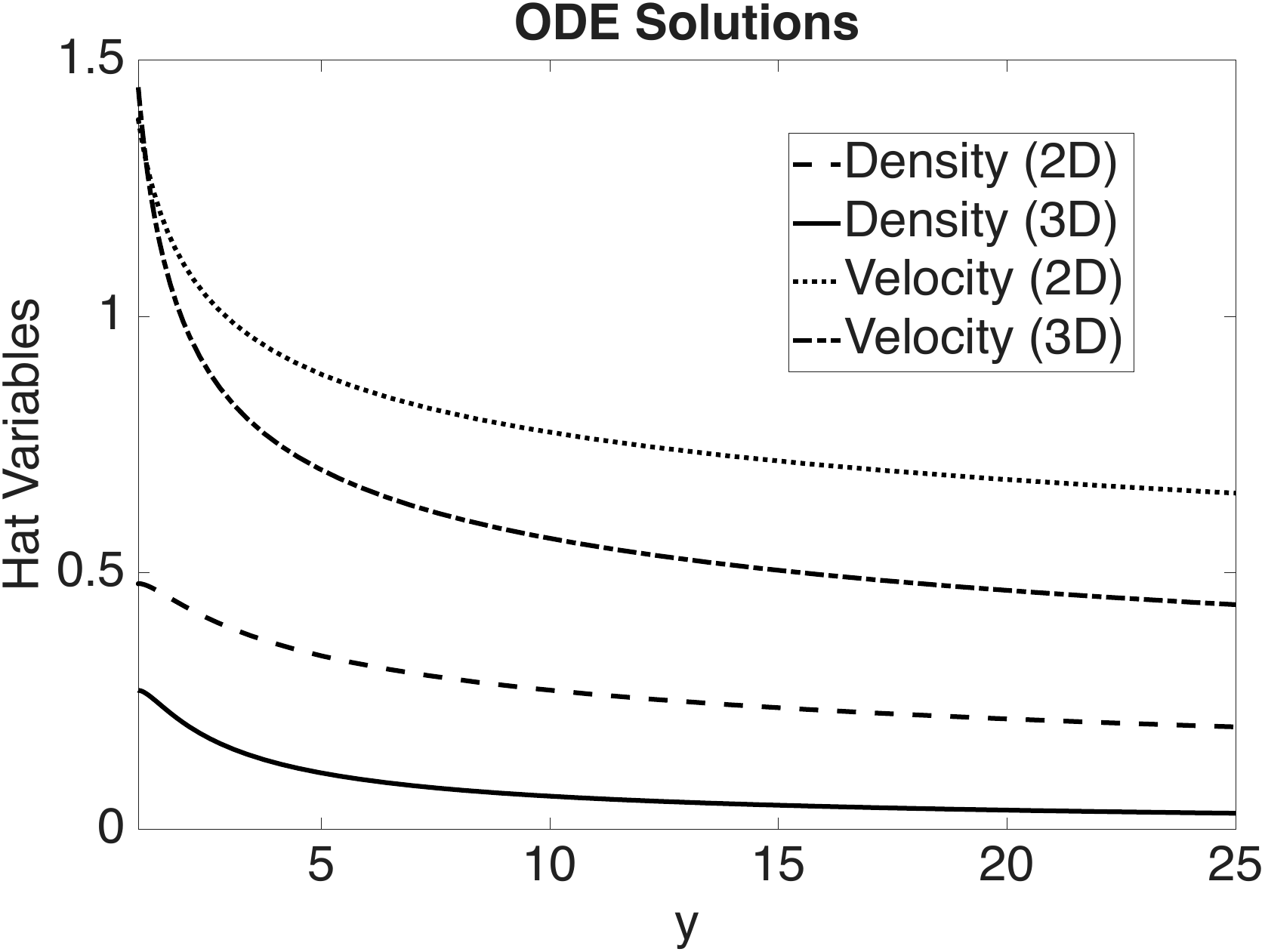}
         \caption{Standard Variables}
         \label{subfig:ode_hat}
     \end{subfigure}
     \hfill
     \begin{subfigure}[h]{0.495\linewidth}
         \centering
         \includegraphics[width=\textwidth]{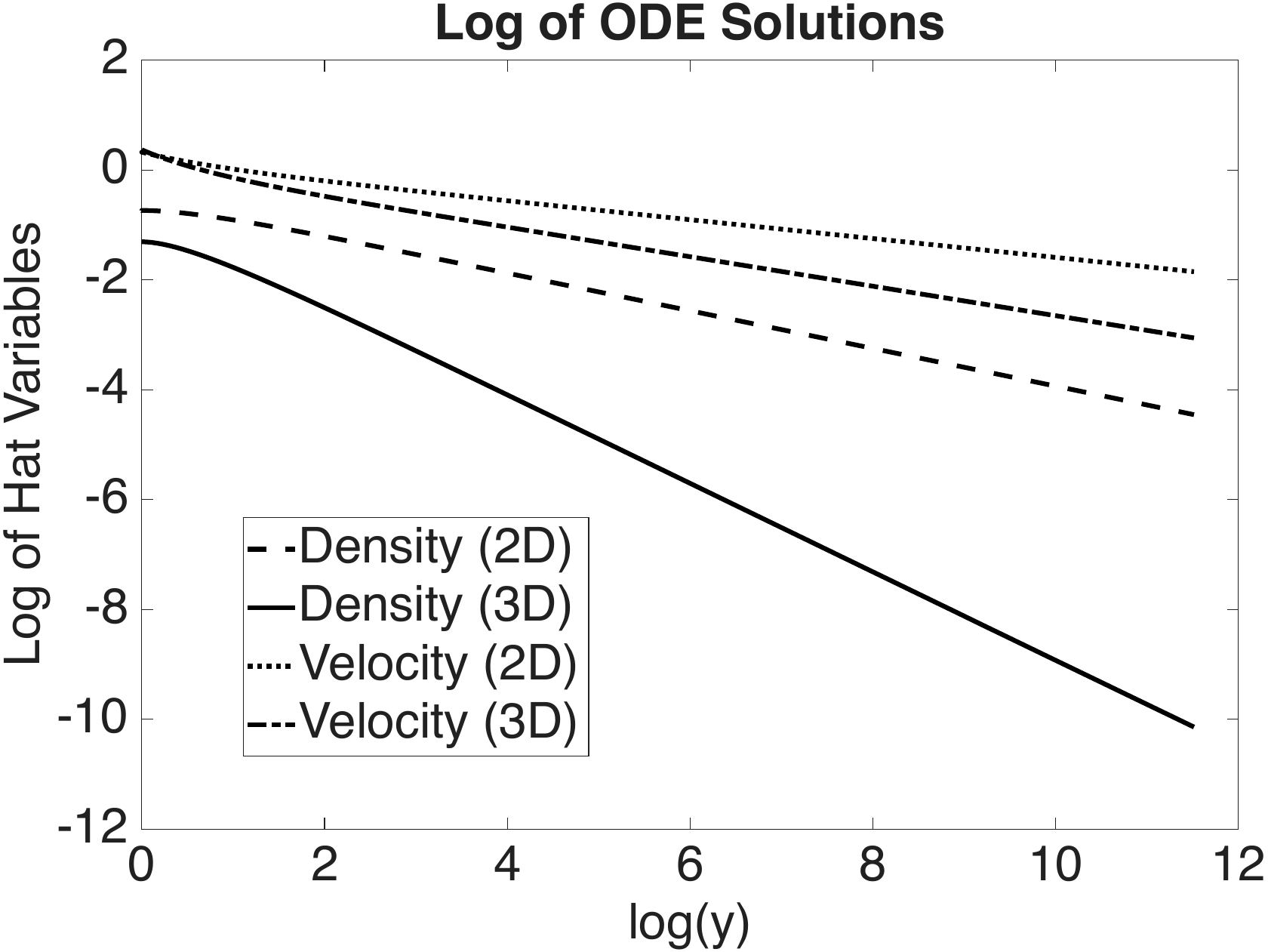}
         \caption{Log-Log Scale}
         \label{subfig:ode_hat_log}
     \end{subfigure}
        \caption{Solutions to ODE system \eqref{eq:ss_system} with initial condition \eqref{eq:ss_ic} for $2$D and $3$D. The initial slope is chosen to be $\alpha = -3$ for $n = 1$ and $\alpha = -2$ for $n = 2$ in \eqref{eq:ss_slope}. The linear behavior of the numerical solutions in the log-log scale of Subfigure \ref{subfig:ode_hat_log} indicate power law behavior of solutions in the far field. The computed slopes of the numerical solutions in log-log scale are displayed in Table \ref{tab:hat_slope}.}
        \label{fig:hat}
\end{figure*}

\subsection{Far Field Behavior of the Similarity Solution}
From Subfigure \ref{subfig:ode_hat_log}, we observe $\rhohat$ and $\vhat$ appear linear in log-log scaling for large values of $y$, corresponding to the far field dynamics of $\rhohat$ and $\vhat$. This implies a negative power law relationship between $y$ and $\rhohat$ and $\vhat$ in this regime. By substituting an ansatz of the form $y^{-k_{\rhohat}}$ and $y^{-k_{\vhat}}$ for the density and velocity respectively into \eqref{eq:ss_system_dirty}, the terms originating from time derivatives dominate the terms originating from spatial derivatives and the source term. The exponents describing the far field power law behavior can be computed as
\begin{equation}\label{eq:hat_slopes}
\begin{split}
    \rhohat(y) \sim y^{-\lambda/{A}}, \quad \vhat(y) \sim y^{-\lambda/(\beta A)}, \\
    -\frac{\lambda}{A} = (n+1)\frac{1 - \sqrt{n+1}}{1 + \sqrt{n+1}} = (n+1)(1 - s^*).
\end{split}
\end{equation}
The slopes of the numerical solutions of $\log(\rhohat)$ and $\log(\vhat)$ for $y \geq 100$ are compared to the theoretical values in \eqref{eq:hat_slopes} in Table \ref{tab:hat_slope}, showing excellent agreement between expected far field behavior and the numerical solution.

\begin{table}[pos=h]
    \centering
    \begin{tabular}{|c|c|c|}
    \hline
     Case  & Numerical & Theoretical \\
    \hline
        Density (2D) & $-0.3431$ & $-0.3431$ \\ \hline
        Density (3D) & $-0.8052$ & $-0.8038$ \\ \hline
        Velocity (2D) & $-0.1717$ & $-0.1716$ \\ \hline
        Velocity (3D) & $-0.2683$ & $-0.2679$ \\ \hline
    \end{tabular}
    \caption{Comparison between the slope of $\log{y}$ to numerically computed $\log{\rhohat}$ and $\log{\vhat}$, shown in Subfigure \ref{subfig:ode_hat_log}, to the theoretical power law scaling given in \eqref{eq:hat_slopes}. The slope is numerically computed for $y \geq 100$.}
    \label{tab:hat_slope}
\end{table}
\subsection{Mass Concentration and $L^p_{\loc}$ Blowup}\label{subsec:L^p}
To better understand the physical properties of the proposed implosion, we want to determine whether the mass concentrates and forms a measure at implosion time. The result of \cite{Chen_2022} states that given initial data with a constant farfield density and finite initial energy with respect to the constant farfield, the solution is bounded in $L^p_{\loc}$ and mass concentration does not occur. Their existence result is given for $L^p_{\loc}$ as the constant farfield density prevents existence in $L^p$ for all of space and to allow for implosions where a state variable approaches infinity. Because the result is stated in terms of finite relative energy, an exact $p$ is not stated but is implied to be at least $\gamma$ - stemming from the relative internal energy term. We show mass concentration does not occur for the cutoff implosion solution and give the maximal $p$ such that the density is bounded in $L^p_{\loc}$ into the implosion time.
\par We study the mass on the interval $[0,R]$ where $R << r_r(t)$ as $t \to 1$. Given a radial density distribution, the $L^p$ norm of $\rho$ on an interval $[a,b]$ is given by
\begin{equation}\label{eq:mass_formula}
    \Vert \rho(r,t)\Vert_{L^p\left([a,b]\right)}= \left( \int_a^b s(n)r^n \rho(r,t)^p dr \right)^{\frac{1}{p}},
\end{equation}
where $s(n)$ is a the surface area of the $n+1$-dimensional unit ball. To evaluate the $L^p$ norm on the interval $[0,R]$, the interval may be split into two sub-intervals: $[0, r_l(t)]$ and $[r_l(t),R]$. On the interval $[0, r_l(t)]$, given the density is the Kidder solution, the $L^p$ norm is 
\begin{equation}\label{eq:kidder_l_p}
\begin{split}
    \Vert \rho(r,t)\Vert_{L^p\left([0,r_l(t)]\right)} &= \left( s(n) \left( \frac{1}{n+3} \right)^{(n+1)/{2}}\frac{1}{2n+2} (r^*)^{2n+2} \right)^{\frac{1}{p}}  \\
    &\times (1-t)^{\frac{(n+1)(p+1)}{2} + \frac{(n+1)^{1/2}(p+1)}{2} - (n+1)p} \\
    &\times (1+t)^{\frac{(n+1)(p+1)}{2} - \frac{(n+1)^{1/2}(p+1)}{2} - (n+1)p}.
\end{split}
\end{equation}
To determine whether the $L^p$ norm blows up as $t \to 1$, the exponent of $(1-t)$ must be nonnegative, which occurs when
\begin{equation}\label{eq:p_crit}
    p \leq p_{\crit} = \frac{(n+1)^{1/2} + 1}{(n+1)^{1/2} - 1}
\end{equation}
when $n > 0$. Note $p_{\crit}$ may be expressed as other quantities as well, so $p_{\crit} = (n+1)(A/\lambda) = \left( (n+1)(s^* - 1)\right)^{-1}$. It is always true $1 \leq \gamma \leq p_{\crit}$ for $n > 0$, establishing the bound of $\rho(r,t)$ in the $L_{\loc}^{\gamma}$ norm. Note as $n \to \infty$, $p_{\crit}$ decreases monotonically and approaches $1$. This implies as the spatial dimension increases, the maximal $L^p$ existence of $\rho(r,t)$ decreases. However, for any finite $n$, the $L^1_{\loc}$ norm is always finite. The mass in the region $[0, r_l(t)]$ (corresponding to $p = 1$) is
\begin{equation}\label{eq:kidder_mass}
\begin{split}
    \Vert \rho(r,t)\Vert_{L^1\left([0,r_l(t)]\right)} &= s(n)\left( \frac{1}{n+3} \right)^{(n+1)/{2}}\frac{1}{2n+2} (r^*)^{2n+2} \times \\
    &(1-t)^{(n+1)^{1/2}} (1+t)^{-(n+1)^{1/2}}.
\end{split}
\end{equation}
\par To evaluate the integral on the remaining sub-interval, we change variables to $y = {r}/{r_l(t)}$ and also write $\rho(r,t)$ in terms of the self-similar form \eqref{eq:ss_ansatz}. We further subdivide the interval to $[1, a]$ and $\left[a, \frac{R}{r_l(t)}\right]$, where $[1,a]$ represents the non-power law behavior of $\rhohat(y)$ while $\rhohat(y)$ is taken to be the power law described in \eqref{eq:hat_slopes} on $\left[a, \frac{R}{r_l(t)}\right]$, motivated by Subfigure \ref{subfig:ode_hat_log}. 
The integral decomposes into
\begin{equation}\label{eq:mass_decompose_1}
\begin{split}
s(n) r_l(t)^{n+1}\tau(t)^p \int_{1}^a \hat{\rho}(y)^p y^n dy
\end{split}
\end{equation}
and
\begin{equation}\label{eq:mass_decompose_2}
s(n) r_l(t)^{n+1}\tau(t)^p\int_a^{\frac{R}{r_l(t)}} Cy^{-\frac{p\lambda}{A}} y^n dy.    
\end{equation}
The integral of \eqref{eq:mass_decompose_1} is finite as we observe $\rhohat$ is monotone decreasing from Figure \ref{fig:hat} and $\rhohat(1)$ is fixed and given by \eqref{eq:ss_ic}. Blowup depends only on the $r_l(t)^{n+1}\tau(t)^p$ term. Recalling $r_l(t) \to 0$ and $\tau(t) \to \infty$ as $t \to 1$ and defined as \eqref{eq:r_l} and \eqref{eq:tau} respectively, we find blowup does not occur as long as $p \leq p_{\crit}$ - the same condition for the Kidder solution. The integral of \eqref{eq:mass_decompose_2} evaluates to 
\begin{equation*}
    \frac{s(n)}{n +1 - p\frac{\lambda}{A}}\left( \tau(t)r_l(t)^{\frac{\lambda}{A}} \right)^p \left(R^{n+1 - \frac{p\lambda}{A}} - a^{n+1 - \frac{p\lambda}{A}} \right)
\end{equation*}
after substituting the power law given in \eqref{eq:hat_slopes}. Potential blowup is independent of $p$, and recognizing that ${A}/{\lambda} = {p_{\crit}}/{n+1}$, the time dependent term cancels out. Assuming $p < p_{\crit}$, only the integral \eqref{eq:mass_decompose_2} contributes to the $L^p_{[0,R]}$ norm of $\rho$. Additionally, the power law dependence of $\Vert \rho(r,t)\Vert_{L^p([0,R])}$ on $R$ as $t \to 1$ is given by
\begin{equation}\label{eq:Lp_power_law}
    R^{\frac{n+1}{p} - \frac{\lambda}{A}}.
\end{equation}

\par From the previous calculations, we find the maximal $p$ such that $\rho(r,t)$ stays bounded in $ L^{p}_{\loc}$, is given by $p_{\crit}$. Blowup of the $L^p_{[0,R]}$ norm depends only on the blowup of the $L^p$ norm of the Kidder solution on $[0,r_l(t)]$ and not the self-similar solution past $r_l(t)$. Since $p_{\crit} > 1$, mass concentration cannot form around the origin as predicted by the result of \cite{Chen_2022}.

\section{Heuristic Argument for Uniqueness}\label{sec:unique}
We discuss a heuristic argument for determining uniqueness of the solution presented thus far. In a similar vein to \cite{Luong_1D}, we leverage the Weak-Strong Uniqueness results introduced in \cite{Dafermos_entropy} and \cite{Diperna_entropy}. For a system of $m$ conservation laws in $\mathbb{R}^n$
\begin{equation}\label{eq:conservation law}
    \partial_t U + \sum_{\alpha = 1}^n \partial_\alpha G_\alpha (u(x,t)) = 0, 
\end{equation}
where $x \in \mathbb{R}^n$ and $u(x,t) \in \mathbb{R}^m$, the functions $\Phi,\Psi$ are an \textit{entropy} and \textit{entropy flux} pair respectively if they map from $\mathbb{R}^m \to \mathbb{R}$, $\Phi$ is convex, and the pair satisfies the entropy inequality
\begin{equation}\label{eq:entropy_con_law}
    \Phi(u)_t + \nabla \cdot \Psi(u) \leq 0.
\end{equation}
Satisfying the entropy inequality is equivalent to satisfying $D\Psi_{\alpha}(u) = D\Phi(u)DG_{\alpha}(u)$ for $\alpha = 1,\cdots,m$. Inequality \eqref{eq:entropy_con_law} is an equality when $u$ is a sufficiently smooth solution of \eqref{eq:conservation law}. The entropy inequality is used to distinguish which weak solutions - in particular, solutions of\eqref{eq:conservation law} with shocks  are physically admissible. For the isentropic Euler equations, the most commonly used entropy is the total energy $E = \frac{1}{2}\rho|\vec{v}|^2 + \frac{C}{\gamma - 1}\rho^{\gamma}$ with entropy flux $E\vec{v} + C\rho^{\gamma}\vec{v}$.

\par The Weak-Strong Uniqueness result of \cite{Dafermos_entropy} and \cite{Diperna_entropy} states if \eqref{eq:conservation law} admits an entropy flux pair $(\Phi,\Psi)$ with strictly convex entropy $\Phi$, then a locally Lipschitz continuous solution is the unique weak solution of \eqref{eq:conservation law} stemming from its initial condition. The cutoff implosion solution we propose solves the isentropic Euler equations which can be expressed in the form of \eqref{eq:conservation law},
\begin{equation}\label{eq:isen_euler_gen}
    \begin{split}
        \rho_t + \nabla \cdot (\rho \vec{v}) &= 0, \\
        (\rho \vec{v})_t + \nabla \cdot (\rho \vec{v} \otimes \vec{v} + C\rho^{\gamma}I)&= 0.
    \end{split}
\end{equation}
\par Even if we assume our solution is piecewise differentiable and thus locally Lipschitz, we cannot directly apply the result due to the unboundedness of the solution as $t \to 1$ and the total energy not being strictly convex on the whole domain. In particular, the total energy loses strict convexity at the origin where the density remains zero for all $t < 1$. We work around this issue by applying the result to a neighborhood around the cutoff point. Specifically, for any fixed time $T < 1$, on the shell $[a, b]$ where $0 < a < r^* < b$, the density is bounded from below by a constant depending on $T$ but strictly greater than zero. This avoids any issues with loss of strict convexity of the associated entropy, and we may directly apply the theorem thereby stating the constructed solution is the unique weak solution up to $T < 1$ on $[a,b]$. While this discussion is not a rigorous treatment of uniqueness, it does suggest our constructed solution to the cutoff implosion problem is the unique solution stemming from the cutoff implosion initial data.

\section{Numerical Simulations of the PDE}\label{sec:numerics}
\begin{figure}[pos=h]
     \centering
     \begin{subfigure}[h]{0.49\linewidth}
         \centering
         \includegraphics[width=\textwidth]{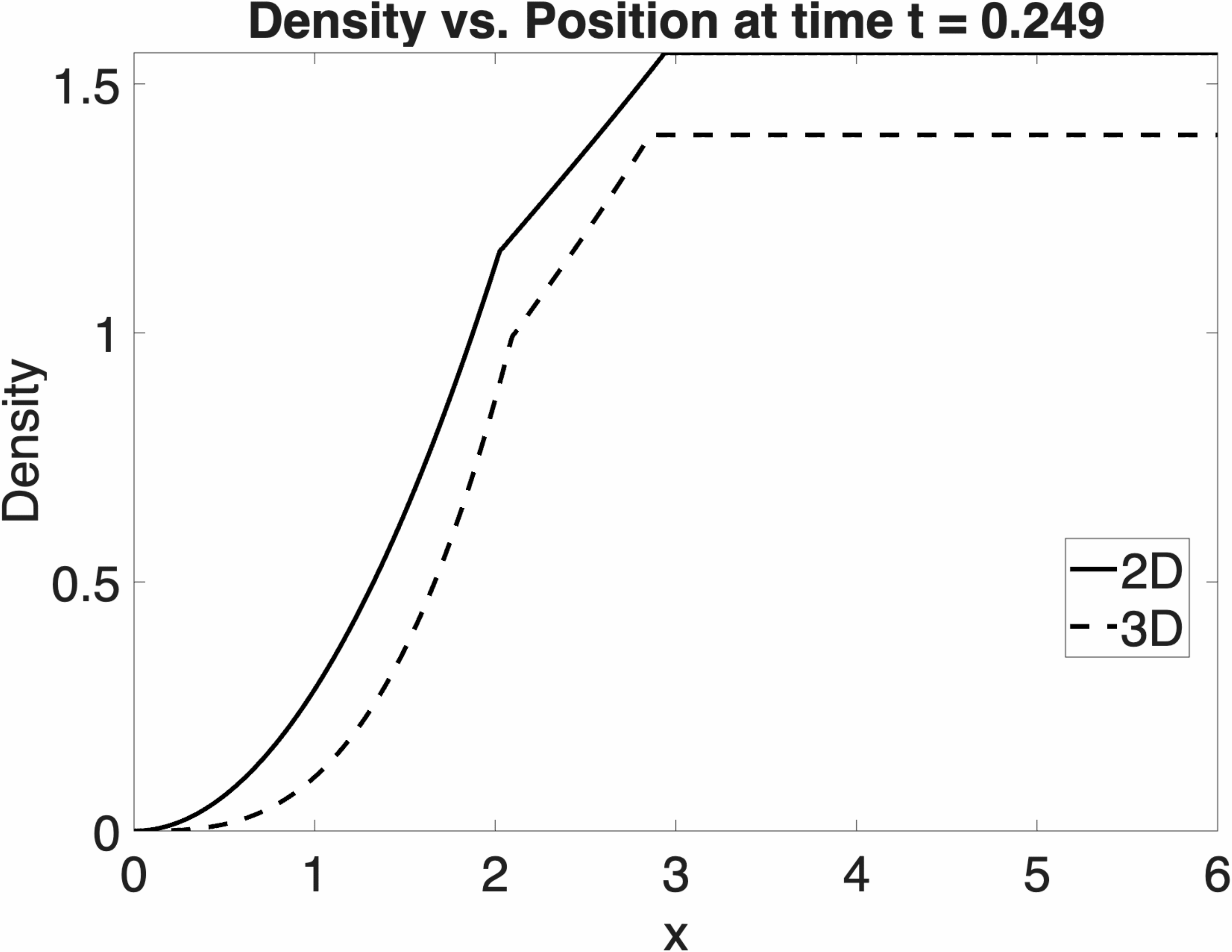}
         \caption{$t \approx 0.25$}
         \label{subfig:density_0.25}
     \end{subfigure}
     \hfill
     \begin{subfigure}[h]{0.49\linewidth}
         \centering
         \includegraphics[width=\textwidth]{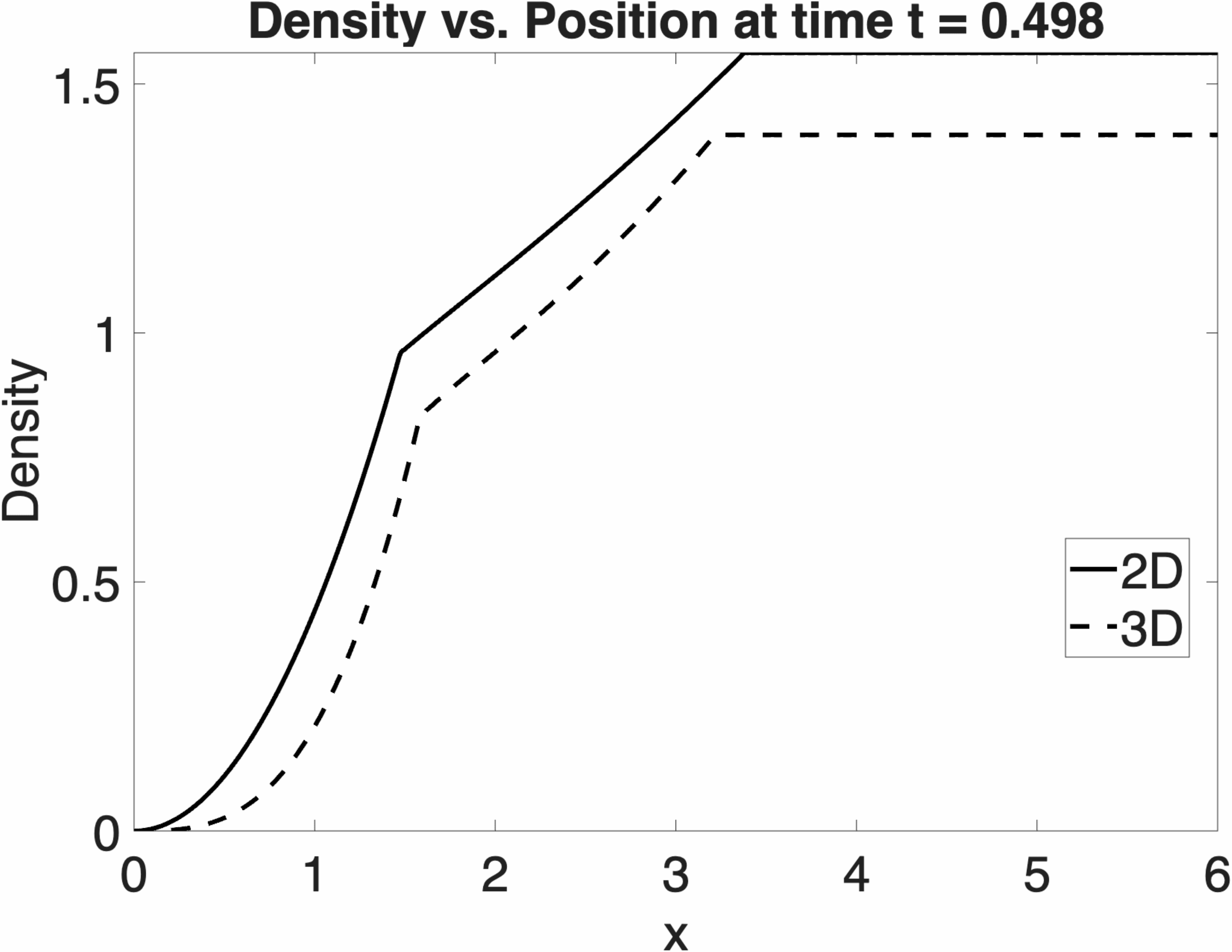}
         \caption{$t \approx 0.5$}
         \label{subfig:density_0.5}
    \end{subfigure}
    \vskip\baselineskip
    \begin{subfigure}[h]{0.49\linewidth}
         \centering
         \includegraphics[width=\textwidth]{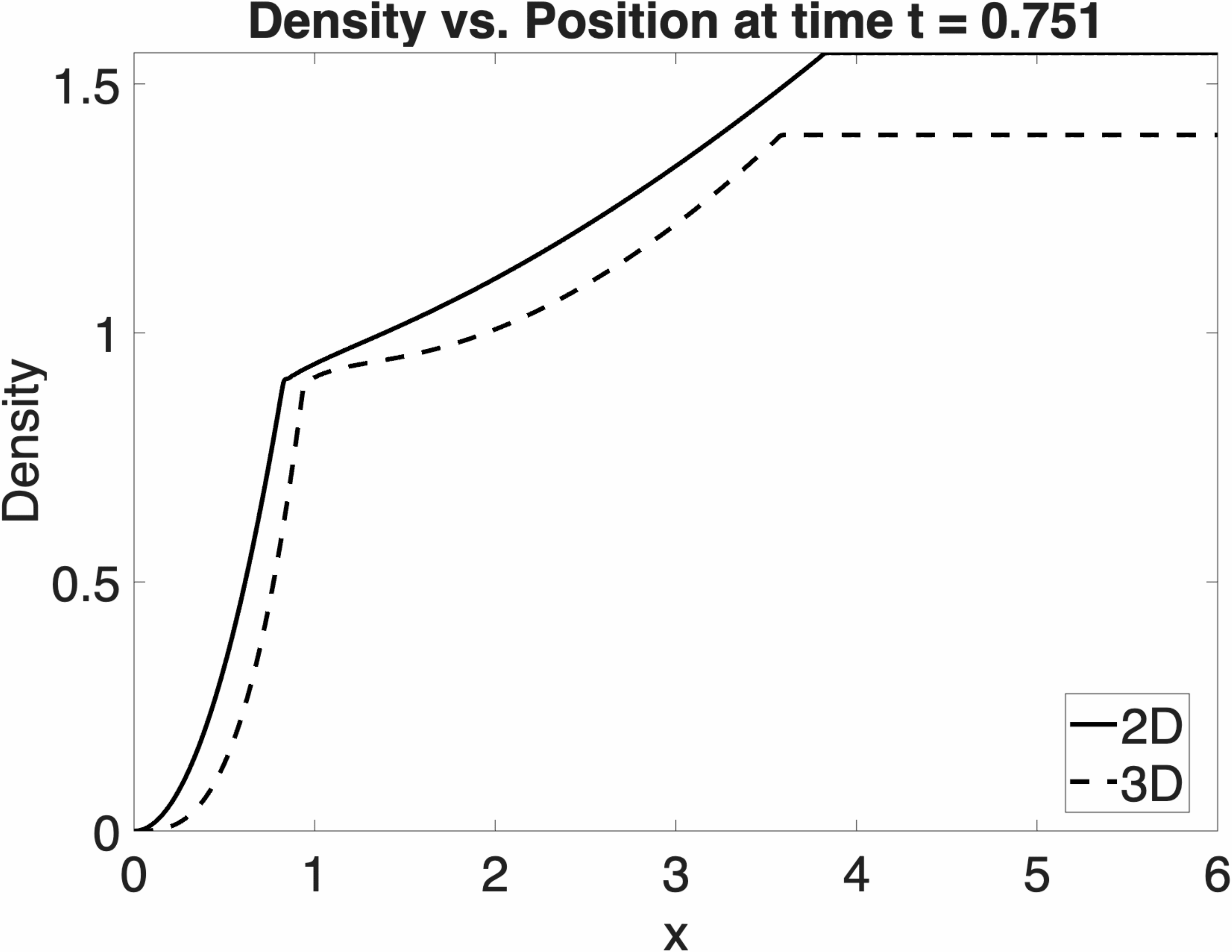}
         \caption{$t \approx 0.75$}
         \label{subfig:density_0.75}
     \end{subfigure}
     \hfill
     \begin{subfigure}[h]{0.49\linewidth}
         \centering
         \includegraphics[width=\textwidth]{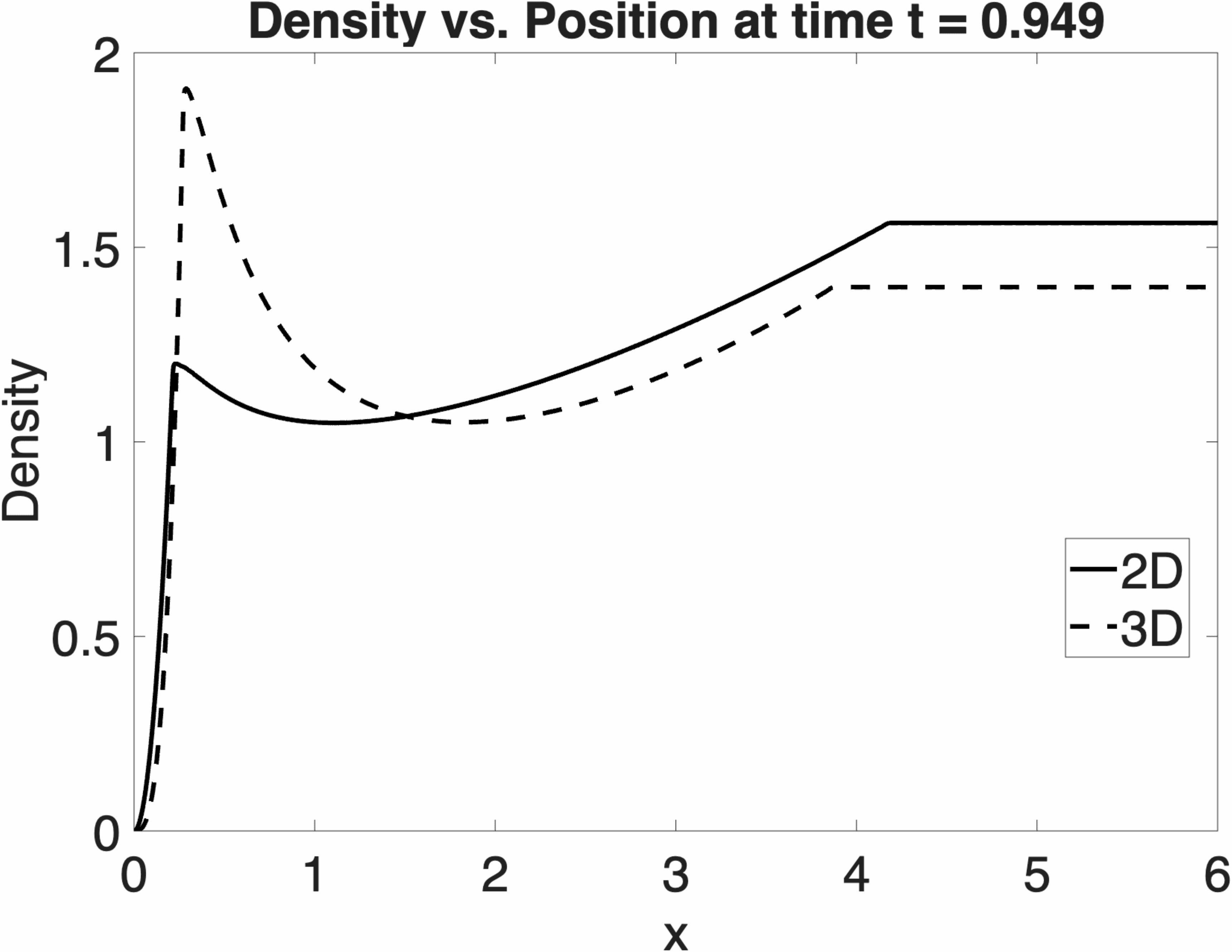}
         \caption{$t \approx 0.95$}
         \label{subfig:density_0.95}
     \end{subfigure}
        \caption{Computed density to the potential flow equations \eqref{eq:isen_mass} and \eqref{eq:isen_mom} with cutoff implosion initial data \eqref{eq:cutoff_imploding_condition} and $r^*$ = 2.5. Approximate times $t = 0.25$, $0.5$, $0.75$ and $t = 0.95$ in $2$D (solid) and $3$D (dashed) illustrating the evolution of the solution are displayed.}
        \label{fig:density}
\end{figure}
\begin{figure}[pos=h]
     \centering
     \begin{subfigure}[h]{0.49\linewidth}
         \centering
         \includegraphics[width=\textwidth]{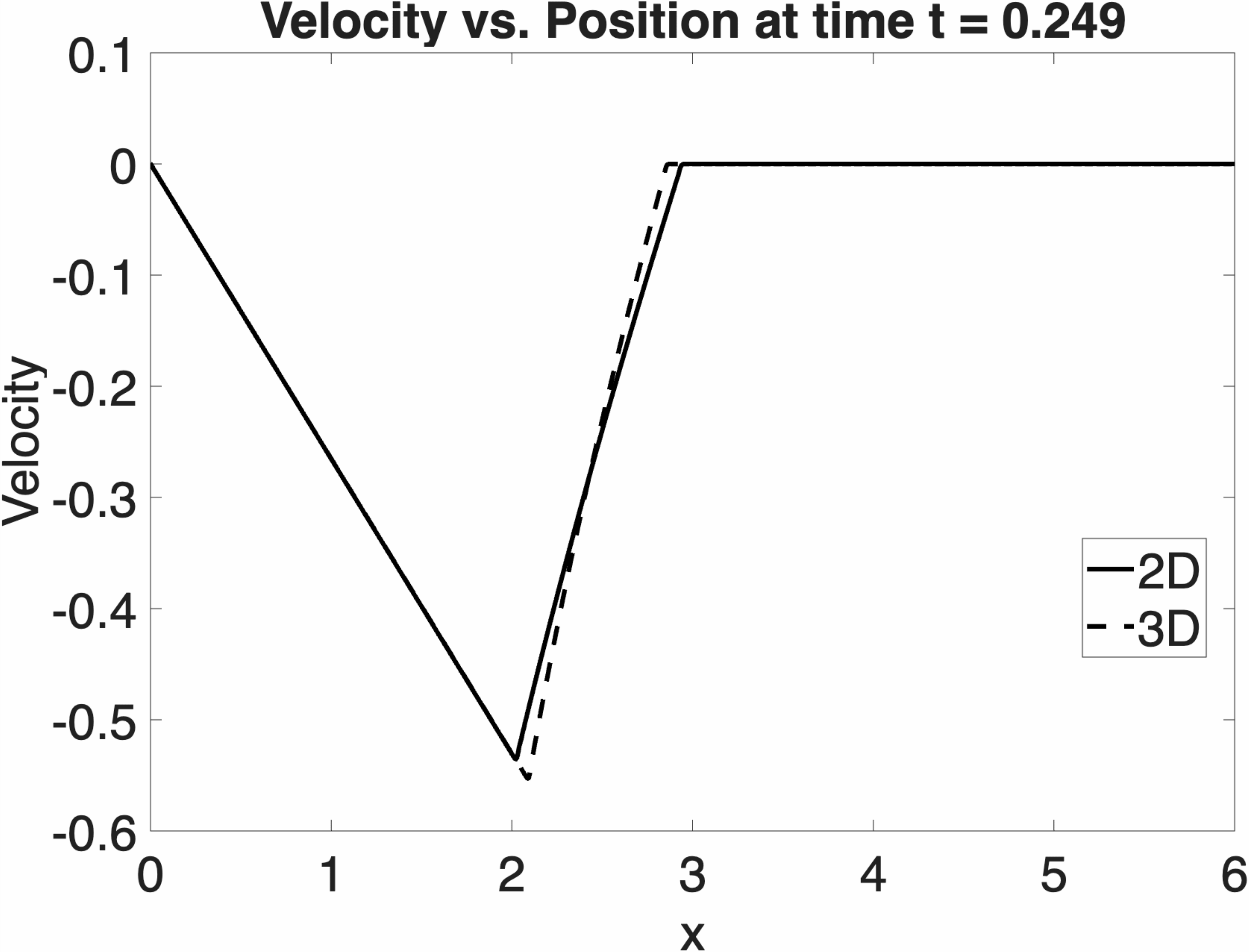}
         \caption{$t \approx 0.25$}
         \label{subfig:velocity_0.25}
     \end{subfigure}
     \hfill
     \begin{subfigure}[h]{0.49\linewidth}
         \centering
         \includegraphics[width=\textwidth]{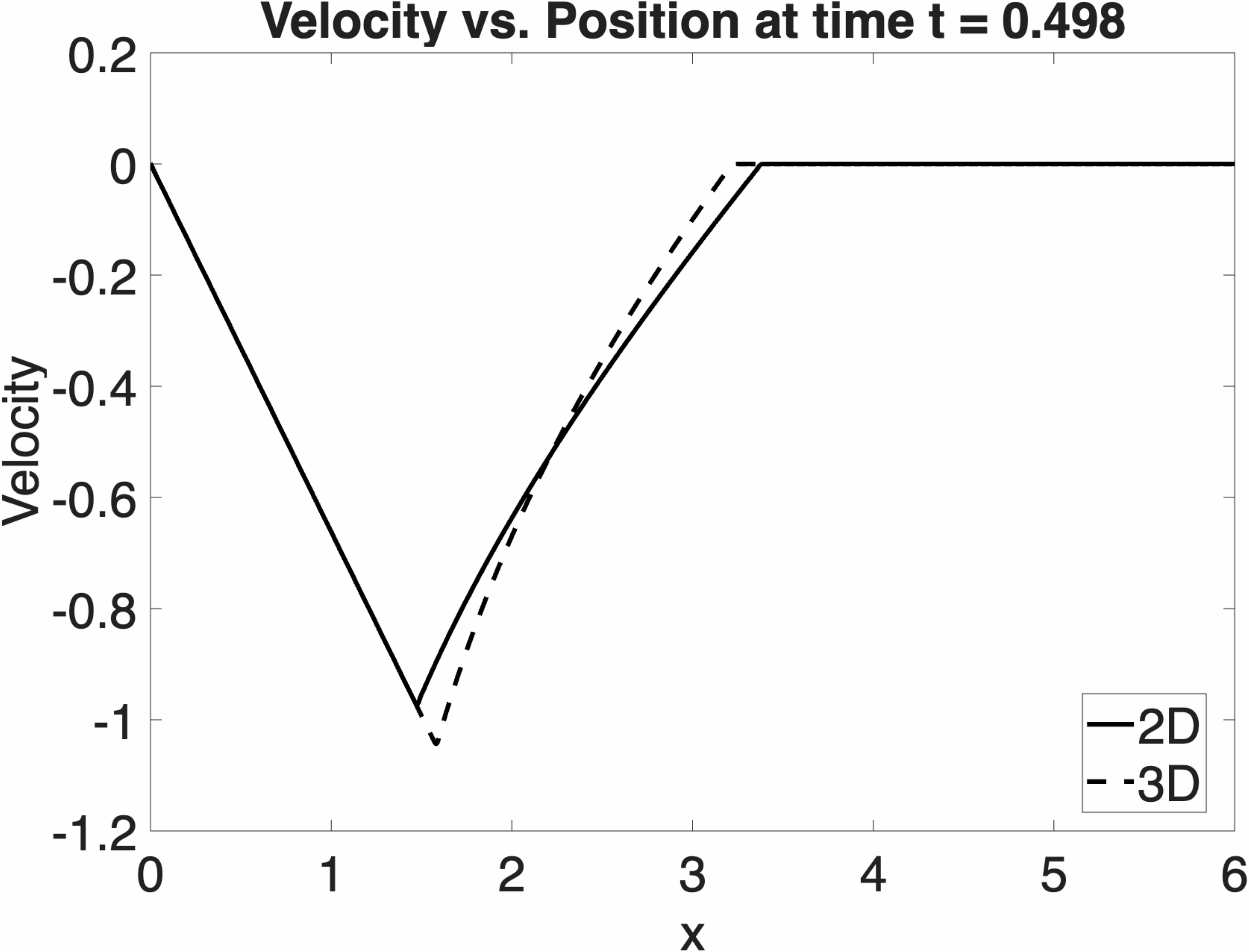}
         \caption{$t \approx 0.5$}
         \label{subfig:velocity_0.5}
    \end{subfigure}
    \vskip\baselineskip
    \begin{subfigure}[h]{0.49\linewidth}
         \centering
         \includegraphics[width=\textwidth]{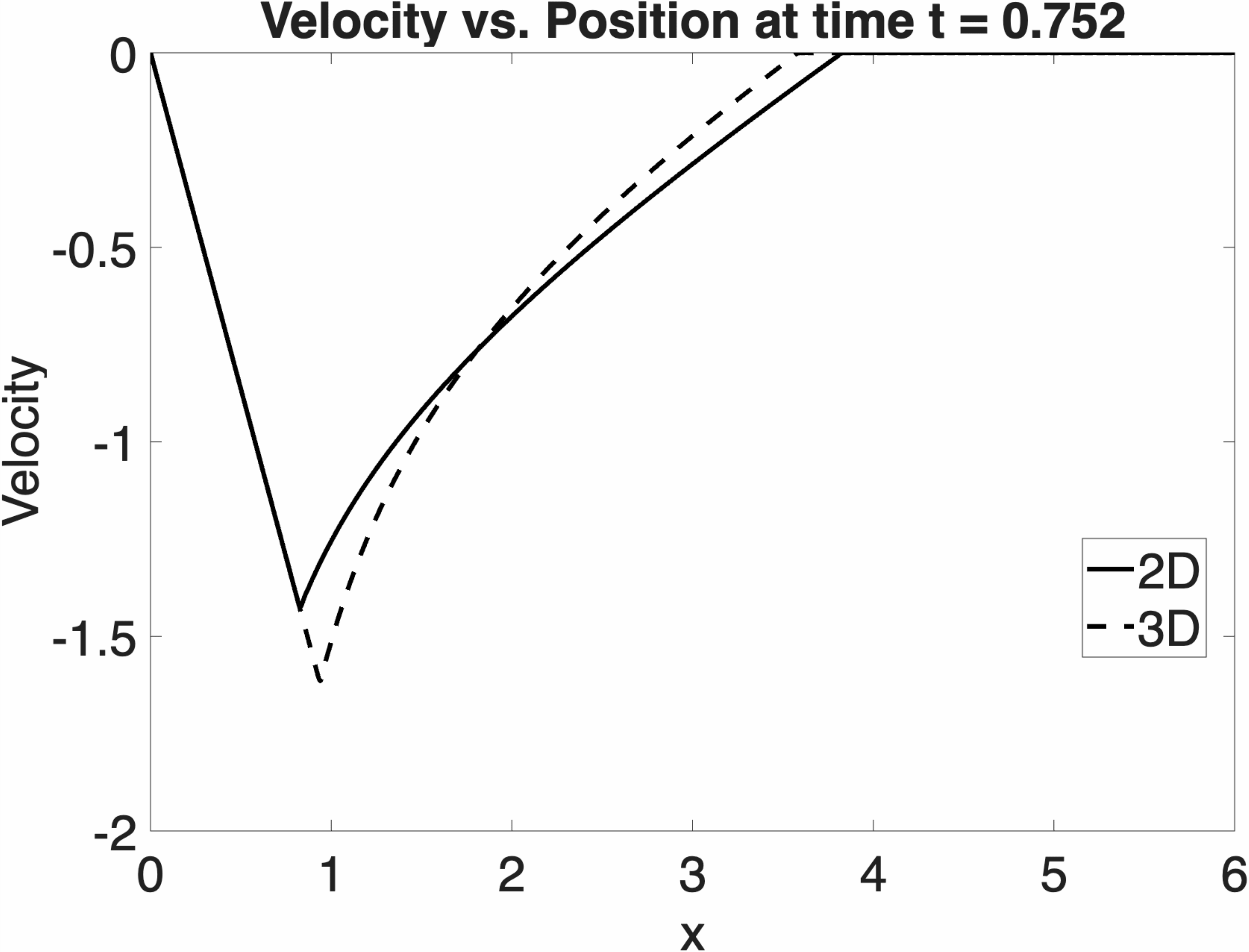}
         \caption{$t \approx 0.75$}
         \label{subfig:velocity_0.75}
     \end{subfigure}
     \hfill
     \begin{subfigure}[h]{0.49\linewidth}
         \centering
         \includegraphics[width=\textwidth]{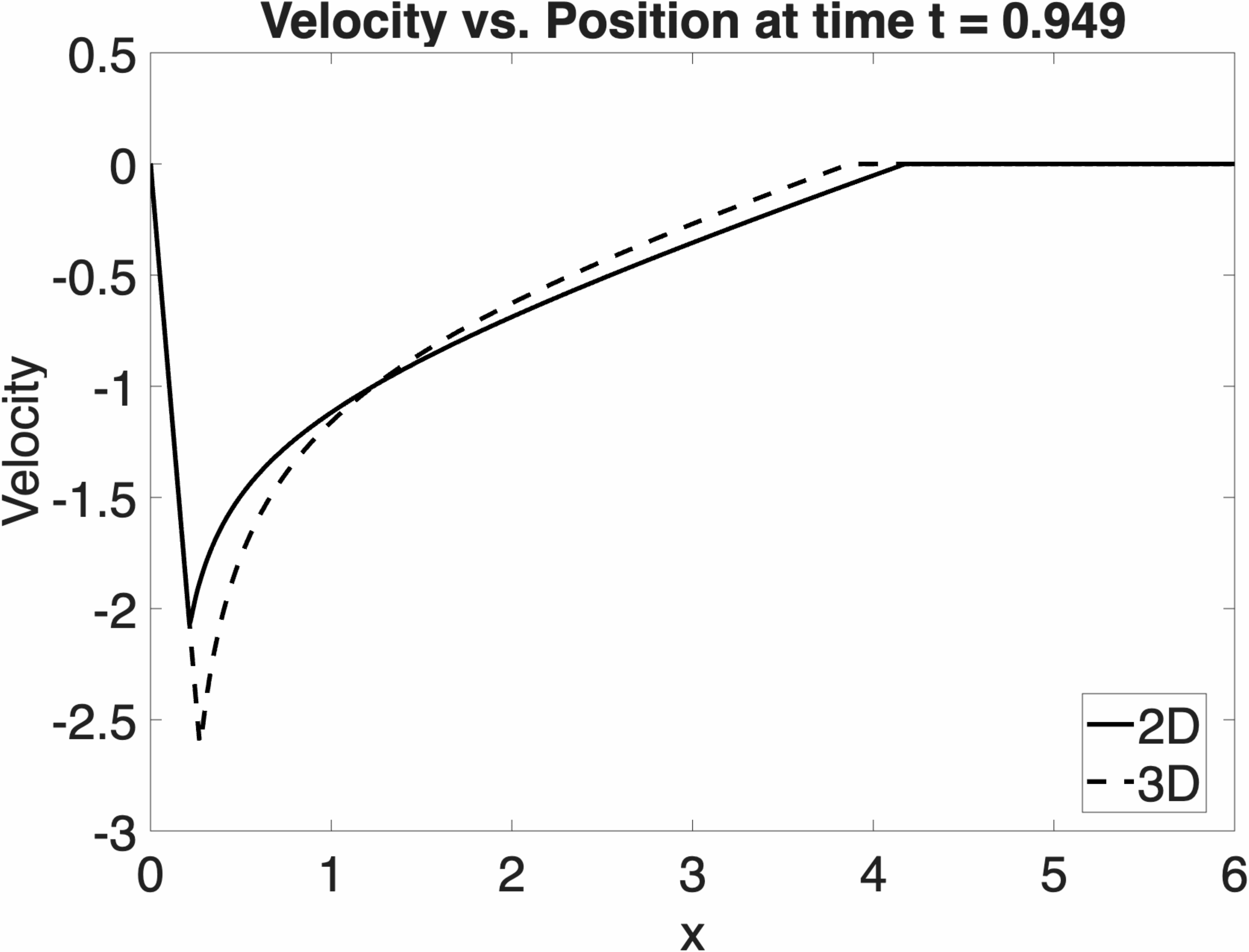}
         \caption{$t \approx 0.95$}
         \label{subfig:velocity_0.95}
     \end{subfigure}
        \caption{Computed velocity to the potential flow equations \eqref{eq:isen_mass} and \eqref{eq:isen_mom} with cutoff implosion initial data \eqref{eq:cutoff_imploding_condition} and $r^*$ = 2.5. Approximate times $t = 0.25$, $0.5$, $0.75$ and $t = 0.95$ in $2$D (solid) and $3$D (dashed) illustrating the evolution of the solution are displayed.}
        \label{fig:velocity}
\end{figure}
We support the analysis done in the previous sections with numerical simulations. Our goal is to verify the formulas for the inner and outer radii along with the imploding density and velocity at the inner radius. Once the simulation progresses close to implosion time, we determine whether the solution approaches a self-similar asymptotic limit. Additionally, we verify the numerical solution satisfies the predicted time-scaling behavior. Finally, the $L^p$ norm of density for various $p$ is computed from the numerical solutions.

\subsection{Lagrangian Leapfrog}
\par The numerical simulation of the potential flow equations \eqref{eq:isen_mass} and \eqref{eq:isen_mom} is performed with a Lagrangian leapfrog scheme as introduced in \cite{VonNeumannArtificialViscosity} and further refined in \cite{NohArtificialViscosity}. In this Lagrangian scheme, cells are defined on the intervals between nodes. The nodes are interspersed throughout the spatial domain and kinematic variables like position and velocity are defined on these nodes. The numerical scheme tracks the evolution of the cells by advecting the nodes. Thermodynamic variables such as the density and pressure are defined to each cell. Historically, this scheme was introduced to better resolve shocks via artificial viscosity as elaborated in \cite{VonNeumannArtificialViscosity} and \cite{NohArtificialViscosity}. We do not predict any shocks in our analysis, so we set the artificial viscosity to zero. 

\subsection{Adaptive Mesh Refinement}
\par As the simulation approaches the implosion time, we perform an adaptive mesh refinement near the implosion point $r_l(t)$ to capture the self-similar profile without losing resolution. Without adaptive mesh refinement, the interval $[0,r_l(t)]$ will shrink in time, and the number of cells within this region will decrease as well. We use a simple criteria and procedure: whenever $r_l(t)$ decreases by a set factor, we increase the amount of nodes in between the origin and $r_l(t)$. 
\par Adaptive mesh refinement allows for high resolution of the implosion with only logarithmic scaling in the number of necessary nodes, especially for a fixed point singularity in one spatial variable (in this case $r$). We add additional mesh points only the spatial interval of the Kidder solution. Although our goal is to resolve the structure to the right of $r_l(t)$, increasing resolution between the origin and $r_l(t)$ still accomplishes this because every particle originating in $(0,r^*)$ will eventually fall to the right of $r_l(t)$ at the time prescribed in \eqref{eq:time_cross}.

\par By choosing to add nodes on the interval $(0,r_l(t))$ corresponding to the Kidder solution, we assign density and velocity values to all nodes using the exact formula for the Kidder solution. Here we leverage the closed form solution on this region. 

\par Adaptive mesh refinement begins after $\rho(r_l(t),t)$ starts to increase in time. Given that $v(r_l(t),t)$ is expected to be a global minimum for velocity in space, we track $r_l(t)$ in the numerical simulation as $r_l(t) = \argmin_{r} v(r,t)$. It is vital to design a way of identifying $r_l(t)$ without directly using the simulated clock time due to the well-known  accumulation small errors every time the time step is added to the clock time.

\subsection{Discussion of Finite Difference Scheme}
\par It is critical to choose a suitable numerical scheme for computing finite time singularities at a small scale.  Most notably, any Eulerian scheme has effective diffusion or dispersion as can be seen by modified equation analysis \cite{Leveque}. Moreover the equation has no diffusive scale; it is purely advective at all scales.  Thus we need a numerical method that does not introduce artificial diffusion or dispersion. Otherwise we can not properly resolve the singularity as it forms at increasingly smaller scales.   We choose a Lagrangian leapfrog scheme over other methods as this scheme is not diffusive. Generally speaking, Lagrangian schemes avoid this problem.  Another example where a Lagrangian scheme is used to compute high resolution numerics of a singularity is in \cite{Yao}, in which a Lagrangian scheme was used to solve the aggregation part of an aggregation diffusion equation.  More discussion on the numerical diffusion of Eulerian methods in hydrodynamics can be found in the introduction of \cite{compuational_euler}. 

\par On the other hand, the Lagrangian leapfrog method of \cite{VonNeumannArtificialViscosity} is dispersive, which means nonphysical oscillations appear in its numerical solution when trying to resolve shocks. Since shocks are not predicted to appear in the solution, we do not expect numerical solutions generated by the Lagrangian leapfrog scheme to suffer from nonphysical oscillations. Thus, we choose the non-diffusive scheme of \cite{VonNeumannArtificialViscosity} to best preserve the fine details of the solution near implosion time.

\par To begin, the initial data \eqref{eq:cutoff_imploding_condition} is discretized. Nodes are assigned to the numerical domain of $[0,2r^*]$ with $N + 1$ nodes such that
\begin{equation*}
    0 = r^1_1 < r^1_2 < \cdots r^1_{N+1} = 2r^*,
\end{equation*}
where $r^n_k$ is the position of the $k$th node at time $t^n$. The right boundary of the numerical domain is chosen to be $2r^*$ so $r_r(t)$ given in equation \eqref{eq:r_r} does not reach the right boundary by $t = 1$. The initial velocity $v^{1/2}_k$ is also defined at each node. Note the velocity is defined on a staggered time grid from the other variables. In between each node lies a cell, in which the density at time $t^n$ between nodes $r^n_k$ and $r^n_{k + 1}$ is defined as $\rho^n_{k + 1/2}$. To assign the initial data, we associate the spatial point $r^n_{k + 1/2} = \frac{r^n_k + r^n_{k+1}}{2}$ to each cell and take $\rho^1_{k + 1/2} = \rho\left(r^1_{k + 1/2},0\right)$, where $\rho(r,0)$ is defined in \eqref{eq:cutoff_imploding_condition}. After the initial density is assigned, the mass $m_{k+1/2}$ can be assigned to each cell and is set to be constant in time. The mass at each node for $k > 1$ is defined as $m_k = \frac{1}{2} \left(m_{k+1/2} + m_{k-1/2} \right)$. Initial geometric quantities, such as the surface area of spheres with radii prescribed the nodes and the volume of the cells, are computed. The initial time step is defined such that sound cannot cross a cell in a single cycle, which requires
\begin{equation}\label{eq:timestep}
\Delta t^n = \chi_C \min_{1 \leq k \leq N} \frac{r^n_{k + 1} - r^n_{k}}{c^n_{k + 1/2}},
\end{equation}
where $c^n_k$ is the speed of sound \eqref{eq:sos} and $\chi_C = 0.5$ is selected to be the Courant Number. 

\begin{figure}[pos=h]
    \centering
    \includegraphics[width=0.95\linewidth]{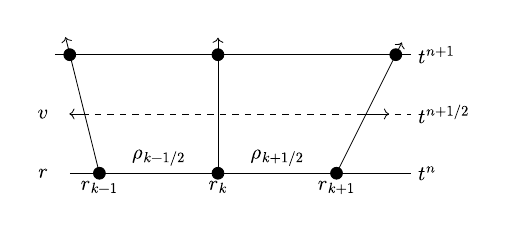}
    \caption{The stencil for the Lagrangian leapfrog method. At integer time levels, the position is defined at the nodes (filled circles) while the density is defined on the cells in between the nodes. The velocities are defined on nodes at fractional time levels and are represented as arrows. In Lagrangian numerical methods, the nodes move with the computed velocity.}
    \label{fig:stencil}
\end{figure}

\par The Lagrangian leapfrog method begins by advancing the velocity of each node. The update is formed by discretizing the Lagrangian formulation of \eqref{eq:isen_mom},
\begin{equation}\label{eq:isen_mom_lagrange}
    \rho \frac{Dv}{Dt} = -\nabla p,
\end{equation}
where $\frac{D}{Dt} = \frac{d}{dt} + v \frac{d}{dr}$ is the material derivative. The discretization of \eqref{eq:isen_mom_lagrange} becomes
\begin{equation*}
    v^{n + 1/2}_k = v^{n - 1/2}_k - \Delta t^n \frac{A^n_k}{m_k} \left(P^n_{k +1/2} - P^n_{k - 1/2} \right),
\end{equation*}
where $A^n_k$ is the surface area of the corresponding dimensional sphere with radius $r^n_k$ and $P^n_{k +1/2}$ is the pressure. In the isentropic case, pressure is related to density via \eqref{eq:pressure} and is discretized as $P^n_{k +1/2} = \left(\rho^n_{k +1/2}\right)^\gamma$. Next, each node advects with the new velocity by
\begin{equation*}
    r^{n+1}_k = r^n_k + \Delta t^n u_k^{n + 1/2}.
\end{equation*}
Geometric quantities are recomputed using the new node positions. To compute the new density, we use the fact that the mass is assumed to be conserved for each cell and define
\begin{equation*}
    \rho^{n+1}_{k + 1/2} = \frac{m_{k + 1/2}}{V^{n+1}_{k + 1/2}},
\end{equation*}
where $V^{n+1}_{k + 1/2}$ is the volume of the spherical shell bounded by spheres with outer radius $r^{n+1}_{k + 1}$ and inner radius $r^{n+1}_{k}$. Finally, the time step is recalculated by equation \eqref{eq:timestep} at time index $n + 1$. A stencil describing the location of variables in relation to the nodes and cells is shown in Figure \ref{fig:stencil}.

\begin{figure}[pos=h]
    \centering
    \includegraphics[width=0.95\linewidth]{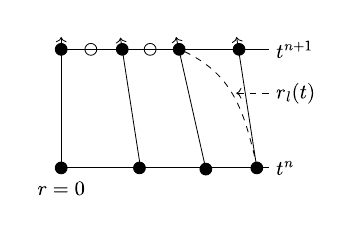}
    \caption{A schematic showing how adaptive mesh refinement is performed. Nodes - the filled circles - advect from time $t^n$ to $t^{n+1}$. The dashed line shows the movement of $r_l(t)$. When adaptive mesh refinement occurs, all of the cells in between $0$ and $r_l(t)$ are split. A new node - indicated by the unfilled circles - is added as the midpoint of the old nodes.}
    \label{fig:amr}
\end{figure}

\par Adaptive mesh refinement first occurs when $\rho(r_l(t),t)$ starts to become a relative maximum. Once this occurs, $r_l$ is identified numerically and saved as $r_l = r^{n+1/2}_i$, where $i =  \argmin_{1 \leq k \leq N+1} v^{n + 1/2}_k$. At every subsequent timestep after the Lagrangian leapfrog steps are finished but before the next time step is computed, $r_l$ is calculated. If the current $r_l$ is a user set fraction of the previously saved $r_l$, then adaptive mesh refinement occurs and the new $r_l$ is saved. When adaptive mesh refinement occurs, all of the cells up to node $i$ are halved such that a new node is placed at the midpoint of every old node, given by $r^{\text{new}}_{2k} = \frac{1}{2}(r^{\text{old}}_{k} + r^{\text{old}}_{k + 1} )$ where ``new'' refers to the nodes after mesh refinement and ``old'' refers to the nodes before mesh refinement. Figure \ref{fig:amr} shows an example of the mesh refinement process. Data is assigned to all of the new nodes by assuming the Kidder solution \eqref{eq:original_imploding_solution} applies at these nodes. The velocity on the new mesh is defined as
\begin{equation}\label{eq:amr_velocity}
    v^{\text{new}}_k = v^{\text{old}}_i r^{\text{new}}_k,
\end{equation}
and then density on the new mesh is defined as
\begin{equation}\label{eq:amr_density}
    \rho^{\text{new}}_{k + 1/2} = \rho^{\text{old}}_{i - 1/2} \left( \frac{r^{\text{new}}_{k + 1/2}}{ r^{\text{old}}_{i - 1/2} } \right)^{n+1}.
\end{equation}
The reassignments in \eqref{eq:amr_velocity} and \eqref{eq:amr_density} are constructed directly from the data without using the time. This is done to avoid any propagation of errors associated with computing the time especially close to implosion time. After the new mesh is defined, the mass is recomputed on the new cells, along with geometric quantities.

\subsection{Simulation Parameters}

\par We simulate the potential flow equations \eqref{eq:isen_mass} and \eqref{eq:isen_mom} with initial condition \eqref{eq:cutoff_imploding_condition} for both $n = 1$ and $n = 2$, corresponding to cylindrical and spherical symmetric geometry respectively. The numerical domain is $[0,5]$ with initially $3001$ nodes and $3000$ cells. Instead of halting the simulation at a set time, we end the simulation once the maximum density has achieved a set value: $10^5$ for $n = 1$ and $10^{10}$ for $n = 2$. The boundary condition of our simulations are $v = 0$ at $r = 0$ and $v = 0$ at $r = 5$, with the right boundary condition justified as the gas has constant density and zero velocity in this region left unperturbed. We select the cutoff position in the initial data to be $r^* = 2.5$. Adaptive mesh refinement begins when $\rho(r_l(t),t)$ starts to increase in time. Empirically, we find this to be around $t = 0.9$ and start adaptive mesh refinement after $t = 0.9$. Successive refinement occurs whenever $r_l(t)$ decreases by a factor of three, and the mesh is refined by splitting every cell in between $[0, r_l(t)]$ into two new cells. 

\par Figures \ref{fig:density} and \ref{fig:velocity} show the evolution of the density and velocity respectively at various times up until $t = 0.95$ to give a general sense of how the solution evolves at times far from $t = 1$. The percent error between predicted values at $r_l(t)$ and computed values at times leading up to $t = 0.9$, before adaptive mesh refinement occurs, is displayed in Table \ref{tab:err}. To better understand the solution behavior near $t = 1$, Figures \ref{fig:2D_blowup} and \ref{fig:3D_blowup} show the density and velocity profiles in the similarity variable $y$ near the blowup in $2$D and $3$D respectively. The time-scaling between $r_l(t)$ and extrema of the density and velocity is shown in Figure \ref{fig:scaling}. The $L^p$ norm on $[0,r]$ for various choices of $p$ are displayed in Figures \ref{fig:mass}, \ref{fig:pgamma}, \ref{fig:p_crit}, and \ref{fig:p_crit_over}.




\begin{table}[pos=h]
    \centering
    \begin{tabular}{|c|c|c|c|}
    \hline
    Feature & Min & Max & Mean \\
    \hline
    $r_l(t)$ (2D) & $-0.24\%$ & $0.50\%$ & $0.08\%$ \\
    \hline
    $r_l(t)$ (3D) & $-0.24\%$ & $0.10\%$ & $-0.05\%$ \\
    \hline
    $\rho(r_l(t),t)$ (2D) & $-0.47\%$ & $0.05\%$ & $-0.13\%$ \\
    \hline
    $\rho(r_l(t),t)$ (3D) & $-0.89\%$ & $-0.16\%$ & $-0.41\%$ \\
    \hline
    $v(r_l(t),t)$ (2D) & $-0.46\%$ & $-0.06\%$ & $-0.28\%$ \\
    \hline
    $v(r_l(t),t)$ (3D) & $-0.62\%$ & $-0.06\%$ & $-0.35\%$ \\
    \hline
    $m(0,r_l(t))$ (2D) & $-0.94\%$ & $2.2\%$ & $0.39\%$ \\
    \hline
    $m(0,r_l(t))$ (3D) & $-1.40\%$ & $0.84\%$ & $-0.19\%$ \\
    \hline
    \end{tabular} 
    \caption{Percent error between the analytical and computed $r_l(t), \rho(r_l(t),t), v(r_l(t),t)$ and $ m(0,r_l(t))$ from $t \in (0, 0.9]$. The percent error formula is $\text{\% Error}_x = 100 \times \frac{x_{\text{comp}} - x_{\text{pred}}}{x_{\text{pred}}}$ where $x_{\text{comp}}$ is the computed value and $x_{\text{pred}}$ is the predicted value. The minimum, maximum, and mean percent errors are displayed.}
    \label{tab:err}
\end{table}

\begin{figure}[pos=h]
    \centering
    \includegraphics[width=0.95\linewidth]{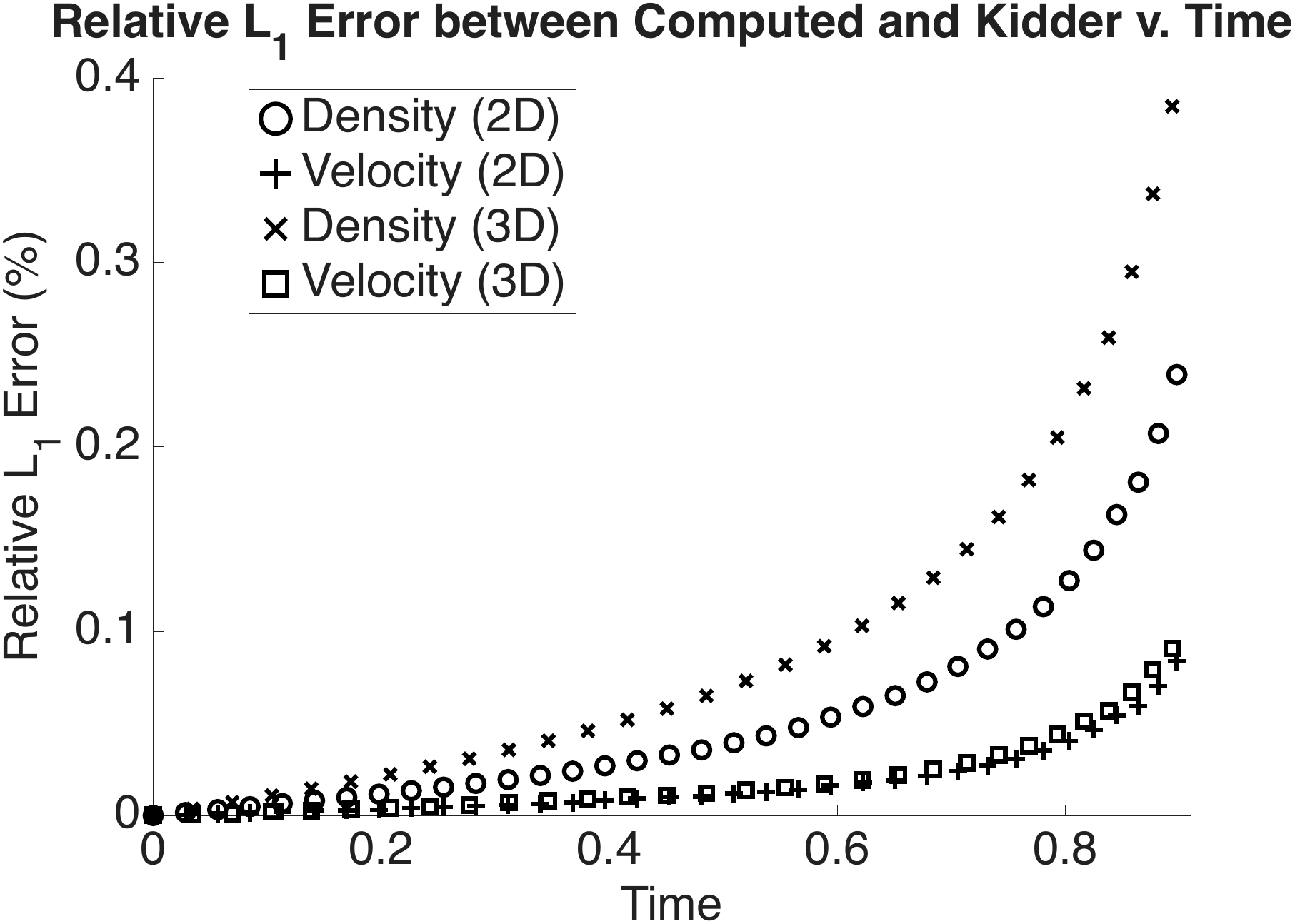}
    \caption{Percent $L^1$ error between the Kidder solution \eqref{eq:original_imploding_solution} computed by \eqref{eq:L1_norm} and the computed solution, taken on the interval $[0, r_l(t)]$ for $0 < t \leq 0.9$.}
    \label{fig:kid_err}
\end{figure}

\subsection{Early Time Dynamics}
\par We compare our analytical expressions for $r_l(t)$ \eqref{eq:r_l}, $\rho(r_l(t),t)$ \eqref{eq:rho_r_l}, $v(r_l(t),t)$ \eqref{eq:v_r_l}, and $ \Vert \rho(r,t)\Vert_{L^1\left([0,r_l(t)]\right)}$ \eqref{eq:mass_formula} to numerical data on times far away from the implosion time. These comparisons are made before adaptive mesh refinement is done in order to verify our predictions before assumptions on the solution are made in the numerical scheme. The percent errors on all of these quantities, displayed in Table \ref{tab:err}, are low and bounded in a $2.5\%$ range. The notable outlier is the mass in the implosion region, which has a slightly higher percent error range than the other quantities. We interpret this as the result of the mass in the imploding region being a quantity that is computed from the directly simulated variables of $\rho$ and $r_l(t)$. The directly simulated variables $\rho,v$, and $r_l(t)$ all enjoy much lower percent error, being bounded by $1\%$ error. The low percent error indicates the expressions for $r_l(t), \rho(r_l(t),t), v(r_l(t),t)$, and $ m(0,r_l(t))$ are correct. 

\par Additionally, we compute the percent $L^1$ error, in space between the Kidder solution and the computed solution on the interval $[0, r_l(t)]$. The percent $L^1$ error is computed as
\begin{equation}\label{eq:L1_norm}
\begin{split}
    \text{\% Error} &= 100 \times \frac{\left\| y^A - y^C \right\|_1}{\left\| y^A \right\|_1} \\
    &= 100 \times \frac{\frac{1}{V} \sum_{i = 1}^{N} \left| y^A(x_i) - y^C_i \right| V_i}{\frac{1}{V} \sum_{i = 1}^{N} \left| y^A_i \right| V_i},
\end{split}
\end{equation}
where the $A$ and $C$ superscripts represent the analytical and computed solutions respectively, $N$ is the number of nodes or volumes, $x_i$ is the location of the $i$th node or volume, $V_i$ is the cell volume, and $V$ is the sum of the cell volumes. In Figure \ref{fig:kid_err}, the $L_1$ error between the Kidder solution and computed solution is displayed for density and velocity in $2$D and $3$D. We observe from Figure \ref{fig:kid_err} the $L_1$ error monotonically increases as time progresses. It is expected for errors to accumulate in time as the simulation progresses. Despite the increasing time, the percent $L_1$ errors are bounded by $0.5\%$. Both the low percent error of features presented in Table \ref{tab:err} and the low percent $L_1$ error between the Kidder solution and computed solution justify the hypothesis the solution remains the Kidder solution on the interval $[0, r_l(t)]$. As such, this justifies our adaptive mesh refinement as the simulation progresses to implosion time.


\subsection{Near Implosion Time}
\par The adaptive mesh refinement allows the numerical scheme to simulate times close to $t = 1$ corresponding to large values of $\rho(r_l(t))$. In the close to implosion regime, we predict the solution should adhere to the dynamics of the ODE system in \eqref{eq:ss_system}. By plotting the numerical data with respect to the similarity variable and normalizing the quantities, we see in Figures \ref{fig:2D_blowup} and \ref{fig:3D_blowup} for 2D and 3D respectively. Note the plotted velocity is normalized by $\min_y v(y) < 0$ to be positive. In all of the comparison plots, the data coincides with each other on the interval $y = [0,1]$ and in the log-log scale as the numerical solution is the Kidder solution on this interval. The displayed times are computed from the maximum density in the simulation and \eqref{eq:rho_r_l}.

\begin{figure*}[pos=h]
     \centering
     \begin{subfigure}[h]{0.49\linewidth}
         \centering
         \includegraphics[width=\textwidth]{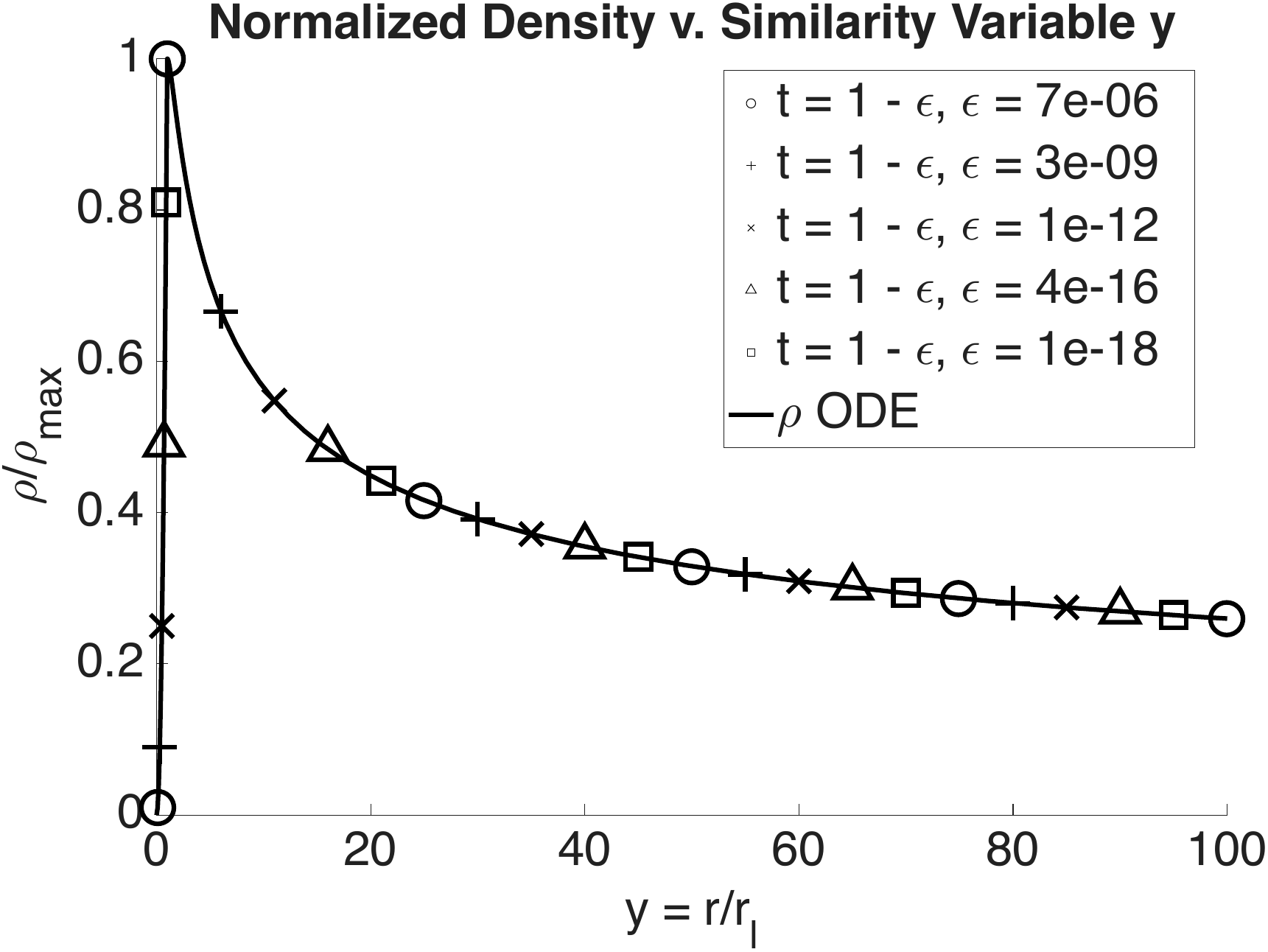}
         \caption{}
         \label{subfig:blowup_density_2D}
     \end{subfigure}
     \hfill
     \begin{subfigure}[h]{0.49\linewidth}
         \centering
         \includegraphics[width=\textwidth]{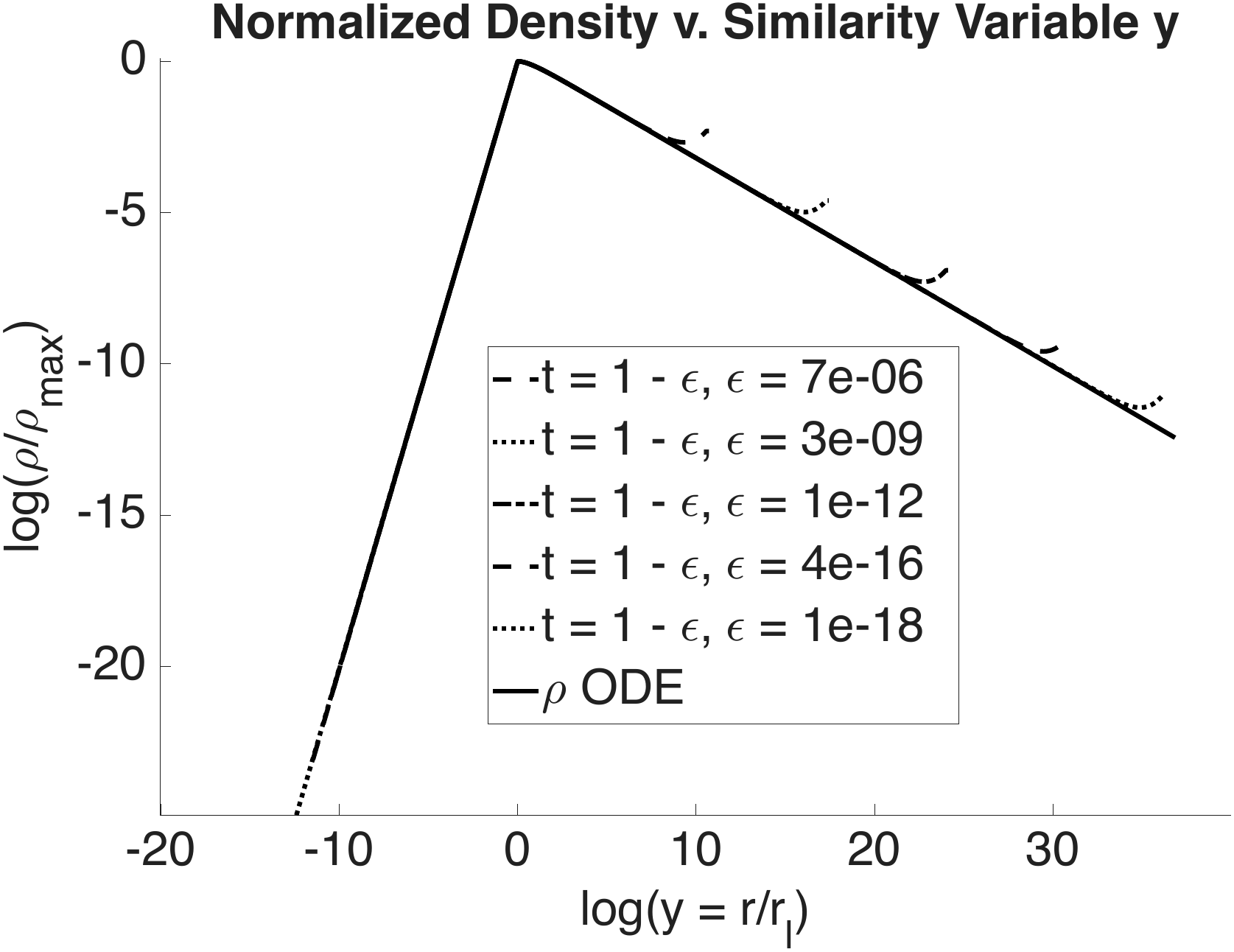}
         \caption{}
         \label{subfig:blowup_density_log_2D}
    \end{subfigure}
    \vskip\baselineskip
    \begin{subfigure}[h]{0.49\linewidth}
         \centering
         \includegraphics[width=\textwidth]{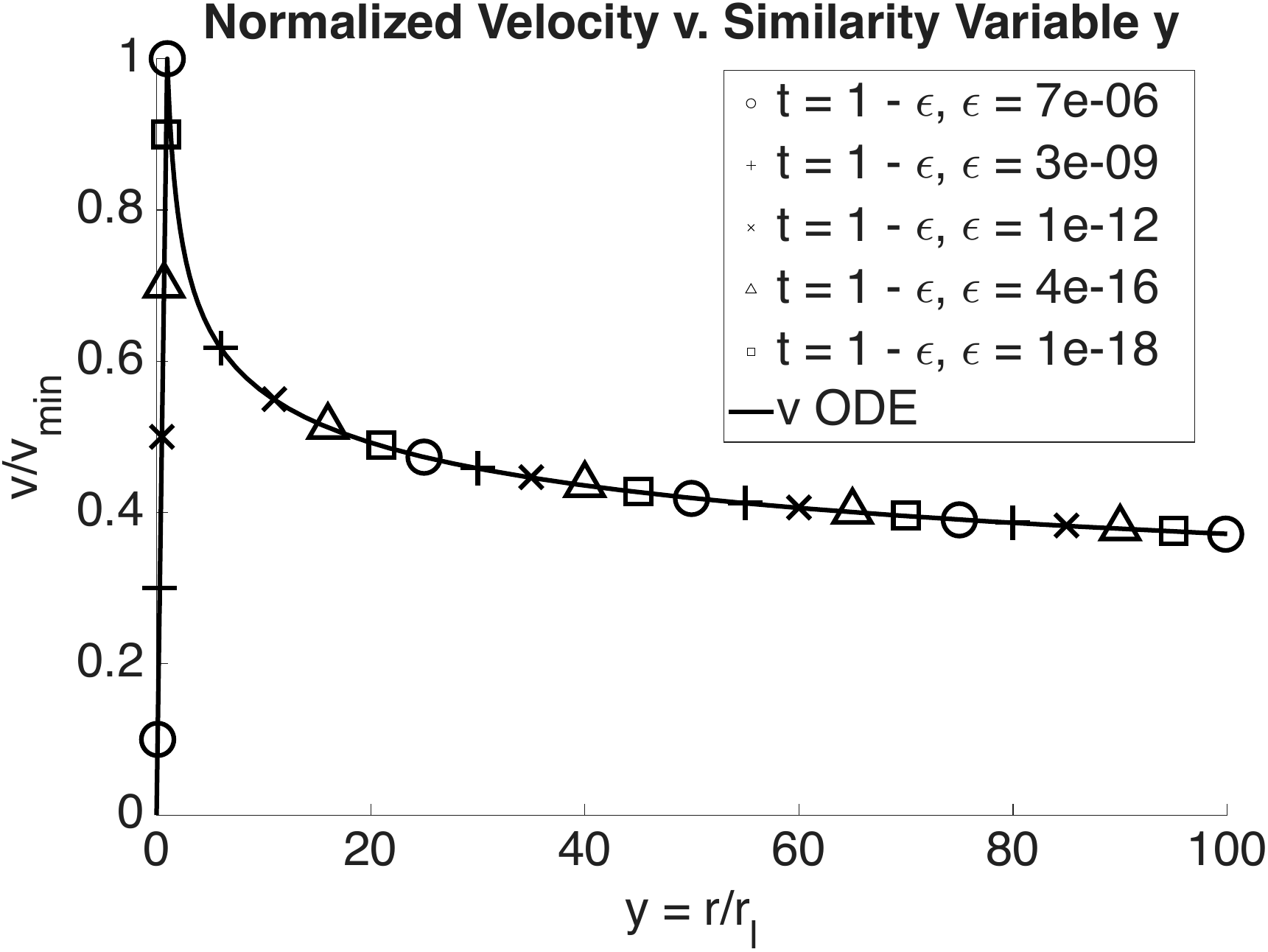}
         \caption{}
         \label{subfig:blowup_velocity_2D}
     \end{subfigure}
     \hfill
     \begin{subfigure}[h]{0.49\linewidth}
         \centering
         \includegraphics[width=\textwidth]{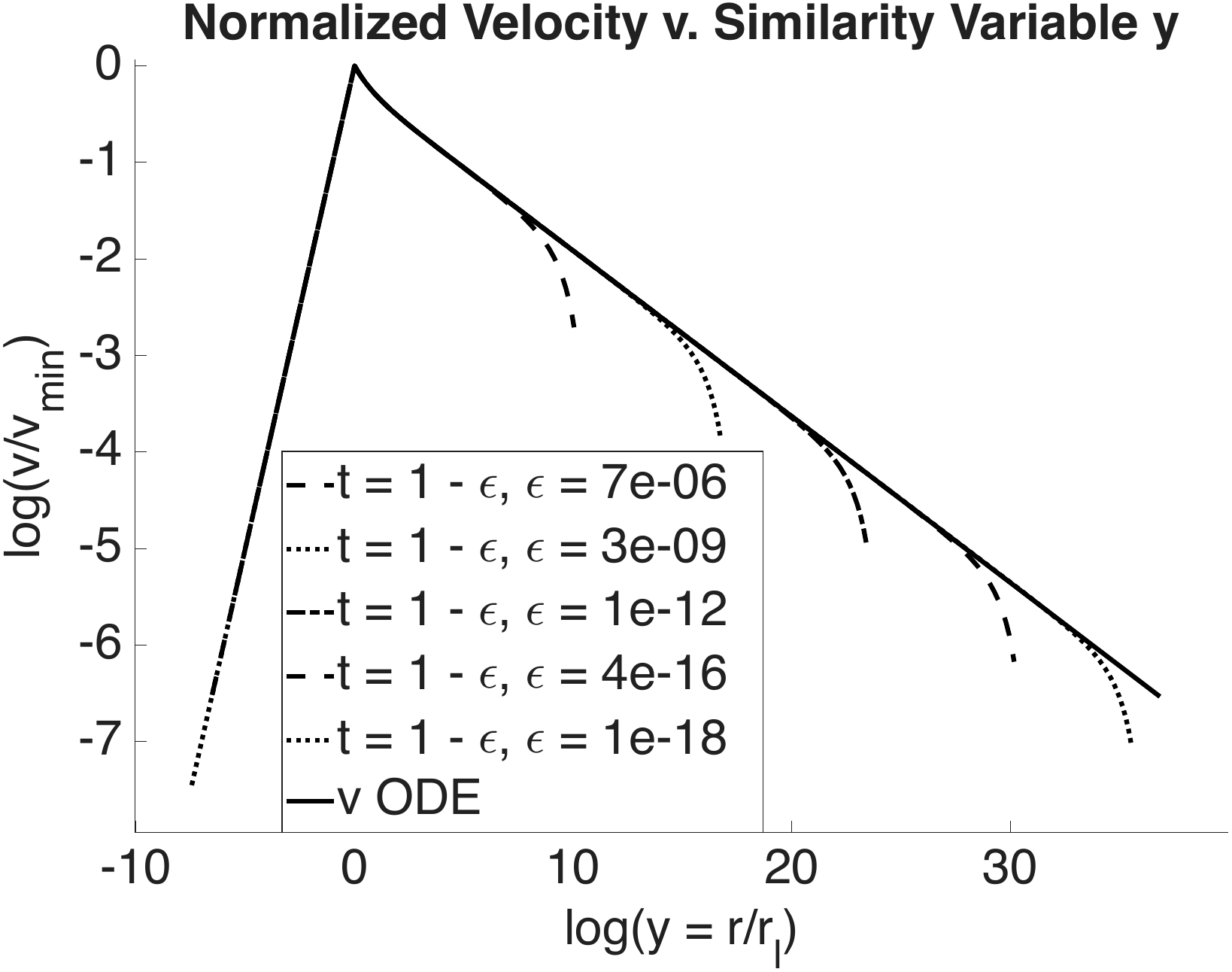}
         \caption{}
         \label{subfig:blowup_velocity_log_2D}
     \end{subfigure}
        \caption{Blowup profiles in 2D with respect to the similarity variable $y = \frac{r}{r_l(t)}$ near $t = 1$ for density in subfigures \ref{subfig:blowup_density_2D} and for velocity in \ref{subfig:blowup_velocity_2D} and in a log-log scale in subfigures \ref{subfig:blowup_density_log_2D} for density and \ref{subfig:blowup_velocity_log_2D} for velocity. The profiles are compared to $\rhohat$ and $\vhat$ solving ODE system \eqref{eq:ss_system}. The selected times represent the solution at $\max{\rho(r,t)} = 10$, $10^2$, $10^3$, $10^4$, $10^5$ respectively.}
        \label{fig:2D_blowup}
\end{figure*}

\begin{figure*}[pos=h]
     \centering
     \begin{subfigure}[h]{0.49\linewidth}
         \centering
         \includegraphics[width=\textwidth]{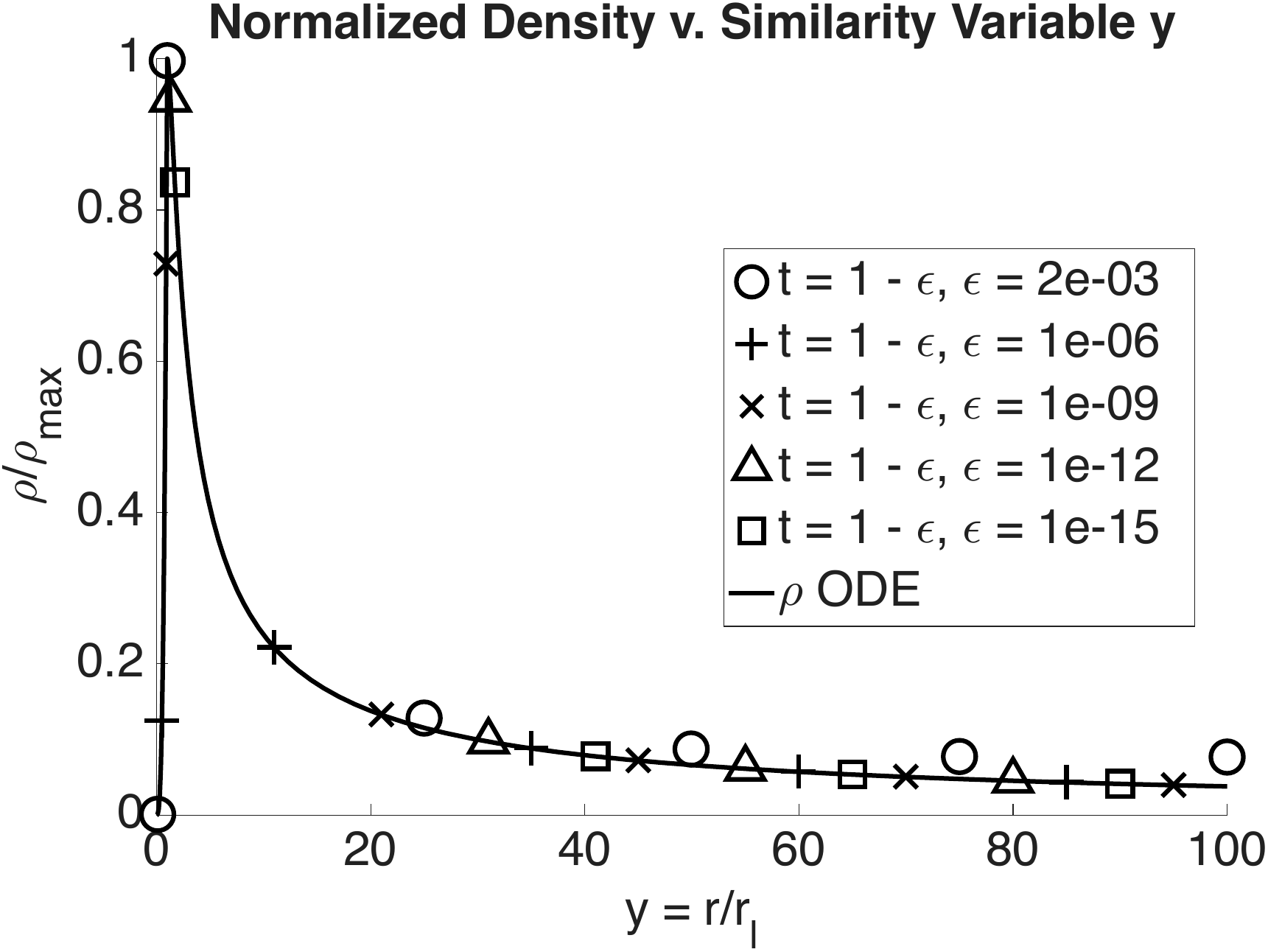}
         \caption{}
         \label{subfig:blowup_density_3D}
     \end{subfigure}
     \hfill
     \begin{subfigure}[h]{0.49\linewidth}
         \centering
         \includegraphics[width=\textwidth]{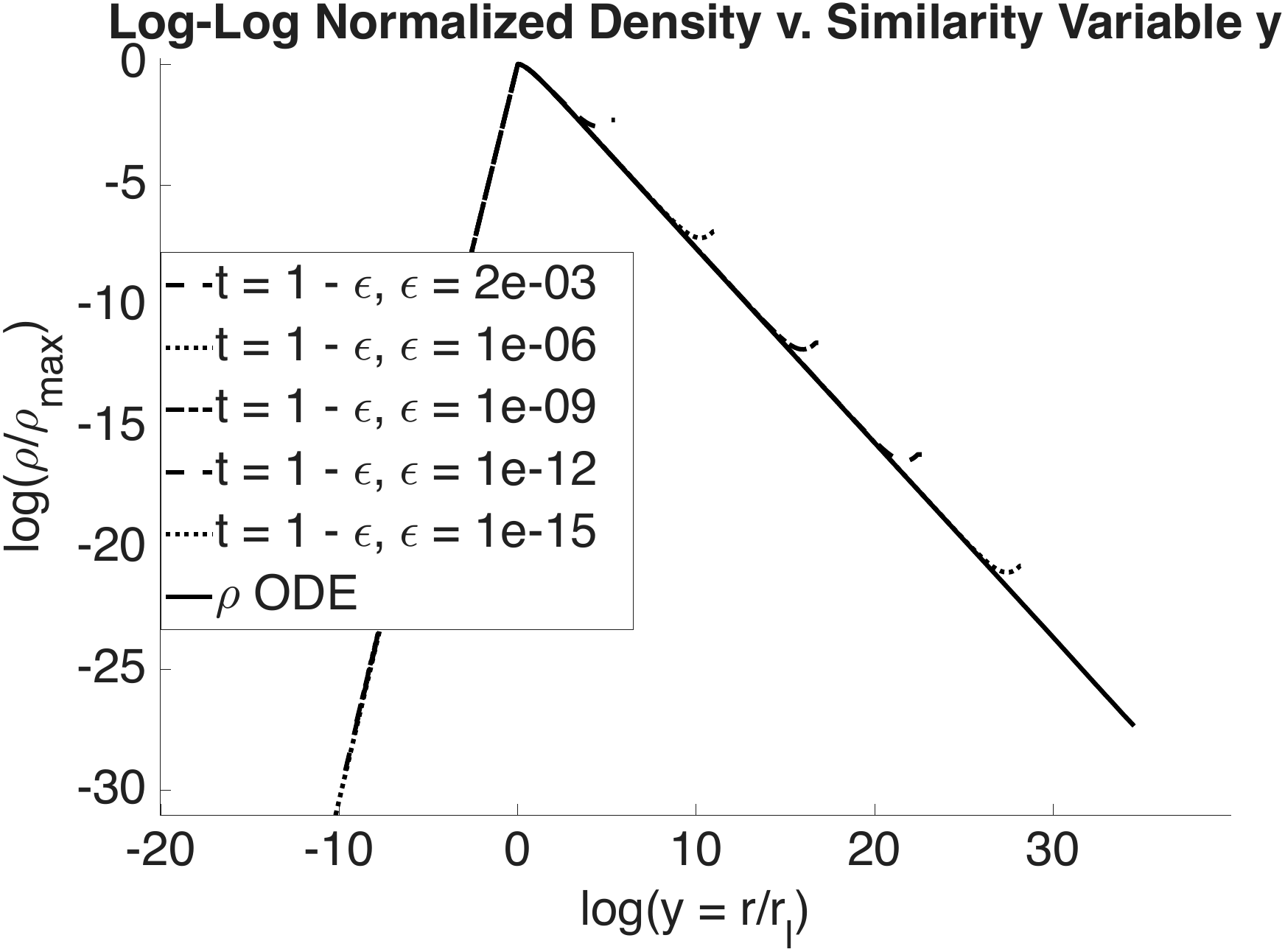}
         \caption{}
         \label{subfig:blowup_density_log_3D}
    \end{subfigure}
    \vskip\baselineskip
    \begin{subfigure}[h]{0.49\linewidth}
         \centering
         \includegraphics[width=\textwidth]{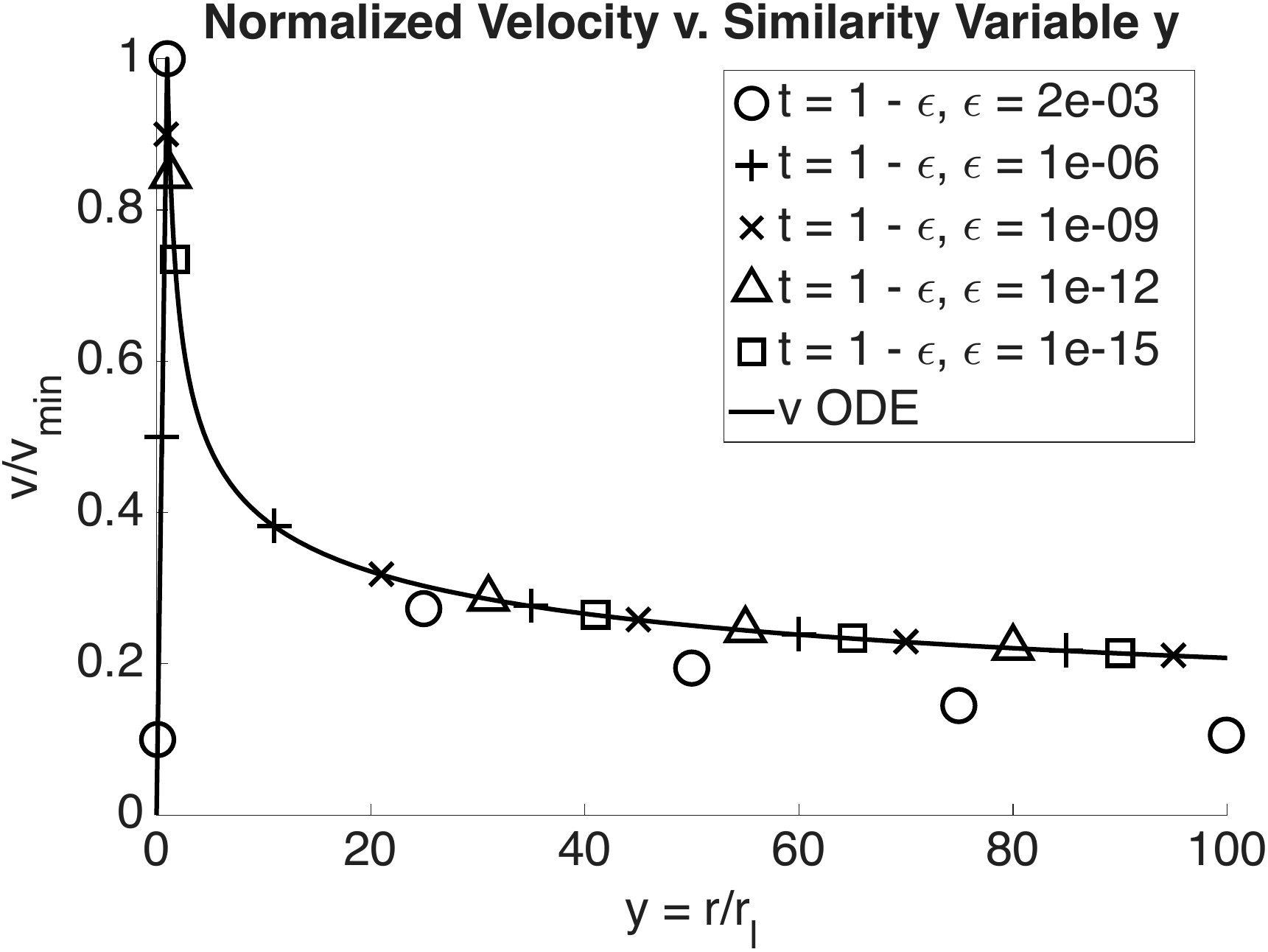}
         \caption{}
         \label{subfig:blowup_velocity_3D}
     \end{subfigure}
     \hfill
     \begin{subfigure}[h]{0.49\linewidth}
         \centering
         \includegraphics[width=\textwidth]{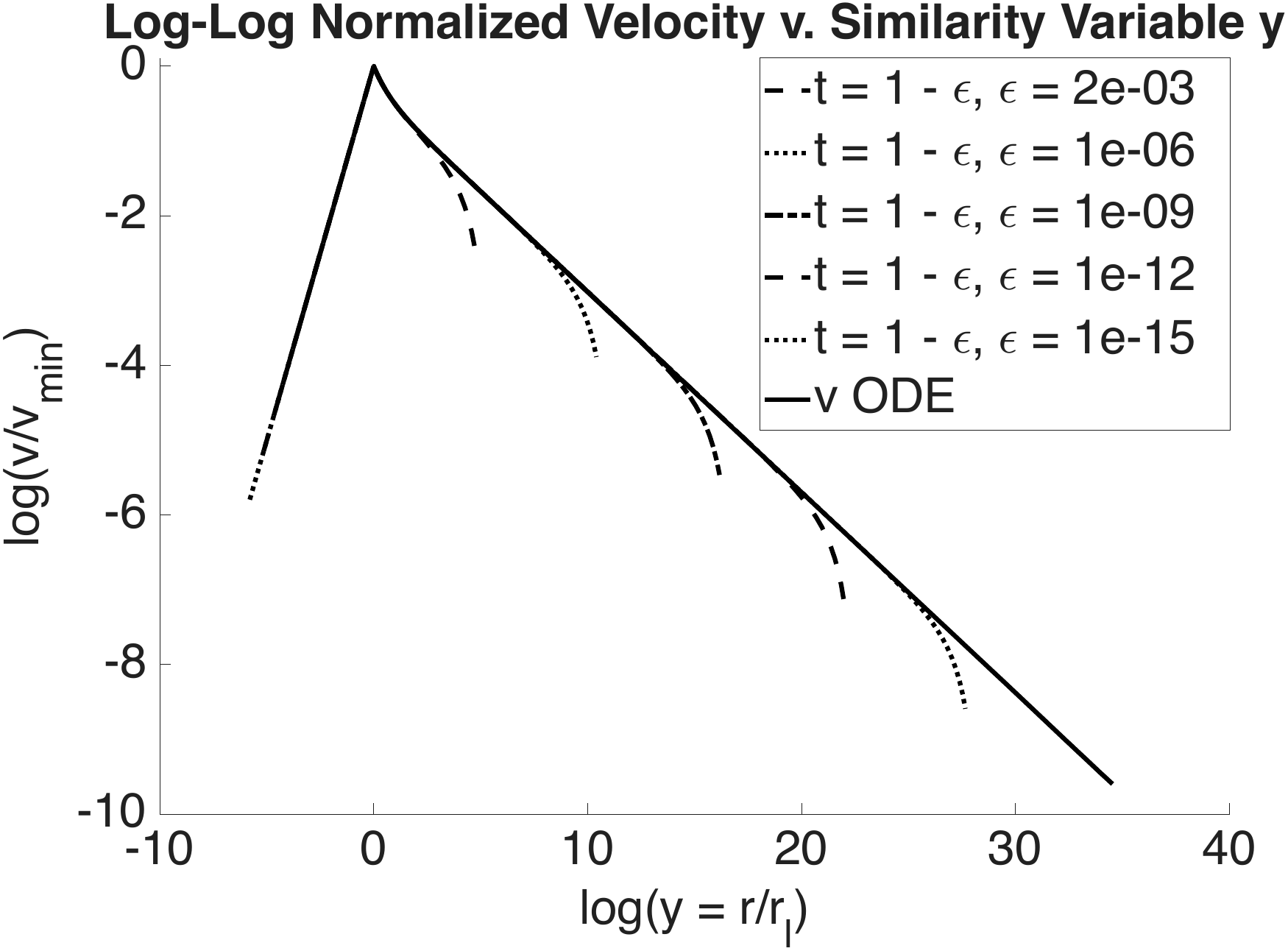}
         \caption{}
         \label{subfig:blowup_velocity_log_3D}
     \end{subfigure}
        \caption{Blowup profiles in 3D with respect to the similarity variable $y = \frac{r}{r_l(t)}$ near $t = 1$ for density in subfigures \ref{subfig:blowup_density_3D} and for velocity in \ref{subfig:blowup_velocity_3D} and in a log-log scale in subfigures \ref{subfig:blowup_density_log_3D} and \ref{subfig:blowup_velocity_log_3D}. The profiles are compared to $\rhohat$ and $\vhat$ solving ODE system \eqref{eq:ss_system}. The selected times represent the solution at $\max{\rho(r,t)} = 10$, $10^3$, $10^5$, $10^7$, $10^9$ respectively.}
        \label{fig:3D_blowup}
\end{figure*} 

\begin{table}[pos=h]
\begin{tabular}{|c|c|c|c|}
\hline
$n$ & $\max{\rho(r,t)}$ & \% Density Error & \% Velocity Error \\ \hline
$1$ & $10^1$            & $0.43 \%$         & $0.17\%$          \\ \hline
$1$ & $10^2$            & $0.35\%$         & $0.11\%$          \\ \hline
$1$ & $10^3$            & $0.30\%$         & $0.09\%$          \\ \hline
$1$ & $10^4$            & $0.27\%$         & $0.07\%$          \\ \hline
$1$ & $10^5$            & $0.26\%$         & $0.06\%$          \\ \hline
$2$ & $10^1$            & $0.98\%$         & $0.31\%$          \\ \hline
$2$ & $10^3$            & $0.12\%$         & $0.18\%$          \\ \hline
$2$ & $10^5$            & $0.08\%$         & $0.13\%$          \\ \hline
$2$ & $10^7$            & $0.06\%$         & $0.11\%$          \\ \hline
$2$ & $10^9$            & $0.05\%$         & $0.12\%$          \\ \hline
\end{tabular}
\caption{Percent $L^1$ error \eqref{eq:L1_norm} between computed $\rho$ and $v$ near implosion time to self-similar profiles $\rhohat$ and $\vhat$. The self-similar profiles $\rhohat$ and $\vhat$ are solutions to \eqref{eq:ss_system} and are shown in Figure \ref{fig:hat}. Percent error is computed with respect to the similarity variable $y$. The interval of consideration is $r_l(t) \leq r \leq 5r_l(t)$ or equivalently $1 \leq y \leq 5$.}
\label{tab:ss_error}
\end{table}

\par Subfigures \ref{subfig:blowup_density_2D} and \ref{subfig:blowup_density_3D} show how the data coincides with $\rhohat$ when plotted against $y$ for density and similarly for $\vhat$ in \ref{subfig:blowup_velocity_2D} and \ref{subfig:blowup_velocity_3D}. The agreement is quantified in Table \ref{tab:ss_error} where the percent $L^1$ error between computed and normalized density and velocity are compared to $\rhohat$ and $\vhat$ which solve ODE system \eqref{eq:ss_system}. The percent $L^1$ error is less than a percent for both density and velocity for all times. Furthermore, as time gets closer to implosion time, the percent $L^1$ error decreases indicating convergence of numerical data to the $\rhohat$ and $\vhat$ profiles. The higher percent $L^1$ errors in density and velocity in the case of $n = 2$, $\max{\rho(r,t)} = 10$ relative to the other errors indicate the solution is still far away from the asymptotic limit. Given the profiles $\rhohat$ and $\vhat$ only explain the solution dynamics near $r_l(t)$, we do not expect the numerically computed solutions to adhere to $\rhohat$ and $\vhat$ for all $y$ and especially for $y$ corresponding to $r = r_r(t)$. This also justifies the choice of interval to compute the percent $L^1$ error to be a small neighborhood around $r_l(t)$: $[r_l(t),5r_l(t)]$. 

\par The adherence and departure from $\rhohat$ and $\vhat$ is further seen in the log-log plots for density in subfigures \ref{subfig:blowup_density_log_2D} and \ref{subfig:blowup_density_log_3D} and velocity in subfigures \ref{subfig:blowup_velocity_log_2D} and \ref{subfig:blowup_velocity_log_3D}. In the log-log plots, the data adheres to $\rhohat$ and $\vhat$ up until some $y$, at which it peels off from the profile. The departure from the profile indicates the data entering the region sufficiently far away from $r = r_l(t)$. The accurate prediction $\rhohat$ and $\vhat$ provide compared to numerical simulations close to the predicted implosion time gives numerical evidence that $\rhohat$ and $\vhat$ are the asymptotic limit of the solution up to implosion time. 

\begin{figure*}[pos=h]
     \centering
     \begin{subfigure}[h]{0.495\linewidth}
         \centering
         \includegraphics[width=\textwidth]{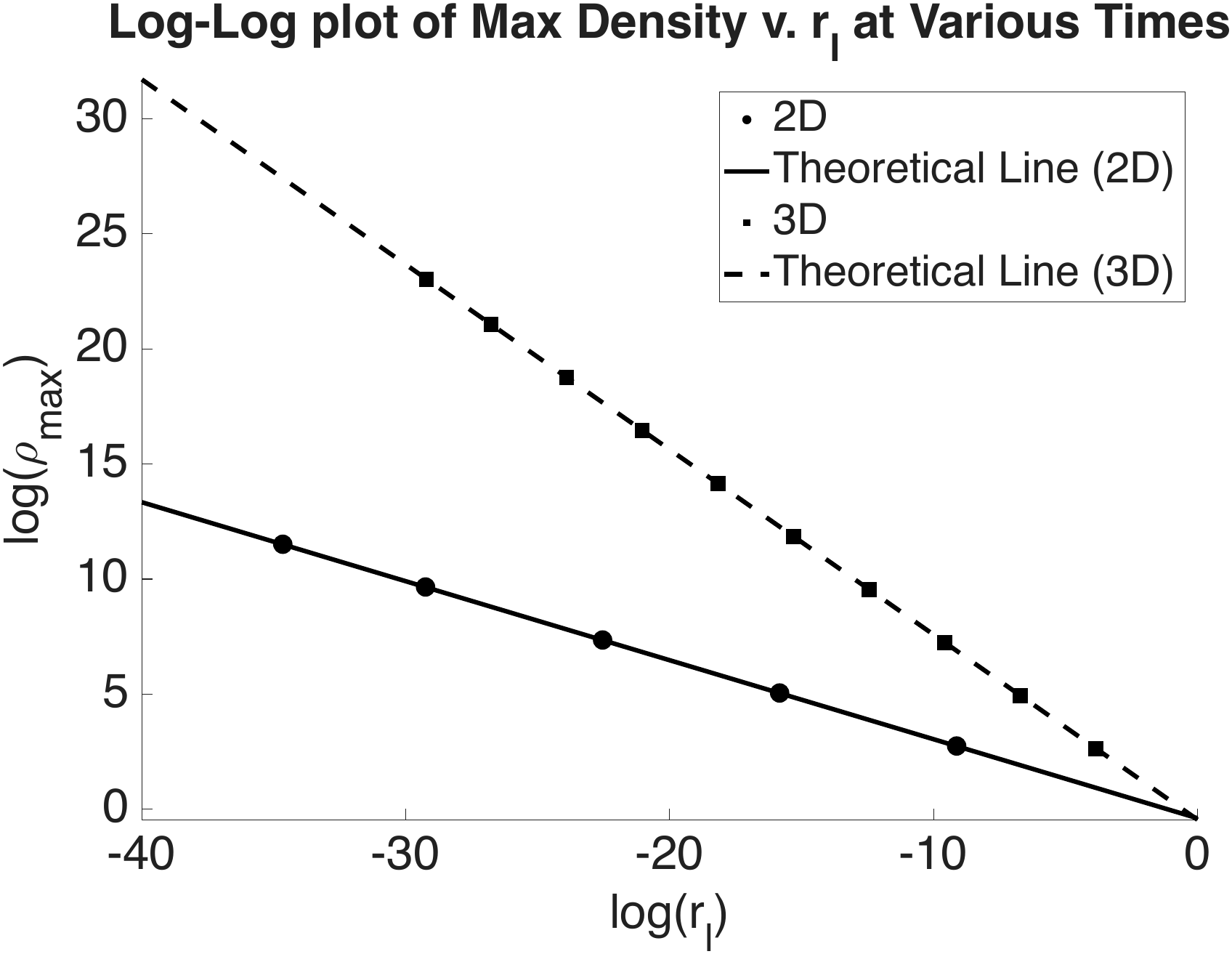}
         \caption{Density}
         \label{subfig:scaling_density}
     \end{subfigure}
     \hfill
     \begin{subfigure}[h]{0.495\linewidth}
         \centering
         \includegraphics[width=\textwidth]{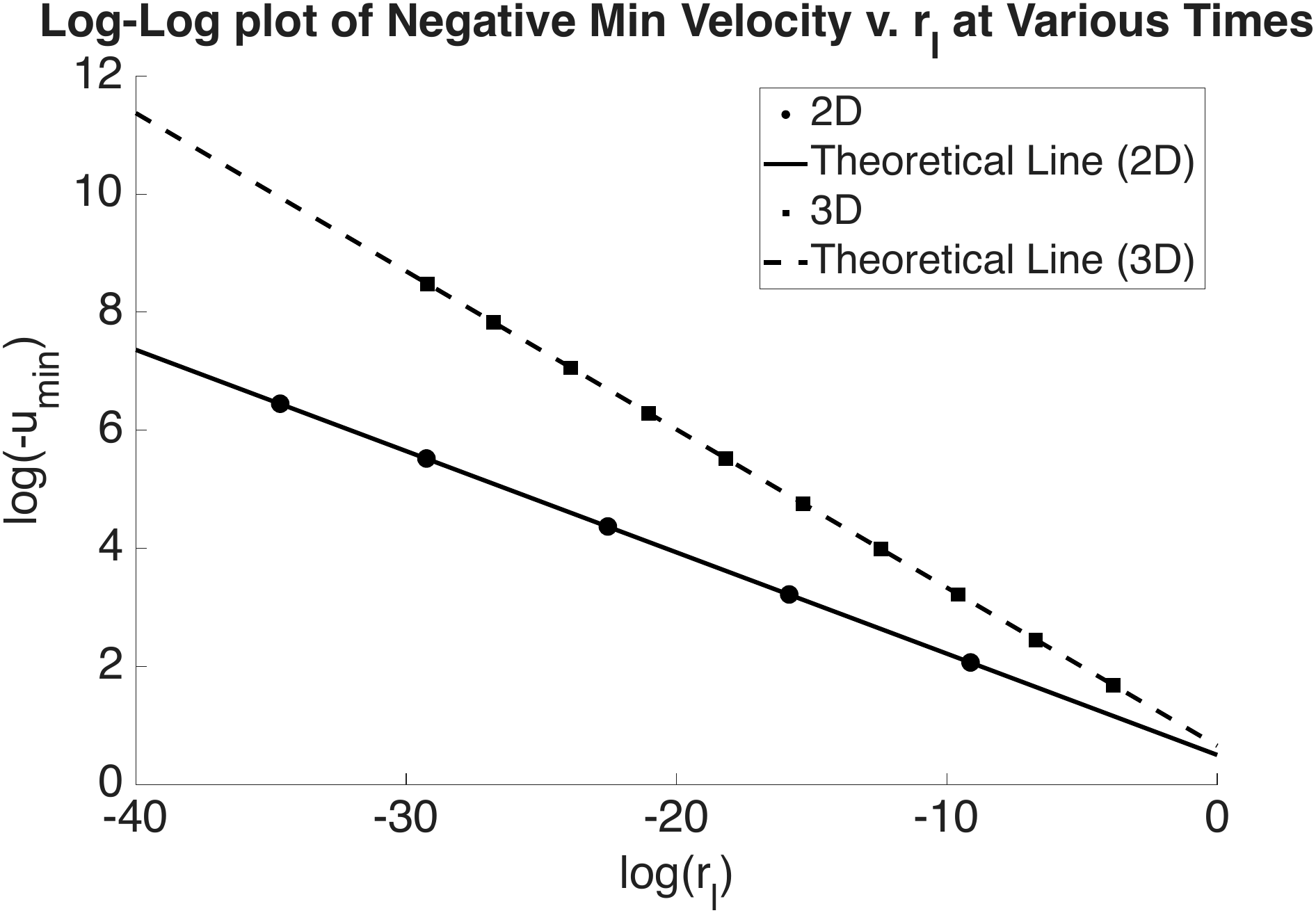}
         \caption{Velocity}
         \label{subfig:scaling_velocity}
     \end{subfigure}
        \caption{Log-log plots of maximum density and negative minimum velocity with respect to $r_l(t)$ at various times. The expected power law relationship \eqref{eq:power_law_tau_r_l} between $r_l$ and density and velocity is plotted and represented as a line. Note no fitting parameters are given in \eqref{eq:power_law_tau_r_l}. The data points fit exactly onto the line, verifying the expected power law relationship.}
        \label{fig:scaling}
\end{figure*}

To verify whether the numerical simulation accurately captures the time-scaling behavior, we check whether the numerical simulation correctly evolves $\tau(t)$ near implosion time. We check this by determining the evolution of the maximum and negative minimum of density and velocity respectively against $r_l(t)$. From the definitions of $\tau(t)$ \eqref{eq:tau} and $r_l(t)$ \eqref{eq:r_l} and the relationship between them, the power law relationship of $\rho(r_l(t),t)$ and $-v(r_l(t),t)$ in terms of $r_l(t)$  is
\begin{equation}\label{eq:power_law_tau_r_l}
\begin{split}
        \rho(r_l(t),t) &= \rhohat(1)\left(\frac{r_l(t)}{r^* 2^{\frac{1 - (n+1)^{-1/{2}}}{2}} }\right)^{-(n+1)(s-1)} , \\
        -v(r_l(t),t) &= \vhat(1)\left(\frac{r_l(t)}{r^* 2^{\frac{1 - (n+1)^{-1/{2}}}{2}} }\right)^{-(s-1)}.
\end{split}
\end{equation} 
We compute $r_l(t)$ from the data by defining $r_l(t) = \argmin_r v(r,t)$. Log-log plots of the extrema versus $r_l(t)$ are given in Figure \ref{fig:scaling}, with density in Subfigure \ref{subfig:scaling_density} and velocity in Subfigure \ref{subfig:scaling_velocity}. The expected power law behavior between $r_l(t)$ and $\tau(t)$ from \eqref{eq:power_law_tau_r_l} is also plotted. We see for both the density and velocity in 2D and 3D the extrema value match directly onto the theoretical power law with correct prefactor, verifying the numerical simulation accurately captures the scaling dynamics between $r_l(t)$ and $\tau(t)$.

\subsection{Blowup in $L^p_{[0,1]}$}
\begin{figure*}[pos=h]
     \centering
     \begin{subfigure}[h]{0.49\linewidth}
         \centering
         \includegraphics[width=\textwidth]{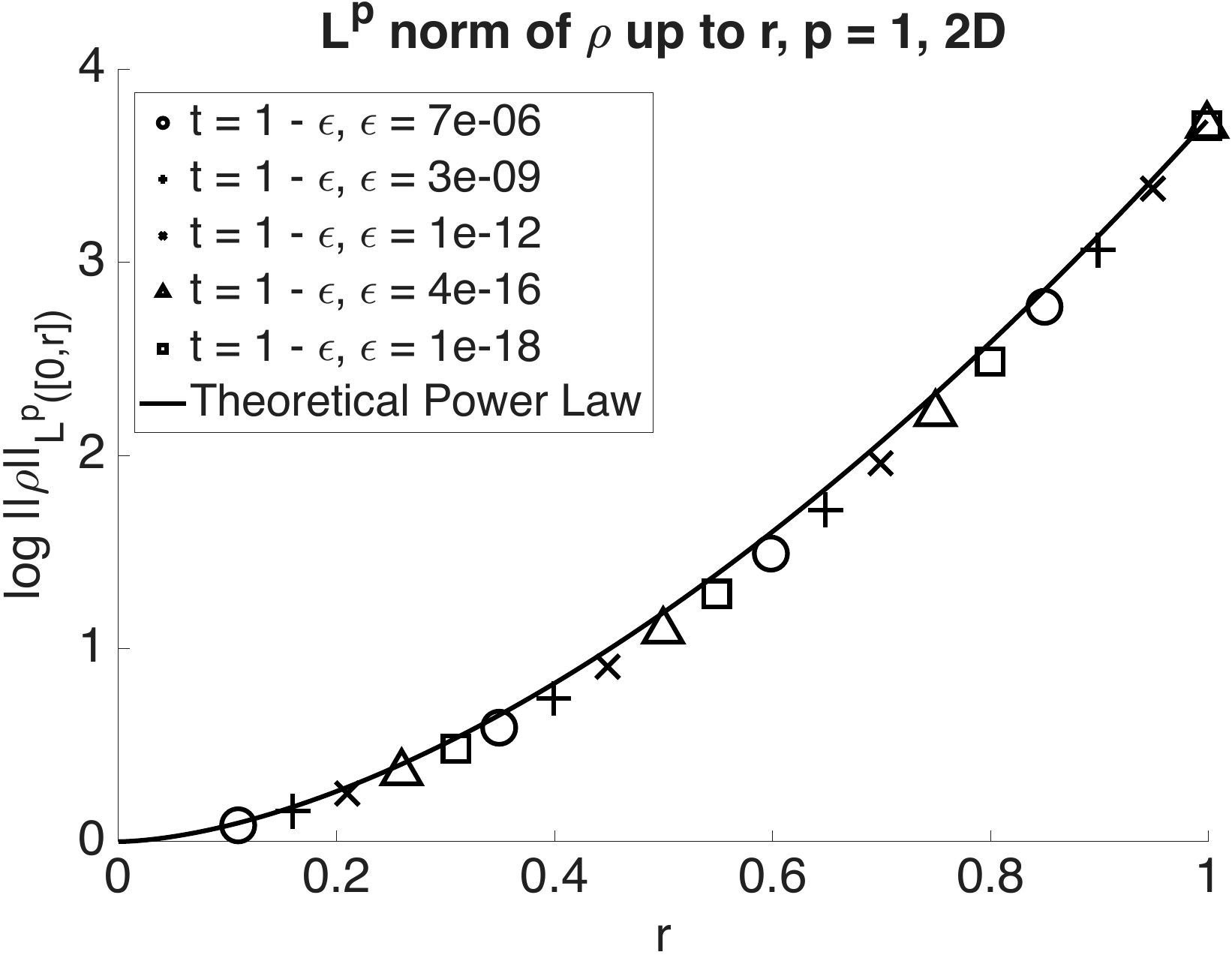}
         \caption{Mass in $2$D}
         \label{subfig:mass_2D}
     \end{subfigure}
     \hfill
     \begin{subfigure}[h]{0.49\linewidth}
         \centering
         \includegraphics[width=\textwidth]{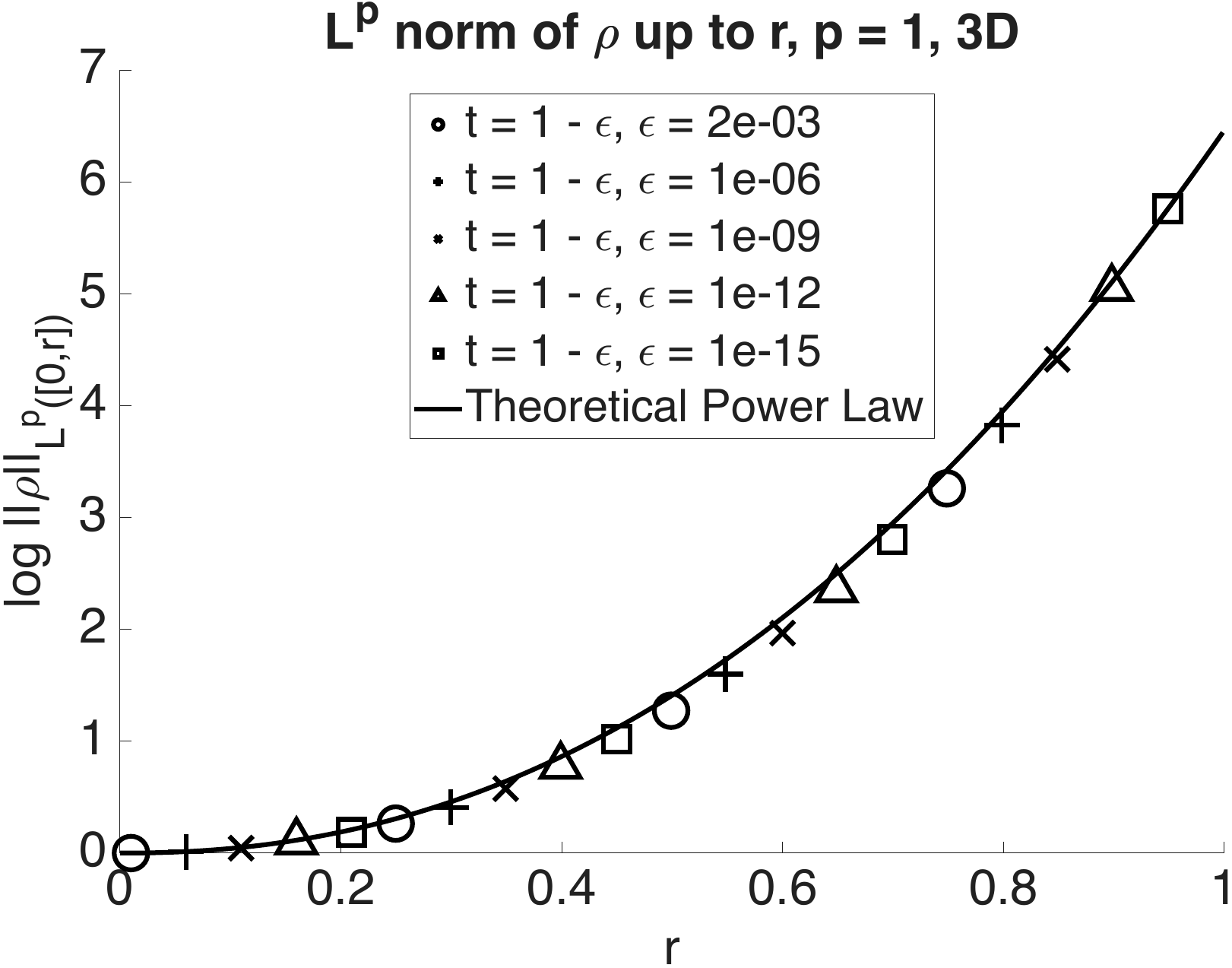}
         \caption{Mass in $3$D}
         \label{subfig:mass_3D}
    \end{subfigure}
    \vskip\baselineskip
    \begin{subfigure}[h]{0.49\linewidth}
         \centering
         \includegraphics[width=\textwidth]{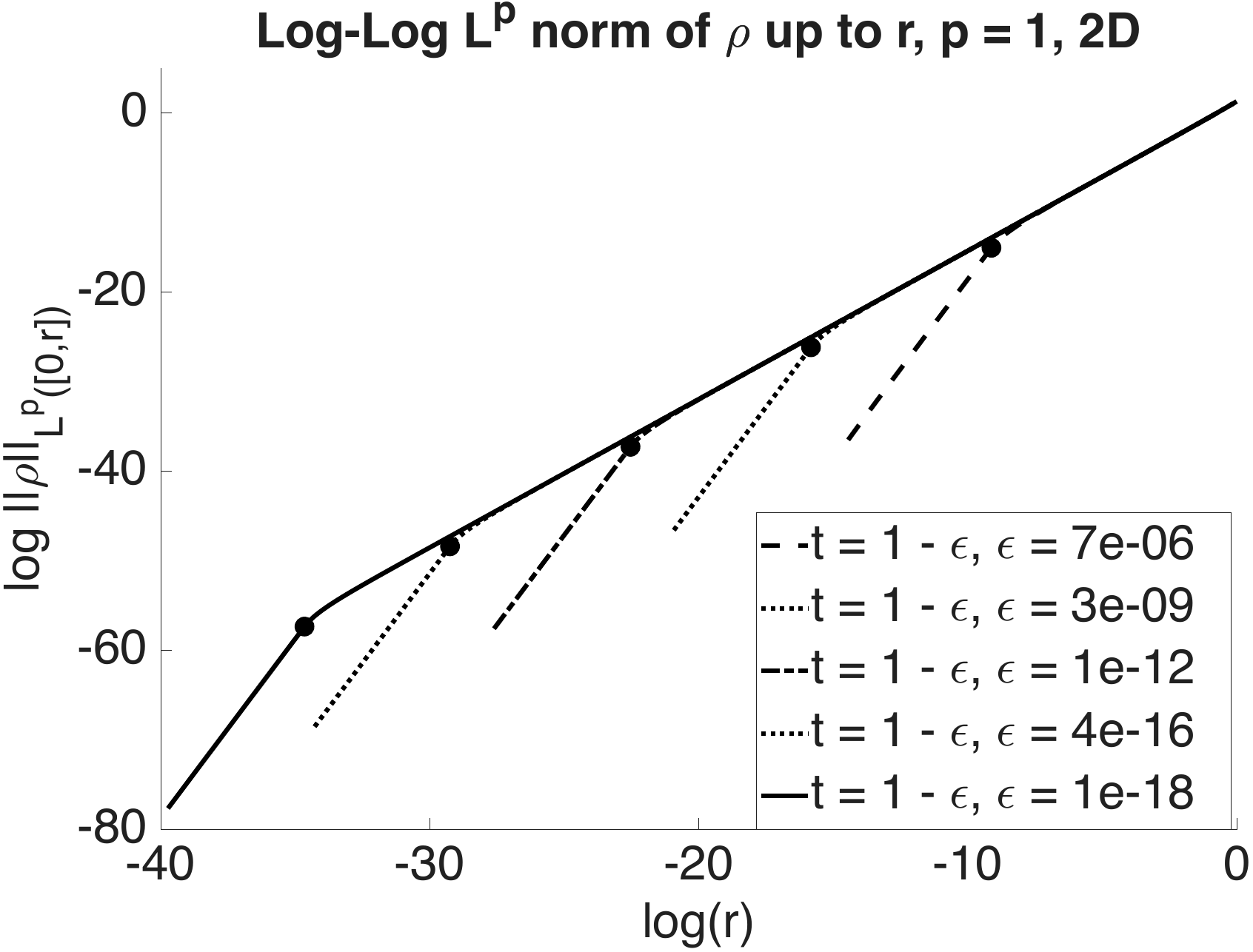}
         \caption{Log Mass in $2$D}
         \label{subfig:mass_2D_log}
     \end{subfigure}
     \hfill
     \begin{subfigure}[h]{0.49\linewidth}
         \centering
         \includegraphics[width=\textwidth]{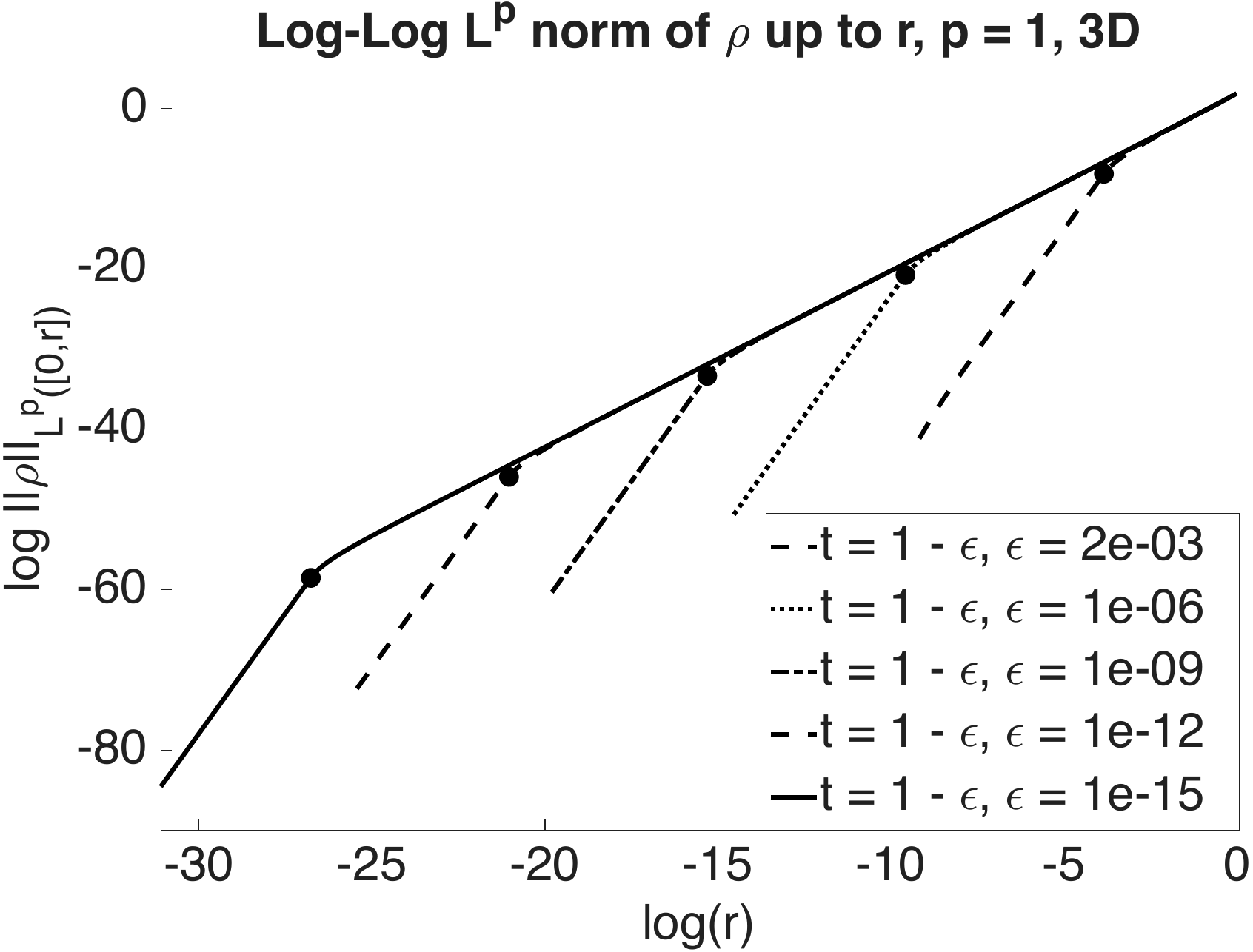}
         \caption{Log Mass in $3$D}
         \label{subfig:mass_3D_log}
     \end{subfigure}
        \caption{The $L^1$ norm of $\rho$ on $[0,r]$ for $2$D and $3$D. Subfigures \ref{subfig:mass_2D} and \ref{subfig:mass_3D} show the radial dependence on mass while subfigures \ref{subfig:mass_2D_log} and \ref{subfig:mass_3D_log} are in log-log scale. The selected times are the same as in Figures \ref{fig:2D_blowup} and \ref{fig:3D_blowup}. A power law curve is plotted in Subfigures \ref{subfig:mass_2D} and \ref{subfig:mass_3D} whose exponent is given by \eqref{eq:Lp_power_law}. The constant is fitted from the numerical data. Black dots are plotted in Subfigures \ref{subfig:mass_2D_log} and \ref{subfig:mass_3D_log} representing $r_l(t)$ at various times.}
        \label{fig:mass}
\end{figure*}

\begin{table}[pos=h]
    \centering
    \begin{tabular}{|c|c|c|c|}
    \hline
     Dimension & $p$ & Numerical & Theoretical \\
    \hline
        $2$ & $1$ & $1.6651$ & $1.6569$ \\ 
    \hline
        $2$ & $\gamma$ & $0.6654$ & $0.6569$ \\
    \hline
        $3$ & $1$ & $2.2070$ & $2.1962$ \\ 
    \hline
        $3$ & $\gamma$ & $1.0087$ & $0.9962$ \\
    \hline
    \end{tabular}
    \caption{Comparison between the slope of numerically computed $\log{\Vert \rho(r,t)\Vert_{L^p\left([0,r]\right)}}$ against $\log{r}$, shown in Subfigures \ref{subfig:mass_2D_log}, \ref{subfig:mass_3D_log}, \ref{subfig:pgamma_2D_log}, and \ref{subfig:pgamma_3D_log} for various $p < p_{\crit}$, to the theoretical power law scaling given in \eqref{eq:Lp_power_law}. The slope is numerically computed for $2r_l(t) \leq r \leq 1$.}
    \label{tab:Lp_slope}
\end{table}

\begin{figure*}[pos=h]
     \centering
     \begin{subfigure}[h]{0.495\linewidth}
         \centering
         \includegraphics[width=\textwidth]{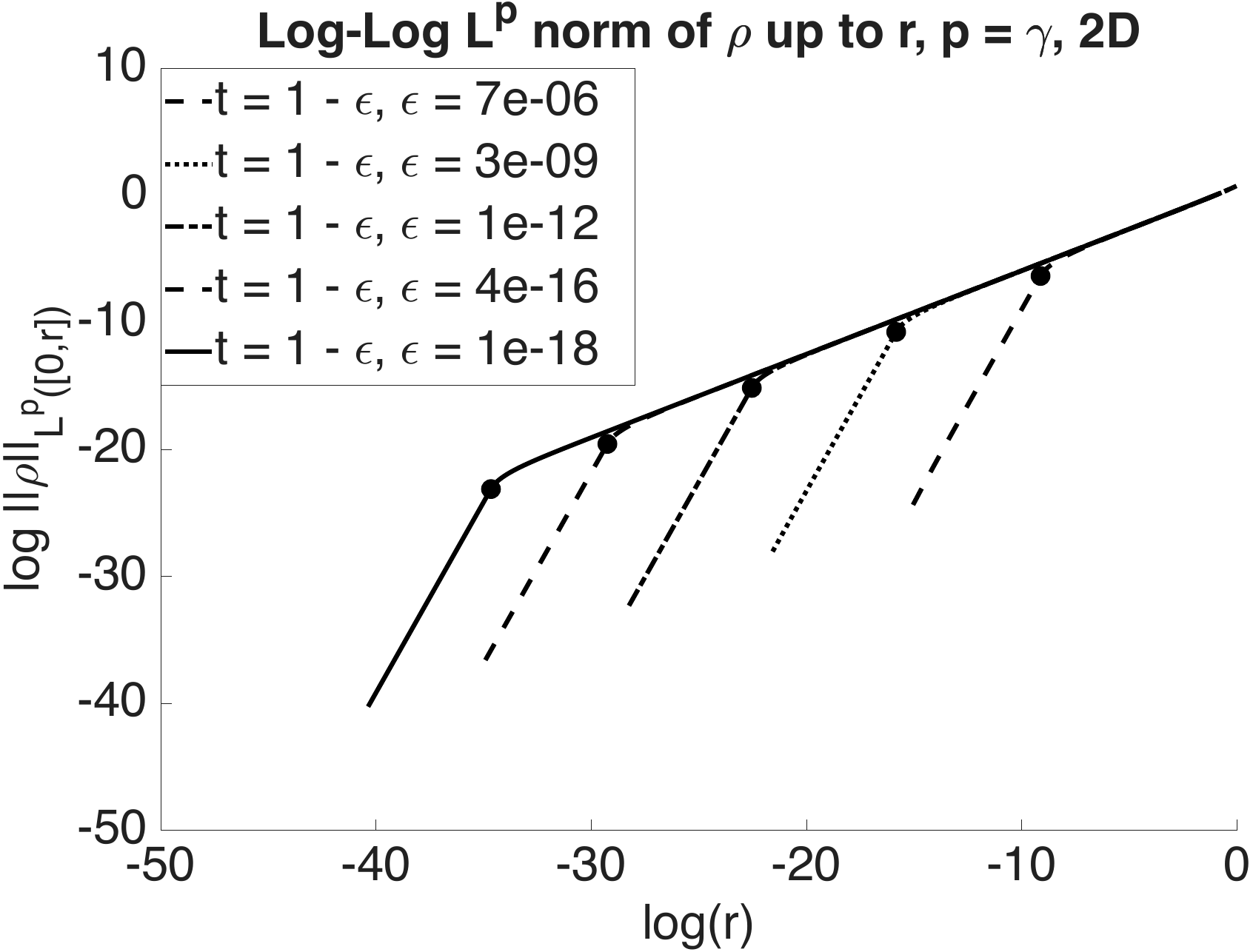}
         \caption{2D}
         \label{subfig:pgamma_2D_log}
     \end{subfigure}
     \hfill
     \begin{subfigure}[h]{0.495\linewidth}
         \centering
         \includegraphics[width=\textwidth]{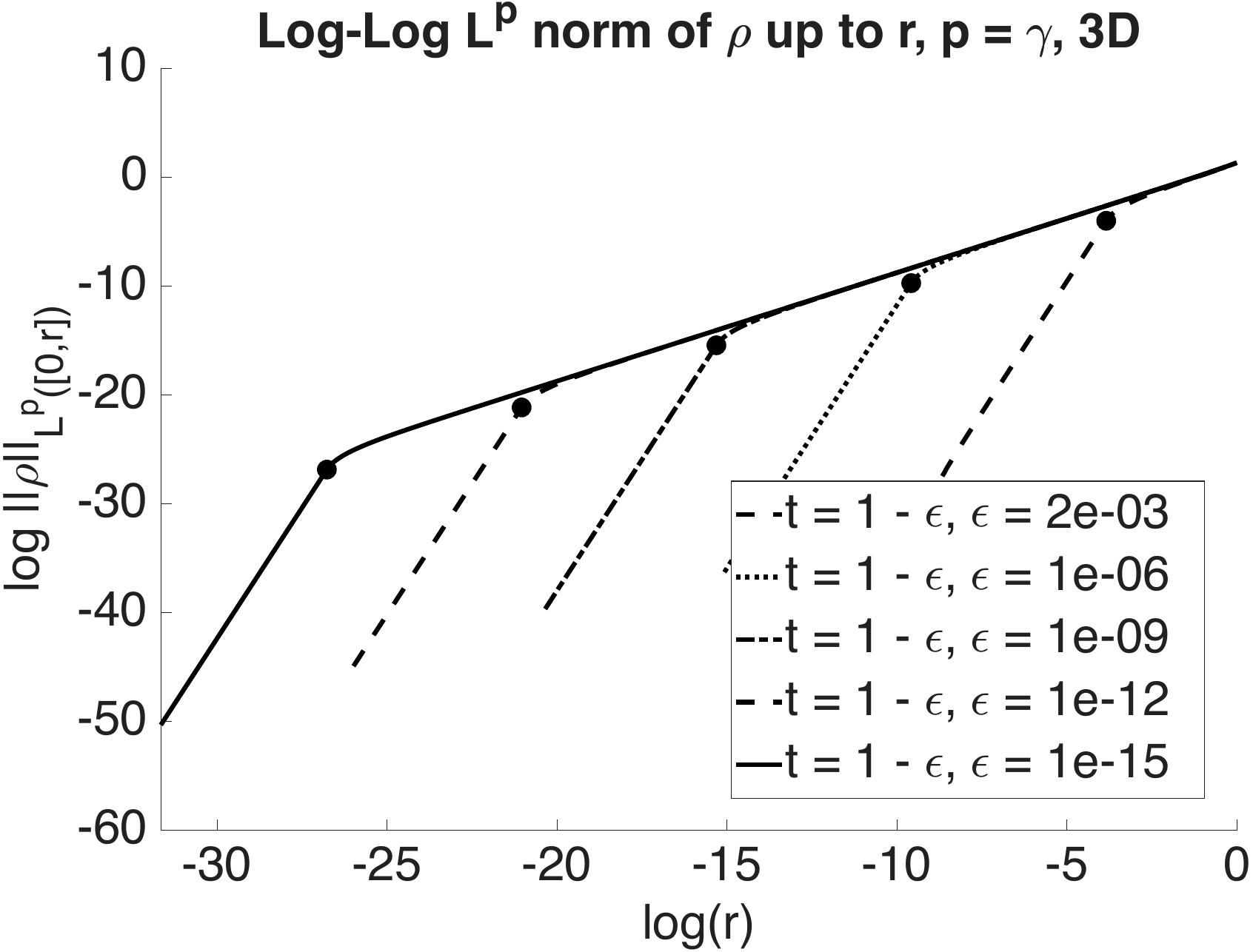}
         \caption{3D}
         \label{subfig:pgamma_3D_log}
     \end{subfigure}
        \caption{The log-log plot of the $L^p$ norm of $\rho$ on $[0,r]$ in $2$D (Subfigure \ref{subfig:pgamma_2D_log}) and $3$D (Subfigure \ref{subfig:pgamma_3D_log}) for $p = \gamma$. Black dots represent $r_l(t)$ at various times.}
        \label{fig:pgamma}
\end{figure*}

\par We conclude by computing the $L^p$ norm of density over $[0,1]$ for various choices of $p$ to investigate the theoretical work done in Subsection \ref{subsec:L^p}. The $L^1_{[0,1]}$ norm, or the mass, is shown in Figure \ref{fig:mass}. From Subfigures \ref{subfig:mass_2D} and \ref{subfig:mass_3D}, we observe no evidence of mass concentration as the mass seems to depend on $r$ as a power law, shown by the adherence of data points to a power law curve with exponent \eqref{eq:Lp_power_law}. This is further supported by the log-log plots of Subfigures \ref{subfig:mass_2D_log} and \ref{subfig:mass_3D_log} where the data collapses onto the same line past $r_l(t)$. Moreover, the slope of this line is approximately the exponent predicted in \eqref{eq:Lp_power_law}, displayed in Table \ref{tab:Lp_slope}. The black dots in Subfigures \ref{subfig:mass_2D_log} and \ref{subfig:mass_3D_log} indicate the location of $r_l(t)$ for each instance of time. For each data curve, we find the logarithm of the mass collapses to the same line when $r > r_l(t)$. The derivation of \eqref{eq:Lp_power_law} assumes the power law behavior of $\rhohat(y)$ is valid when $y$ is sufficiently greater than $1$ - or equivalently when $r$ is sufficiently greater than $r_l(t)$. On the intermediate region slightly greater than $r_l(t)$, $\rhohat$ is not a power law which explains why the adherence of data to a common line occurs slightly past $r_l(t)$ and not exactly at $r_l(t)$. 
\par Similar behavior is found in Figure \ref{fig:pgamma}, which shows the log-log plot of the $L^\gamma$ norm over $[0,r]$ for solutions to the cutoff implosion problem in $2$D and $3$D. We see the same adherence to a common curve as in Subfigures \ref{subfig:mass_2D_log} and \ref{subfig:mass_3D_log}, showing no blow up in $L^\gamma$, as implied by the global existence result in \cite{Chen_2022}. The slope of this common line is tabulated in Table \ref{tab:Lp_slope} and matches the predicted value in \eqref{eq:Lp_power_law}. 

\begin{figure}[pos=h]
    \centering
    \includegraphics[width=0.95\linewidth]{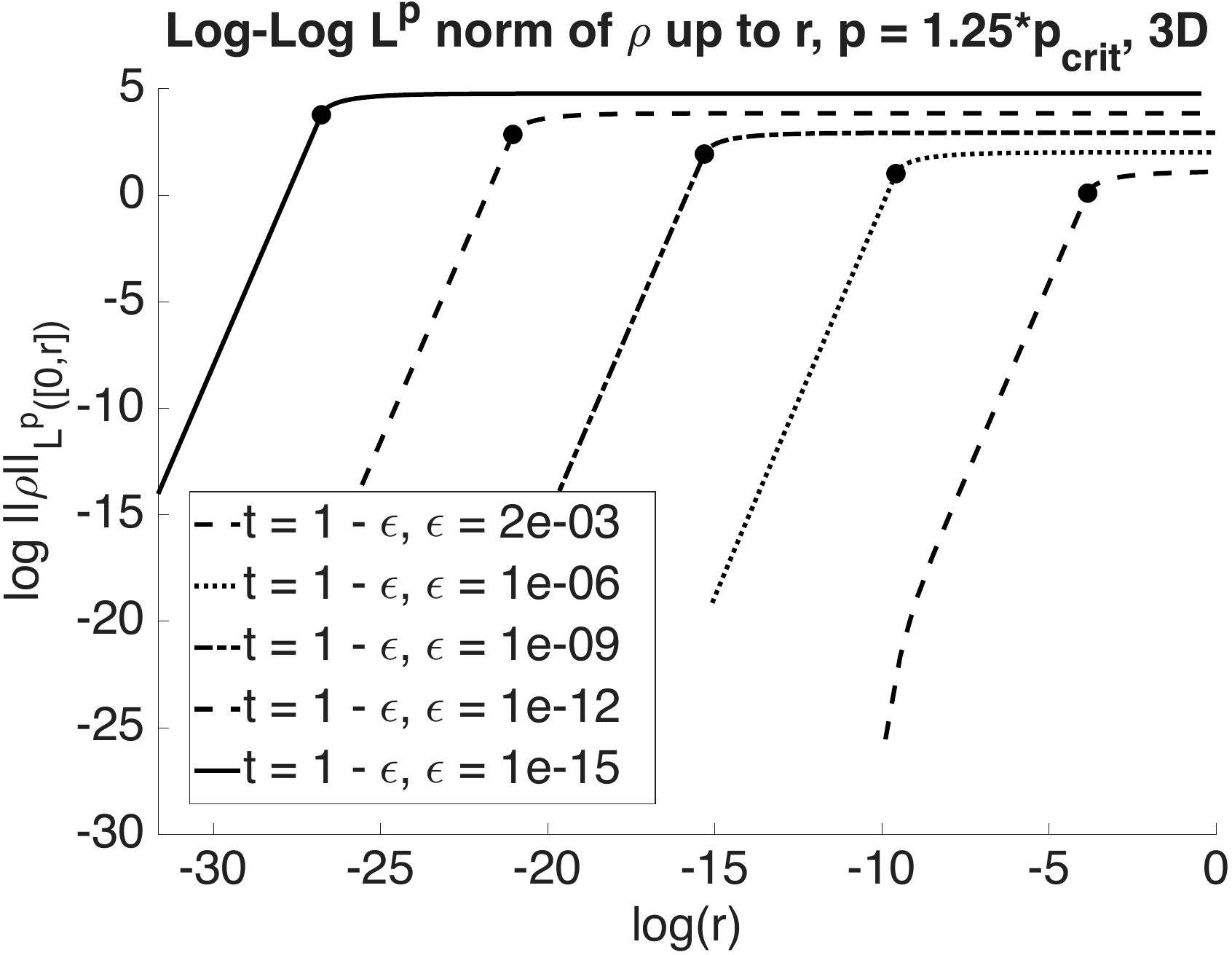}
    \caption{Log-log plot of $L^p$ norm of $\rho$ on $[0,r]$ in $3$D for $p = 1.25p_{\crit}$, given by \eqref{eq:p_crit}. Black dots represent $r_l(t)$ at various times.}
    \label{fig:p_crit_over}
\end{figure}

\par In contrast, an example of $L^p$ blowup is shown in Figure \ref{fig:p_crit_over} where $p$ is set to $1.25p_{\crit}$ and $p_{\crit}$ is given by \eqref{eq:p_crit} for $n = 2$. The $L^p$ norm behaves drastically different when $p > p_{\crit}$. The $L^p$ norm dramatically increases with time (note Figure \ref{fig:p_crit_over} is in log-log scale) and does not collapse to any common curve in log-log scale. Figure \ref{fig:p_crit} shows an example of how the $L^p$ norm behaves when $p$ is exactly $p_{\crit}$. Instead of collapsing onto a power law curve, the $L^p$ norm increases for very small $r$ and then tapers off for larger $r$. As time progresses, the $L^p$ norm also increases at a given $r$, which could imply concentration formation. The growth of the $L^p$ norm in time is less dramatic than the $p = 1.25p_{\crit}$ example however.

\begin{figure*}[pos=h]
     \centering
     \begin{subfigure}[h]{0.495\linewidth}
         \centering
         \includegraphics[width=\textwidth]{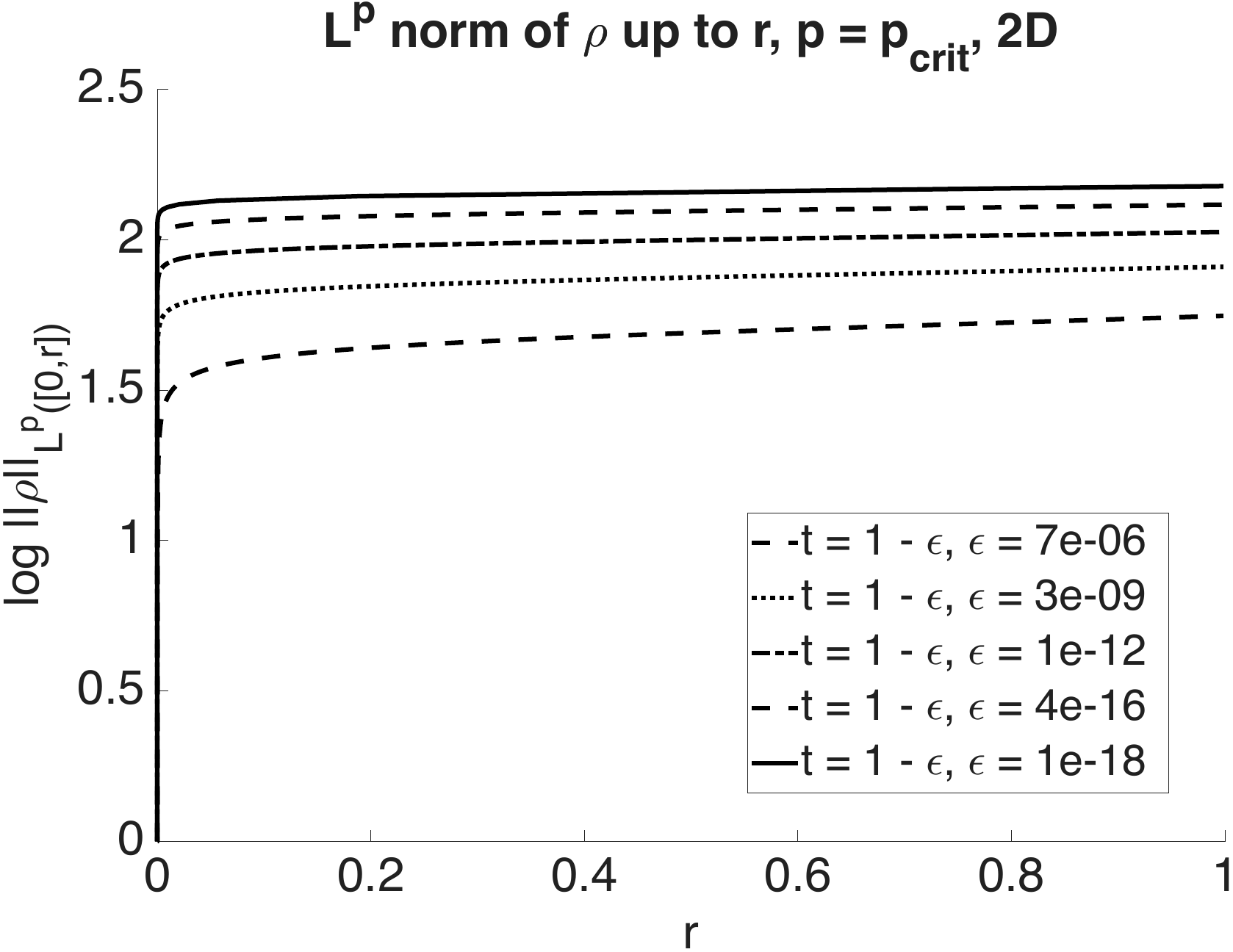}
         \caption{$2$D}
         \label{subfig:p_crit_2D}
     \end{subfigure}
     \hfill
     \begin{subfigure}[h]{0.495\linewidth}
         \centering
         \includegraphics[width=\textwidth]{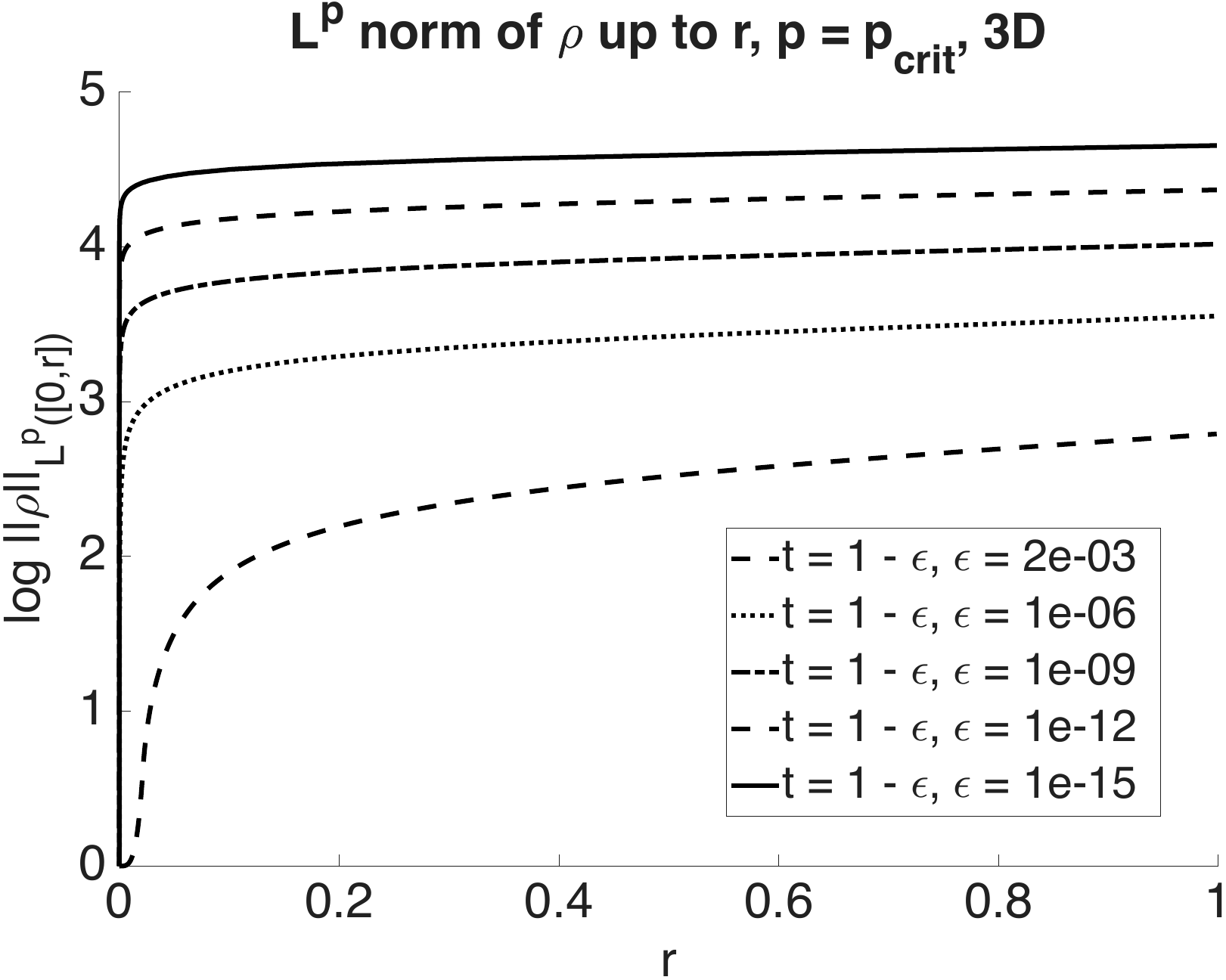}
         \caption{$3$D}
         \label{subfig:p_crit_3D}
     \end{subfigure}
     \vskip\baselineskip
     \begin{subfigure}[h]{0.495\linewidth}
         \centering
         \includegraphics[width=\textwidth]{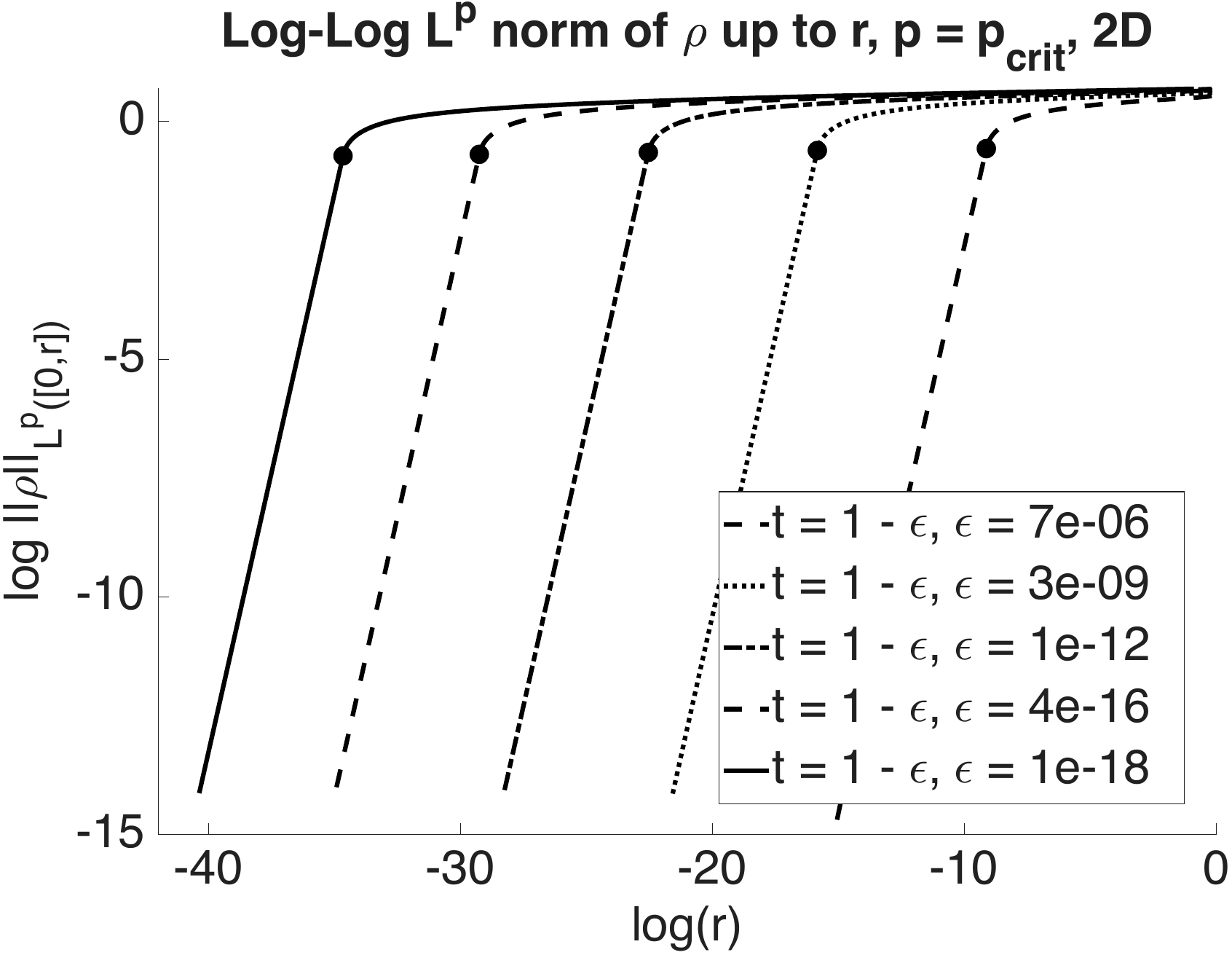}
         \caption{Log-Log $2$D}
         \label{subfig:p_crit_2D_log}
     \end{subfigure}
     \hfill
     \begin{subfigure}[h]{0.495\linewidth}
         \centering
         \includegraphics[width=\textwidth]{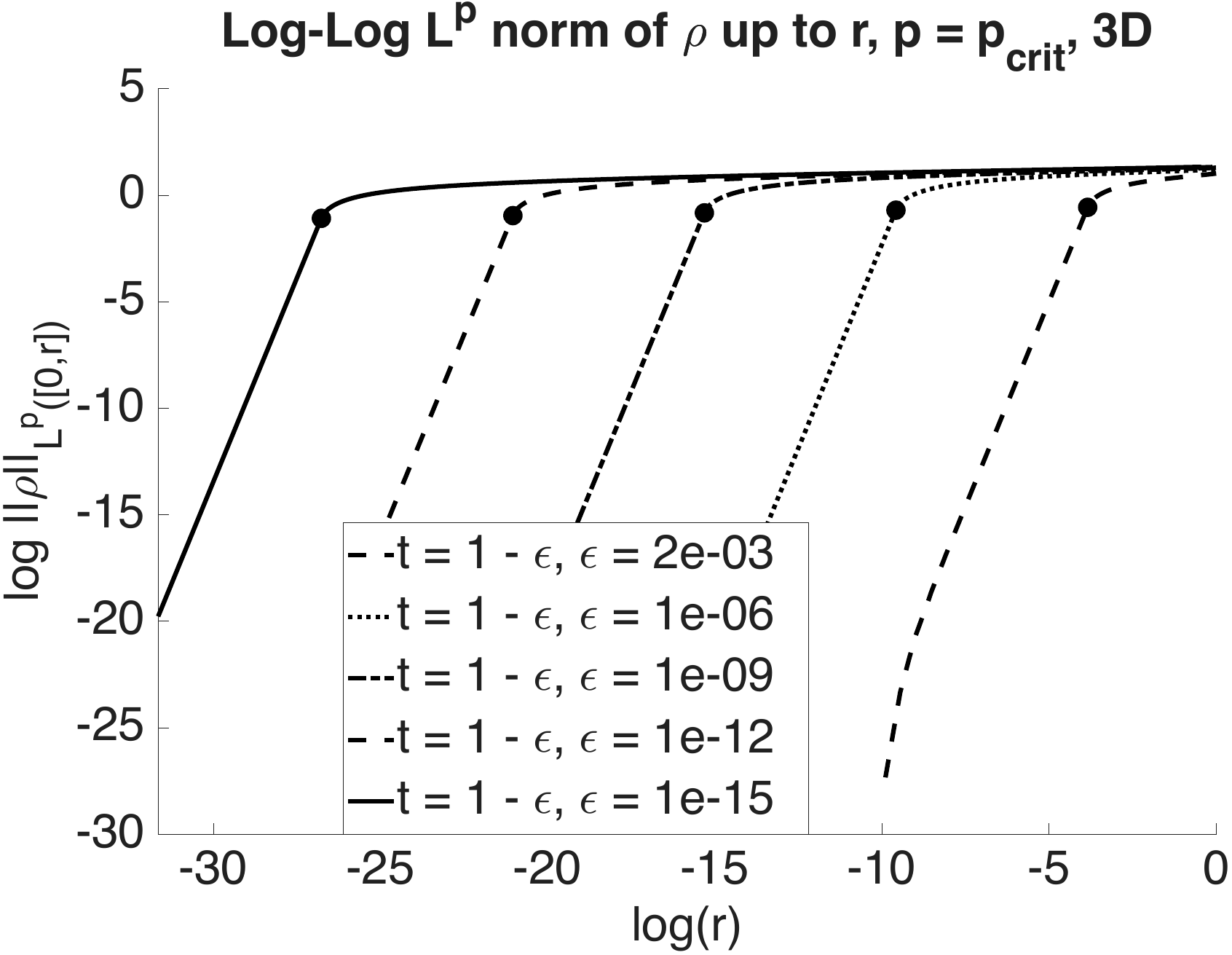}
         \caption{Log-Log $3$D}
         \label{subfig:p_crit_3D_log}
     \end{subfigure}     
        \caption{The $L^p$ norm of $\rho$ on $[0,r]$ in $2$D (Subfigure \ref{subfig:p_crit_2D}) and $3$D (Subfigure \ref{subfig:p_crit_3D}) for $p = p_{\crit}$, given by \eqref{eq:p_crit}. Subfigures \ref{subfig:p_crit_2D_log} and \ref{subfig:p_crit_3D_log} are in log-log scale. Black dots in subfigures \ref{subfig:p_crit_2D_log} and \ref{subfig:p_crit_3D_log} represent $r_l(t)$ at various times.}
        \label{fig:p_crit}
\end{figure*}

\par In the derivation of \eqref{eq:p_crit}, we found $L^p$ blowup occurs only because of the Kidder solution and non-power law behavior of $\rhohat$ near $r_l(t)$. We see from Subfigures \ref{subfig:p_crit_2D_log} and \ref{subfig:p_crit_3D_log} most of the growth is on the interval $[0, r_l(t)]$. In contrast, the $L^p$ norm does not grow much for $r$ much greater than $r_l(t)$ - supporting the observation the farfield power law behavior of $\rhohat$ does not contribute to $L^p$ blowup.

\section{Conclusion}\label{sec:conclusion}
In this work, we study the solution to the cutoff implosion problem with initial data formed from the Kidder solution inside a specied radius, and constrant pressure outside that radius, with the radius and the constant chosen so that the initial data is continuous.  We then describe the region of influence stemming from the cutoff point and derive explicit expressions for the density and velocity at the inner radius of the region of influence. In $1$D, these terms agree with the analysis done in \cite{Luong_1D} and no implosion occurs. However, in higher dimensions, implosion occurs.

\par To better understand the implosion profile, we solve for a self-similar solution that matches the asymptotic behavior observed in the numerics. Applying the self-similar ansatz yields an ODE system in the similarity variables, and we draw connections to the self-similar dynamical system presented in \cite{Guderley1942} and \cite{implosion_1}. The similarity ODE has a degeneracy at the starting point, requiring a specific direction integration.

\par We emphasize these computed similarity solutions are \textit{exact} self-similar solutions for $y \geq 1$ which are asymptotic limits to the PDE cutoff implosion numerical solution. To complete the description of the self-similar profiles, one must understand the profiles for $0 \leq y \leq 1$, corresponding to radial coordinates near $r = 0$. Recall in \cite{shao2025blowup3dcompressiblenavierstokes} the authors construct smooth self-similar solutions for a sequence of blowup exponents $s_n \to s^*$, where $s^*$ is exactly the blowup exponent considered in this work. A complete description of the self-similar profile could provide a connection to the sequence of smooth self-similar solutions of \cite{shao2025blowup3dcompressiblenavierstokes}. Additionally, one may use the full self-similar profile as initial data for hydrodynamic simulations.

\par Numerical simulations of the PDE are performed with a Lagrangian leapfrog scheme enhanced with adaptive mesh refinement. Various predicted quantities are compared to their numerical counterparts for times far from the implosion time. We find low percent error between these numerical and theoretical quantities. Approaching the implosion time, the numerical solution near blowup time agrees with the self-similar profiles, indicating the self-similar profiles are the asymptotic limit of the implosion. 

\par The motivation behind continuously cutting off the Kidder initial data is to determine whether it is necessary to have unbounded density in the far field and still maintain implosion. In the Kidder solution, all Lagrangian particles arrive at the origin at implosion time regardless of their initial position. These far field behaviors render the Kidder solution difficult to directly apply for physical implementation. The analysis done in this work shows the far field can be replaced with a much milder far field condition of constant density. This reveals the mechanics of the Kidder implosion does not rely on an infinite stream of fluid to the origin from the whole domain and can be achieved with a finite amount of fluid on a finite region.

\par The implosion results presented in this work are in contrast to the result established in \cite{Luong_1D} where cutting off the Kidder initial data in $1$D suppresses the implosion. Analysis done in this current work is done in general dimension and reveals a bifurcation in inner radius behavior when $n$ shifts from $n = 0$ to any larger index. The bifurcation indicates a feature inherit in the geometry of higher dimensions causes the implosion, further supported by the fact the implosion result holds for any $n > 0$ - not necessarily integer. Although non-integer dimension is physically uninterpretable, the radially symmetric, compressible Euler equations formally admits non-integer $n$ in \eqref{eq:con_mass}. The sharp difference in behavior from $1$D to any higher dimension indicates the cause may be focusing at the origin in higher dimensions. We comment we see no restrictions on how large $n$ can be from our analysis. Preliminary numerical investigation shows the predictions made in this paper holds for large $n$. It may be interesting to consider the ramifications of this work in the limit as $n$ approaches infinity. 

\par Numerical simulations support the predictions made in this work. The fact the imploding solution and its landmark features can be accurately computed is a promising indicator of the numerical stability of the cutoff implosion solution. In contrast, the solution of \cite{implosion_1} is numerically unstable as shown by \cite{biasi_compute}. One of the marked differences between the cutoff implosion solution and the implosions of \cite{buckmaster_implosion,Jenssen2020,implosion_1} is the implosion point in the cutoff implosion solution is non-stationary. A heuristic explanation given in the introduction of \cite{Luong_1D} argues why stationary implosions may struggle with instability. Given the cutoff implosion solution seems to enjoy numerical stability properties, searching for imploding solutions where the implosion occurs at a traveling point may be a fruitful strategy.

\par After the finite time implosion occurs, it is natural to ask whether it is possible to continue the solution. Physically, this corresponds to an \textit{explosion} and works such as \cite{Guderley1942} and \cite{vicol_explosion} consider such explosions originating from implosions. Another direction is to consider smoother cutoffs, after it has been established the Kidder implosion does not require unbounded initial data in the far field.

\appendix
\section{Equivalence of Self Similar Systems}\label{ap:equal}
In this appendix, we show the non-autonomous system \eqref{eq:ss_system} is the same as the non-autonomous system instantiated in \cite{Guderley1942} and elaborated in \cite{implosion_1} by comparing system \eqref{eq:ss_system} to the notation and derivation presented in \cite{biasi_compute}. 

\par First, rescale system \eqref{eq:ss_system} by dividing the top equation by $A^{n+2}$ and the bottom equation by $A^2$ where $A$ is defined in \eqref{eq:A}. We also rescale $\rhohat$ and $\vhat$ by setting $\rhohat = {\rhohat}/{A^{n+1}}$ and $\vhat = -{\vhat}/{A}$ (the negative rescaling is to match the definition of the self-similar ansatz for velocity in \cite{implosion_1}). In doing so, the initial conditions for $\rhohat$ and $\vhat$ in \eqref{eq:ss_ic} no longer depend on $r^*$. After the rescaling, system \eqref{eq:ss_system} becomes
\begin{equation}\label{eq:ss_system_2}
\begin{split}
\frac{\lambda}{A}\rhohat + y \rhohat' + (\rhohat \vhat)' + \frac{n\rhohat\vhat}{y} = 0, \\
\frac{\beta \lambda}{A} \vhat + y\vhat' + \vhat \vhat' + \gamma\rhohat^{\gamma - 2}\rhohat' = 0.
\end{split}
\end{equation}
Note we have defined pressure with constant $C = 1$ while $C = (\gamma- 1)/{\gamma}$ in \cite{implosion_1}. 

\par Let $s^*$ denote the inverse of the $(1-t)$ exponent in $r_l(t)$, so
\begin{equation}\label{eq:s}
    s^* = \frac{2}{1 + (n+1)^{-1/{2}}}.
\end{equation}
The exponent $s^*$ is denoted as $r$ in \cite{implosion_1} and \cite{biasi_compute}. Note the value of \eqref{eq:s} is the upper bound of the blowup exponents in \cite{implosion_1} and \cite{shao2025blowup3dcompressiblenavierstokes}. Recall $\beta = {1}/(n+1)$ which is equal to ${1}/{\ell} = (\gamma - 1)/{2}$. The final step to show equivalence between the systems is to show the coefficients ${\lambda}/{A}$ and ${\beta \lambda }/{A}$ in front of $\rhohat$ in the top equation and $\vhat$ in the bottom equation of \eqref{eq:ss_system_2} are $\ell(s^* - 1)$ and  $s^* - 1$ respectively. Recalling the value of ${\lambda}/{A}$ from \eqref{eq:hat_slopes}, we compute $s^* - 1$ and see
\begin{equation*}
    s^* - 1 = -\frac{1 - \sqrt{n+1}}{1 + \sqrt{n+1}}
\end{equation*}
confirming ${\lambda}/{A} = \ell(s^*-1)$ and ${\beta \lambda}/{A} = s^* - 1$. Thus, we have established the equivalence of the non-autonomous systems and presented here and in the literature. The autonomous systems are subsequently equivalent by applying the same Emden transform to both.

\printcredits

\section*{Acknowledgments}
The authors would like to thank Alexis Vasseur for his insights in the application of weak-strong uniqueness results and Gui-Qiang Chen for informing the authors of global existence results for the compressible Euler equations. The first author would like to thank Evan Davis and Yifan Gu for preliminary conversations on implosions and potential flow and Richard Bowers for conversations on the Lagrangian leapfrog numerical method. The authors would like to declare that a version of Claude 4.5 for LANL was used for editing and to check the reproducibility of the numerical method.

\bibliographystyle{cas-model2-names}

\bibliography{cas-refs}



\end{document}